\documentclass[11pt,reqno]{amsart}
\usepackage{amsmath,amsfonts,amscd,amssymb,graphicx,mathrsfs}

\usepackage[usenames]{color}
\usepackage{setspace}
\usepackage{array,multirow,booktabs,longtable}
\RequirePackage{tikz-cd}
\usepackage{comment}
\usepackage{tikz}
\usepackage{array} 
\usepackage{url}
\usepackage{dirtytalk}
\usepackage{cleveref}

\DeclareMathOperator{\End}{End} 
\DeclareMathOperator{\Ext}{Ext}

\DeclareMathOperator{\ad}{ad}

\DeclareMathOperator{\Hom}{Hom}
\DeclareMathOperator{\Sym}{Sym}

\DeclareMathOperator{\Supp}{Supp}

\DeclareMathOperator{\Soc}{Soc}

\DeclareMathOperator{\Spec}{Spec} 
\definecolor{lblue}{rgb}{0.3,0.0,4.4}
\definecolor{lred}{rgb}{4.3,0.0,0.4}

\newcommand{\Qcoh}[1]{\mathcal{M}(#1)}
\newcommand{\Coh}[1]{\mathcal{M}_{\mathrm{Coh}}(#1)}
\newcommand{\Hol}[1]{\mathcal{M}_{\mathrm{Hol}}(#1)}
\newcommand{\Lmod}[1]{#1\text{-}{\mathrm{mod}}}

\newcommand{\Pa}{W'}

\newcommand{\N}{\mathbb{N}} 
 
\newcommand{\Ind}{\mathrm{Ind}}

\newcommand{\bc}{\mathbf{c}}

\renewcommand{\o}{\otimes}

\newcommand{\iso}{\stackrel{\sim}{\rightarrow}} 
 
\renewcommand{\H}{\mathsf{H}}
\newcommand{\Tr}{\mathrm{Tr}}

\newcommand{\Z}{\mathbb{Z}} 
 
\newcommand{\Ker}{\mbox{Ker }}
\newcommand{\bpi}{\boldsymbol{\pi}}
\newcommand{\h}{\mathfrak{h}}

\newcommand{\op}{\mathrm{op}}
\newcommand{\DR}{\mathrm{DR}}

\newcommand{\C}{\mathbb{C}}

\newcommand{\Ham}{\mathbb{H}}
\newcommand{\dd}{\mathscr{D}}
\newcommand{\gr}{\mathrm{gr}}
\newcommand{\mc}{\mathcal}
\newcommand{\rk}{\mathrm{rk}}
\newcommand{\ms}{\mathscr}

\newcommand{\idot}{\bullet}
\newcommand{\ds}{\dots}
\newcommand{\Irr}{\mathrm{Irr}\, }
\newcommand{\reg}{\mathrm{reg}}
\renewcommand{\SS}{\mathrm{SS}}

\newcommand{\Stab}{\mathrm{Stab}}
\newcommand{\uN}{\underline{N}}
\newcommand{\Cs}{\mathbb{C}^{\times}}

\newcommand{\mf}{\mathfrak}

\newcommand{\eu}{\mathsf{eu}}

\newcommand{\git}{\ensuremath{/\!\!/\!}}
\newcommand{\D}{\mathbb{D}}

\newcommand{\sH}{\mc{H}}
\newcommand{\F}{\mc{F}}

\newcommand{\K}{\mathbb{K}}
\newcommand{\kk}{\ensuremath{\Bbbk}}

\newcommand{\mr}{\mathrm}

\newcommand{\St}{\mathscr{X}}
\newcommand{\an}{\mathrm{an}}

\newtheorem{thm}{Theorem}[section]
\newtheorem{lem}[thm]{Lemma}
\newtheorem{defn}[thm]{Definition}
\newtheorem{prop}[thm]{Proposition}
\newtheorem{cor}[thm]{Corollary}
\newtheorem{assume}[thm]{Assumption}
\newtheorem{remark}[thm]{Remark}
\newtheorem{conjecture}[thm]{Conjecture}
\newtheorem{example}[thm]{Example}
\newtheorem{question}[thm]{Question}
\newtheorem{problem}[thm]{Problem}
\newtheorem{condition}[thm]{Condition}

\definecolor{lightblue}{rgb}{0.8,0.8,0.9}
\definecolor{lightred}{rgb}{0.9,0.8,0.8}
\numberwithin{equation}{section}

\begin{document}

	\title[Pull-back and push-forward for holonomic modules]{Pull-back and push-forward functors for holonomic modules over Cherednik algebras} 
	
	\author{Gwyn Bellamy}
	
	\address{School of Mathematics and Statistics, University Place, University of Glasgow, Glasgow,  G12 8QQ, UK.}
	\email{gwyn.bellamy@glasgow.ac.uk}
	
\author{Pavel Etingof}

	\address{Department of Mathematics, Massachusetts Institute of Technology, 77 Massachusetts Avenue, Cambridge, MA 02139, USA.}
\email{etingof@math.mit.edu}

\author{Daniel Thompson}
	
	\address{Department of Mathematics, 120 E Cameron Avenue, CB \#3250, 329 Phillips Hall, Chapel Hill, NC 27599, USA.}

	\begin{abstract}
		In this article we continue the study of holonomic modules over sheaves of Cherednik algebras, initiated by the third author in \cite{ThompsonHolI}. Under mild assumptions on the parameters, we first develop a theory of $b$-functions to prove that push-forward along open embeddings preserves holonomicity. This implies that pull-back along closed embeddings also preserves holonomicity. We use these facts to show that both push-forward and pull-back under a large class of morphisms, which we call ``melys'', preserve holonomicity. Since duality preserves holonomicity, we deduce that extraordinary push-forward and extraordinary pull-back also exist for holonomic modules. 
		
		As a consequence, we give a general classification of irreducible holonomic modules similar to the classification of irreducible holonomic $\dd$-modules as minimal extensions of integrable connections on locally closed subsets. Finally, we prove that Ext-groups between holonomic modules are finite-dimensional and explore applications of our work to the classification of aspherical parameters and existence of finite-dimensional modules for sheaves of Cherednik algebras.   
	\end{abstract}
	
	\maketitle
	
{\small
	\tableofcontents
}

\section{Introduction}\label{sec:intro} 

Associated to a finite group $W$ acting on a smooth variety $X$, defined over an algebraically closed field $\kk$ of characteristic zero, is a family of sheaves $\sH_{\omega,\bc}(X,W)$ of algebras, called Cherednik algebras. Here, $\omega$ is a cohomology class of differential forms defining a sheaf of twisted differential operators on $X$ and $\bc$ is a function on the set of $W$-orbits of reflections on $X$; see \Cref{sec:melysintro} below. First introduced in \cite{EG} in the case where $X = \h$ is a linear representation, they were defined and studied in full generality by the second author in \cite{ChereSheaf}. In \cite{BernsteinLosev}, Losev introduced the notion of holonomic modules for Cherednik algebras, generalizing the usual definition of holonomic $\dd$-modules. Examples of holonomic modules include modules in category $\mc{O}$ when $X = \h$, or modules coherent over $\mc{O}_X$ for general $X$. Holonomic modules were studied by the third author in \cite{ThompsonHolI}, where results such as Kashiwara's Theorem for modules over Cherednik algebras were established.

\subsection{Melys morphisms}\label{sec:melysintro}

As for $\dd$-modules, one would like to use morphisms between varieties in order to construct and study holonomic modules. Unfortunately, it is clear that one cannot expect pull-back or push-forward of modules for Cherednik algebras to exist for arbitrary $W$-equivariant morphisms; see \Cref{rem:nopullbackmelys}. In order to give a sufficient condition for existence, we first recall how the algebras $\sH_{\omega,\bc}(X,W)$ are parametrized. We begin by choosing a twist $\omega \in \mathbb{H}^2(X,\Omega_X^{1,2})^W$, just as in the usual construction of the ring of twisted differential operators. Here $\Omega_X^{1,2}$ is a truncation of the de Rham complex, described in \Cref{sec:truncatedeRham}. In the introduction, \textit{we assume $\omega$ is Zariski locally trivial.} A pair $(w,Z)$ is said to be a \textit{reflection} if $w \in W$ and $Z$ is a connected component of the fixed point set $X^w$ having codimension one in $X$. The group $W$ acts on the set $\mc{S}(X)$ of all such pairs and we choose $\bc \colon \mc{S}(X) \to \kk$, a function constant on $W$-orbits. 

\newtheorem*{defn:melys}{Definition~\ref{defn:melys}}
\begin{defn:melys}
	A ($W$-equivariant) morphism $\varphi \colon Y \to X$ is said to be $\bc$-melys\footnote{``Melys'' is the Welsh word for ``sweet''.} if 
\[
\varphi^{-1}(Z) \subset Y^w
\]
for all $(w,Z) \in \mc{S}(X)$ with $\bc(w,Z) \neq 0$.
\end{defn:melys}

We note that any locally closed embedding $Y \hookrightarrow X$ is melys. Also, if $\bc = 0$ then every equivariant morphism is melys. The following result then forms the foundation of the article. 

\newtheorem*{thm:melyspullbackgeneral}{\Cref{thm:melyspullbackgeneral}}
\begin{thm:melyspullbackgeneral}
		If $\varphi \colon Y \to X$ is a $\bc$-melys morphism then $\varphi^* \sH_{\omega,\bc}(X,W)$ is a $\sH_{\varphi^* \omega,\varphi^* \bc}(Y,W)$-$\varphi^{-1} \sH_{\omega,\bc}(X,W)$-bimodule. 
\end{thm:melyspullbackgeneral}

As a consequence we can define the pull-back $\varphi^0 \ms{M}$ and push-forward $\varphi_0 \ms{N}$ of $\sH$-modules for any $\bc$-melys morphism, together with their derived extensions $\varphi^!$ and $\varphi_+$.  
 
\subsection{Open embeddings}

The aim of this work is to prove that the functors $\varphi^!$ and $\varphi_+$ preserve complexes with holonomic cohomology. This must be done in stages, the most important of which is to consider the case where $\varphi$ is an open embedding. Before that is possible, two key results need establishing. Currently, these results can only be established under a certain (weak) genericity condition (Condition~\ref{cond:good2}) on the function $\bc$. This condition roughly says that, for every $x \in X$, if $W^{\circ}$ is the subgroup of the stabilizer $W_x \subset \mathrm{GL}(T_x X)$ generated by all pseudo-reflections then $\bc$ is spherical on every irreducible factor of $W^{\circ}$ isomorphic to $G(m,p,n)$ and is regular on the other irreducible factors. 

If $X$ is affine and $f \in \mc{O}(X)^W$ non-zero then $D(f) = (f \neq 0)$ is a $W$-stable open set. 

\newtheorem*{thm:holonomicextension}{\Cref{thm:holonomicextension}}
\begin{thm:holonomicextension}
	Assume Condition~\ref{cond:good2} holds. If $\ms{M}$ is an irreducible holonomic $\sH_{\omega,\bc}(D(f),W)$-module then, for each $x \in X$, there exists a $W$-stable affine open neighbourhood $V$ of $x$ in $X$ and an irreducible holonomic $\sH_{\omega,\bc}(V,W)$-module $\ms{N}$ with $\ms{N} |_{D(f) \cap V} \cong \ms{M} |_{D(f) \cap V}$. 
\end{thm:holonomicextension}

When the conclusion of Theorem~\ref{thm:holonomicextension} holds, we say that the extension property holds for holonomic modules. The second key result is that holonomic modules should have finite length. We show:

\newtheorem*{prop:finLength}{\Cref{prop:finLength}}
\begin{prop:finLength}
If Condition~\ref{cond:good2} holds then holonomic $\sH_{\omega,\bc}(X,W)$-modules have finite length, as objects in the category of all quasi-coherent $\sH_{\omega,\bc}(X,W)$-modules. 
\end{prop:finLength}

We expect that both Theorem~\ref{thm:holonomicextension} and Proposition~\ref{prop:finLength} hold \textit{without any conditions} on the function $\bc$. Indeed, in the original version \cite{BETv1} of this article, we used Theorem~1.1 and Theorem~1.2 of \cite{BernsteinLosev} to prove Theorem~\ref{thm:holonomicextension} and Proposition~\ref{prop:finLength} without any conditions on the function $\bc$. However, I. Losev recently announced \cite{BernsteinLosevV4} that the proof of Theorem~1.1 of \cite{BernsteinLosev} contains a gap. In the case of $\dd$-modules, these results can be established using standard ring-theoretic results since, for a smooth affine variety $X$, the holonomic modules are precisely those of minimal Gelfand-Kirillov dimension (equal to $\dim X$). But, for Cherednik algebras, the Gelfand-Kirillov dimension of a holonomic module can take any value $0 \le d \le \dim X$ if the parameter $\bc$ is chosen appropriately so these ring-theoretic results are not sufficient.

\textit{For the remainder of the introduction,} we assume that the conclusion of both Theorem~\ref{thm:holonomicextension} and Proposition~\ref{prop:finLength} hold.

\newtheorem*{thm:jpushforwardholo}{\Cref{thm:jpushforwardholo}}
\begin{thm:jpushforwardholo}
	Let $j \colon U \hookrightarrow X$ be an open embedding. If $\ms{M}$ is a holonomic $\sH_{\omega,\bc}(U,W)$-module then the complex $j_+ \ms{M}$ has holonomic cohomology. 
\end{thm:jpushforwardholo}

We introduce a theory of $b$-functions for Cherednik algebras in order to prove the above theorem. Assume that $X$ is affine and $U = D(f)$ for some $W$-invariant function $f$. Let $s$ be a formal variable and $\sH_{\omega,\bc}(U,W)[s]$ the polynomial ring with coefficients in $\sH_{\omega,\bc}(U,W)$. Just as for $\dd$-modules, one can introduce, for each $\sH_{\omega,\bc}(U,W)$-module $\ms{M}$, the $\sH_{\omega,\bc}(U,W)[s]$-module $\ms{M} f^s$. Then the analogue of Bernstein's Theorem on the existence of a $b$-function is the following statement, which is a consequence of \Cref{prop:Msholnomic}. 

\begin{thm}\label{thm:melyspullbackgeneralholo}
	The $\sH_{\omega,\bc}(X,W)_{\kk(s)}$-module $j_0 (\ms{M} f^s)_{\kk(s)}$ is holonomic and for any $m \in \ms{M}$ there exists $0 \neq b(s) \in \kk[s]$ and an operator $D \in \sH_{\omega,\bc}(X,W)[s]$ such that 
	$$
	D(f m \otimes f^{s}) = b(s) m \otimes f^s.
	$$
\end{thm} 

From this, we deduce Theorem~\ref{thm:jpushforwardholo} above. Taking $\omega = 0$ and $\ms{M} = \mc{O}$, we recover the classical $b$-function equation (but now for Cherednik algebras). 

\begin{cor}
	For any $W$-invariant function $f$ there exists $0 \neq b(s) \in \kk[s]$ and an operator $D \in \H_{\bc}(X,W)[s]$ such that 
	$$ 
	D(f^{s+1}) = b(s) f^s. 
	$$
\end{cor}

If $i \colon Y \to X$ is a closed embedding then it is straightforward to show that $i_+$ preserves holonomicity (Corollary~\ref{cor:holoclosedpushforward}). It is then a consequence of Theorem~\ref{thm:jpushforwardholo}, applied to the usual open-closed diagram, that $i^!$ also preserves holonomicity; see Corollary~\ref{cor:pullbackclosedholo}.

\subsection{Preservation of holonomicity}

Once the case of an open (or closed) embedding is understood, there is a standard argument for $\dd$-modules that reduces the case of a general morphism to the composition of a closed embedding and a projection. Moreover, one may reduce to the situation where all spaces involved are affine $n$-spaces $\mathbb{A}^n$, for various $n$. In the key example of the projection $p \colon \mathbb{A}^n \times \mathbb{A}^m \to \mathbb{A}^n$, the functor $p_+$ is shown to preserve holonomicity by Bernstein's trick, using the Fourier transform on $\mathbb{A}^n \times \mathbb{A}^m$ to relate the cohomology of $p_+$ to that of $i^!$. In the case of Cherednik algebras, our approach is broadly the same. However, it becomes far more involved because an arbitrary $\bc$-melys morphism cannot be factored as a melys closed embedding followed by a melys projection. 

Working backwards, we first show that if $\mf{k}, \h$ are linear representations of $W$ and $\varphi \colon \mf{k} \to \h$ is $\bc$-melys, then $\varphi^!$ and $\varphi_+$ preserve holonomicity; Theorem~\ref{thm:linearpreservholonomic}. This is done by showing that there is a factorization 
\begin{equation}\label{eq:factor1}
    \mf{k} \hookrightarrow \h \times \mf{k}^W \stackrel{\varphi(\mathbf{r})}{\longrightarrow} \h \times \mf{k}^W \stackrel{p}{\twoheadrightarrow} \h
\end{equation}
into melys morphisms, where the first is a closed embedding, the last is the projection and the finite morphism $\varphi(\mathbf{r})$ is a composition of simple morphisms of the form $x \mapsto x^{r_i}$ on each $1$-dimensional irreducible factor $\h_i$ of $\h$. One can show directly that $\varphi(\mathbf{r})^!$ and $\varphi(\mathbf{r})_+$ preserve holonomicity. As was already shown in the paper \cite{ThompsonHolI}, Bernstein's trick to show that $p_+$ and $p^!$ preserve holonomicity works equally well for Cherednik algebras.  

Finally, we use \'etale covers to reduce to a local situation where $y \in Y^W$ and $x = \varphi(y) \in X^W$. We then show in Proposition~\ref{prop:melysliftstolinear} that there exists a commutative diagram
\[
	\begin{tikzcd}
	\mf{k} \ar[rr,"\Phi"] & & \h \\
	Y' \arrow[u,hook] \ar[r,"\phi"'] 	& Y  \ar[r,"\varphi"']   & X \ar[u,hook]
	\end{tikzcd} 
\]
of melys morphisms, where the vertical ones are (melys) closed embeddings, $\phi$ is \'etale and $\mf{k},\mf{h}$ are vector spaces on which $W$ acts linearly. This allows us to conclude (\Cref{thm:preservholonomicback} and \Cref{thm:preservholonomicforward}):

\begin{thm}\label{thm:mainholonomicintro}
	Let $\varphi \colon Y \to X$ be a $\bc$-melys morphism. Then $\varphi^!$ and $\varphi_+$ restrict to functors between the derived categories $D^b_{\mr{Hol}}(\sH_{\varphi^* \omega,\varphi^* \bc}(Y,W))$ and $D^b_{\mr{Hol}}(\sH_{\omega,\bc}(X,W))$. 
\end{thm}

As noted in \cite{ThompsonHolI}, there is a duality on the derived category of complexes with coherent cohomology that preserves holonomicity (it does \textit{not} induce an exact duality on the abelian category of holonomic modules). Therefore, Theorem~\ref{thm:mainholonomicintro} allows us to define the adjoints $\varphi_!,\varphi^+$ to $\varphi^!$ and $\varphi_+$. 

When $\kk = \C$, there is a notion of regular singularities for holonomic modules; see Definition~\ref{defn:regsingdefn}. It is natural to expect:

\newtheorem*{conj:preservregsing}{\Cref{conj:preservregsing}}
\begin{conj:preservregsing}
The functors $\varphi^!,\varphi_!,\varphi^+, \varphi_+$ and $\D_X$ preserve regular holonomic modules. 
\end{conj:preservregsing}

\subsection{A classification of irreducible holonomic modules}

The existence of pull-back and push-forward functors for holonomic modules allows one to give a general classification of irreducible holonomic modules. This is done in Section~\ref{sec:classifcationholonomic}. In order to state the main result, we introduce some definitions. 

A subgroup $\Pa$ of $W$ is said to be parabolic if it is the stabilizer of some point $x \in X$. The $W$-connected components $\St$ of the locally closed subset of $X$ consisting of points with stabilizer conjugate to $\Pa$, as $\Pa$ runs over all conjugacy classes of parabolics, define a finite stratification (the stabilizer stratification) of $X$. Let $N = N_W(\Pa)$ be the normalizer of $\Pa$ in $W$. A smooth affine subset $Z$ of a stratum $\St$ is said to be \textit{parallelizable} if its normal bundle $\mc{N}_{X/Z}$ is $\Pa$-equivariantly trivial and $n(Z) = Z$ or $n(Z) \cap Z = \emptyset$ for each $n \in N$; see Definition~\ref{defn:admissibleset} for the precise definition. The support of an irreducible holonomic module is the closure of the $W$-saturation of some parallelizable subset. 

To each parallelizable subset $Z$ we associate \textit{admissible data} (Definition~\ref{defn:admissibledatum}) $(Z,M,\lambda,\xi)$ consisting of an irreducible integrable $\omega$-connection $M$ on $Z$, an irreducible representation $\lambda$ of $\Pa$ and an irreducible representation $\xi$ of the twisted group algebra $\kk_{\tau} \uN(M,\lambda)$. This final object requires some explanation. One can show that the pair $(M,\lambda)$ defines a $\sH_{\omega,\bc}(U_0,\Pa)$-module $M \boxtimes L(\lambda)$ on some neighbourhood $U_0$ of $Z$ in $X$ (with support equal to $Z$). Clifford Theory implies that there exists a subgroup $\uN(M,\lambda)$ of $N/\Pa$ and a $2$-cocycle $\tau$ of $\uN(M,\lambda)$ such that 
\[
\End_{\sH_{\omega,\bc}(U_0,N)} \Ind_{\Pa}^N (M \boxtimes L(\lambda))\cong \kk_{\tau} \uN(M,\lambda)
\] 
is isomorphic to the associated twisted group algebra; see Theorem~\ref{thm:CliffordDade}. We choose an irreducible representation $\xi$ of this twisted group algebra. 

\newtheorem*{thm:mainclassificationnRCA}{\Cref{thm:mainclassificationnRCA}}
\begin{thm:mainclassificationnRCA}
	The irreducible holonomic $\sH_{\omega,\bc}(X,W)$-modules are parametrized, up to isomorphism, by the equivalence classes of admissible data as $\St$ runs over all strata.  
\end{thm:mainclassificationnRCA}

More specifically, we describe how one constructs an irreducible holonomic $\sH_{\omega,\bc}(X,W)$-module from a datum $(Z,M,\lambda,\xi)$. There is an obvious equivalence $(Z,M,\lambda,\xi) \sim (Z',M',\lambda',\xi')$ on the set of all admissible data and we show that the resulting irreducible holonomic modules are isomorphic if and only if the data are equivalent.      

As an application of the classification of irreducible holonomic modules, we establish the following finiteness result on Ext-groups between holonomic modules.

\newtheorem*{cor:dimextRCA2}{\Cref{cor:dimextRCA2}}
\begin{cor:dimextRCA2}
	Let $\ms{M},\ms{N}$ be holonomic $\sH := \sH_{\omega,\bc}(X,W)$-modules. Then 
$$
\dim_{\kk} \Ext_{\sH}^i(\ms{M},\ms{N}) < \infty, \quad \forall i \ge 0. 
$$ 
\end{cor:dimextRCA2}

\subsection{Other applications}\label{sec:otherapplications}

Finally, we note other applications of our main results. 

First, one can ask when $\sH_{\omega,\bc}(X,W)$ is a sheaf of simple algebras. We show:

\newtheorem*{prop:simplealgebraH}{\Cref{prop:simplealgebraH}}
\begin{prop:simplealgebraH}
Assume that $X$ is irreducible and $W$ acts effectively on $X$. The set of parameters $(\omega,\bc) \in \mathbb{H}^2(X,\Omega^{1,2}_X)^W \oplus \mc{S}(X)^W$ where the sheaf of algebras $\sH_{\omega,\bc}(X,W)$ is not simple is contained in a countable union of affine hyperplanes. 
\end{prop:simplealgebraH}

Recall that a parameter $(\omega,\bc)$ is said to be aspherical if there exists a non-zero $\sH_{\omega,\bc}(X,W)$-module $\ms{M}$ killed by the trivial idempotent $e \in \kk W$. 

\newtheorem*{prop:asphericalvalueshyper}{\Cref{prop:asphericalvalueshyper}}
\begin{prop:asphericalvalueshyper}
Assume that $X$ is irreducible and $W$ acts effectively on $X$. The set of aspherical values is contained in a finite union of affine hyperplanes in $\mathbb{H}^2(X,\Omega^{1,2}_X)^W \oplus \mc{S}(X)^W$.
\end{prop:asphericalvalueshyper}

We define shift functors $\mc{S}_{\bc \to \bc'}$ on the category of holonomic modules. If $\mc{A}$ denotes the set of all reflection hypersurfaces in $X$ then fix a $W$-stable subset $\Omega \subset \mc{A}$ and let $\bc, \bc'$ be parameters such that $\bc(w,Z) = \bc'(w,Z)$ for all $(w,Z)$ with $Z \notin \Omega$. We construct an exact shift functor 
\[
\mc{S}_{\bc \to \bc'} \colon \mc{M}(\sH_{\omega,\bc}(X,W)) \to \mc{M}(\sH_{\omega,\bc'}(X,W))
\]
from $\sH_{\omega,\bc}(X,W)$-modules to $\sH_{\omega,\bc'}(X,W)$-modules. Theorem~\ref{thm:jpushforwardholo} implies that $\mc{S}_{\bc \to \bc'}$ restricts to a functor between the respective categories of holonomic modules; see Proposition~\ref{prop:shiftfunctor}.  
	
Finally, for each $x \in X$ we also introduce a \textit{Jacquet functor} 
\[
\mc{J}_x \colon \mc{M}_{\mr{Coh}({\mc{O}})}(\sH_{\omega,\bc}(X,W)) \to \mc{O}_{\bc_x}(\h,W_x)
\]
from the category of $\sH_{\omega,\bc}(X,W)$-modules coherent over $\mc{O}_X$ to the category $\mc{O}$ for the pair $(\h,W_x)$, where $\h = T_x X$. We show in Lemma~\ref{lem:Jacqeutfunctor} that $\mc{J}_x$ is an exact functor and $\mc{J}_x(\ms{M}) = 0$ if and only if $x \notin \Supp_X \ms{M}$.

\subsection{Notation}

As in \cite{BorelDmod}, pull-back and push-forward of sheaves are denoted $\varphi^{-1}$ and $\varphi_{\idot}$ respectively. For a quasi-coherent $\mc{O}$-module $\ms{M}$, pull-back and push-forward are denoted $\varphi^{*} \ms{M} = \mc{O}_Y \o_{\varphi^{-1} \mc{O}_X} \varphi^{-1} \ms{M}$ and $\varphi_*(\ms{M}) = \varphi_{\idot}(\ms{M})$. Finally, for modules over Cherednik algebras and $\bc$-melys morphisms $\varphi$, we have $\varphi^0, \varphi_0$ and their associated derived functors $\varphi^!, \varphi_+$. 

\subsection*{Acknowledgements}

We would like to thank the referee for an extremely detailed review which greatly improved the exposition in the article. The first author was partially supported by a Research Project Grant from the Leverhulme Trust and by the EPSRC grants EP-W013053-1 and EP-R034826-1. The work of the second author was partially supported by the NSF grant DMS-2001318. We would like to thank Ivan Losev for helpful correspondence. 

\section{Sheaves of Cherednik algebras}

In this section we introduce sheaves of Cherednik algebras on a smooth variety. 

\subsection{Conventions}\label{sec:conv} Let $\kk$ be an algebraically closed field of characteristic zero. By a scheme, we will always mean a separated scheme of finite type over $\kk$. A variety will mean a reduced scheme, not necessarily irreducible. In particular, they may be disconnected. Sheaves of modules over a sheaf of rings will be denoted $\ms{M},\ms{N}$,... whilst modules over a ring will be denoted $M,N$,... Unadorned tensor products will always be over $\kk$. The sheaf of vector fields, resp. one-forms, on a smooth variety $X$ is denoted $\Theta_X$, resp. $\Omega^1_X$. If $X$ is a variety then $\mc{O}(X)$ is shorthand for $\Gamma(X,\mc{O}_X)$. 

All our spaces will be equipped with the action of a finite group $W$. We do not assume that this action is effective. Throughout, we make the following:

\begin{assume}
All spaces admit a (finite) cover by $W$-stable affine open sets. 
\end{assume}

Under the above assumption, the quotient $X/W$ exists as a scheme \cite[Proposition~1.3]{BertinNotes} with the closed points of $X/W$ in bijection with the $W$-orbits in $X$ since $\kk$ is algebraically closed. We note that if $X$ is quasi-projective, then it admits a cover by $W$-stable affine open sets \cite[Lemma~1.2]{BertinNotes}. 

\begin{remark}
As explained in \cite[Remark~2.32]{ChereSheaf}, we do not need to assume that our spaces admit a cover by $W$-stable affine open subsets in order to define sheaves of Cherednik algebras. Instead, one can take \'etale $W$-stable covers by affine open sets. Our assumption is made to simplify the exposition.
\end{remark}

A subset $Z \subset X$ is said to be \textit{$W$-connected} if it is $W$-stable and $W$ acts transitively on the connected components of $Z$. The subset $Z$ is said to be \textit{$W$-irreducible} if it is $W$-stable and $W$ acts transitively on the irreducible components of $Z$. If $\pi \colon X \to X/W$ is the quotient map then $Z$ is $W$-connected if and only if $\pi(Z)$ is connected and it is $W$-irreducible if and only if $\pi(Z)$ is irreducible. If $w \in W$, then $X^w$ denotes the set of all points fixed under the automorphism $w$.

Since we wish to deal with objects such as the skew group ring $\mc{O}_X \rtimes W$, we work throughout on the quotient variety $X/W$. Note that one could equivalently work with the \textit{$W$-equivariant Zariski topology on $X$}: a subset $U \subset X$ is an open subset in this topology if and only if it is  open in the Zariski topology and $W$-stable. Pushing forward by $\pi$, we consider $\mc{O}_X,\Omega^1_X$ and $\mc{O}_X \rtimes W$ etc. as sheaves on $X/W$. 

If $\ms{F}$ is a locally free $W$-equivariant coherent sheaf on $X$ then each $x \in X$ admits an affine $W$-stable neighbourhood $U$ such that $\ms{F} |_U$ is free (though not necessarily equivariantly trivial). This follows from the fact that finitely generated projective modules over the semilocal ring $\mc{O}_{X,W \cdot x}$ of rational functions regular on the orbit $W \cdot x$ are actually free.  

The morphisms $\varphi : Y \rightarrow X$ that we consider will always be assumed to be $W$-equivariant.

\subsection{}\label{sec:truncatedeRham} Let $X$ be a smooth variety over $\kk$. Let $Z$ be a smooth subvariety of $X$ of codimension one and write $\mc{O}_X(Z)$ for the line bundle associated to $Z$, thought of as an effective divisor. Then $\mc{O}_X(Z)$ is the sheaf of functions with a pole of order at most one along $Z$. Locally, the ideal defining $Z$ is principal, generated by one section, $f_Z$ say. Then the element  
$$
d \log f_Z := \frac{d f_Z}{f_Z} 
$$
is a section of $\Omega^1_X(Z) = \Omega^1_X \o \mc{O}_X(Z)$. Contraction $i$ of differential forms by vector fields defines a pairing $\Theta_X \o \Omega^1_X (Z) \rightarrow \mc{O}_X(Z)$ given by $(\nu,\omega) \mapsto i_{\nu} (\omega)$. 

Let $\Omega_X^{1,2}$ be the two term subcomplex $\Omega_X^1 \stackrel{d}{\longrightarrow} (\Omega_X^2)^{\mathrm{cl}}$, concentrated in degrees $1$ and $2$, of the algebraic de-Rham complex of $X$, where $(\Omega_X^2)^{\mathrm{cl}}$ denotes the subsheaf of closed forms in $\Omega_X^2$. Sheaves of twisted differential operators on $X$ are parameterized, up to isomorphism, by the second hypercohomology group $\mathbb{H}^2(X,\Omega_X^{1,2})$; see \cite[Section~2.1]{BBJantzen}. Given $\omega \in \mathbb{H}^2 ( X,\Omega_X^{1,2})$, the corresponding sheaf of twisted differential operators is denoted $\dd_{\omega}(X)$.

\subsection{Dunkl-Opdam operators}\label{sec:Dunkloperator} Recall that $X$ is assumed smooth. We do not assume $X$ is connected, but assume it is $W$-connected; the general case is obtained by treating $W$-connected components separately. In particular, $X$ is equidimensional. Let $\mc{S}(X)$ be the set of pairs $(w,Z)$ where $w \in W$ and $Z$ is a connected component of $X^w$ of codimension one. Any such $Z$ is smooth. A pair $(w,Z)$ in $\mc{S}(X)$ will be called a \textit{reflection} of $(X,W)$ and $Z$ a \textit{reflection hypersurface}. The group $W$ acts on $\mc{S}(X)$, and we fix $\bc : \mc{S}(X) \rightarrow \kk$ to be a $W$-equivariant function with $W$ acting trivially on $\kk$.

Recall from \cite{BBJantzen} that an $\mc{O}_X$-module $\mc{P}$ is a \textit{Picard algebroid} if there exists a Lie bracket $[ - , - ] \colon \mc{P} \o \mc{P} \to \mc{P}$ with morphisms $i \colon \mc{O}_X \to \mc{P}$ and $\sigma \colon \mc{P} \to \Theta_X$ (the anchor map) forming a short exact sequence $0 \to \mc{O}_X \to \mc{P} \to \Theta_X \to 0$ such that $\sigma$ is a morphism of Lie algebras, $[i(1),p] = 0$ for all $p \in \mc{P}$, and 
$$
[p_1,fp_2] = f[p_1,p_2] + \sigma(p_1)(f) p_2, \quad \forall \, p_1,p_2 \in \mc{P}, f \in \mc{O}_X. 
$$
The Picard algebroid is said to be $W$-equivariant if there are isomorphisms $\psi_w : w^* (\mc{P}) \stackrel{\sim}{\rightarrow} \mc{P}$ of algebroids, satisfying the usual cocycle condition, such that the inclusion $\mc{O}_X \hookrightarrow \mc{P}$ and anchor map $\sigma : \mc{P} \rightarrow \Theta_X$ are $W$-equivariant. Each class $[\omega] \in \mathbb{H}^2(X,\Omega_X^{1,2})^W$ can be represented by an invariant $2$-cocycle $\omega$. The corresponding Picard algebroid $\mc{P}_{\omega}$ is $W$-equivariant. We fix one such $W$-equivariant Picard algebroid $\mc{P}_{\omega}$. 

We say that an affine open neighbourhood $U$ of $x \in  X$ is \textit{good} if it is $W$-stable and $Z \cap U$ is principal for all $(w,Z) \in \mc{S}(X)$. 

\begin{lem}\label{lem:goodneighbourhood}
	Each point $x \in X$ admits a good open neighbourhood. 
\end{lem}

\begin{proof}
	We may assume $X$ is affine. Let $\mc{O}_{X,W \cdot x}$ denote the semilocal ring of functions regular on the orbit $W \cdot x$. Since $X$ is smooth, \cite[Exercise~20.5]{MatCom} says that each indecomposable summand of $\mc{O}_{X,W \cdot x}$ is a UFD. Thus, for each $(w,Z) \in \mc{S}(X)$ there exists $g_Z \in \mc{O}(X)$ such that $g_Z$ has no zeros in $W \cdot x$ and $Z \cap (g_Z \neq 0)$ is principal. Replacing $g_Z$ by $\prod_{w \in W} w(g_Z)$, we may assume $g_Z$ is $W$-invariant. Finally, let $g = \prod_{(w,Z)} g_Z$. Then $(g \neq 0)$ is the required good open neighbourhood of $x$. 
\end{proof}

We choose a covering $\{ U_i \}$ of $X$ consisting of good open sets. This means we can fix functions $f_{Z,i}$ defining $U_i \cap Z$ for all $(w,Z) \in \mc{S}(X)$. The union of all $Z$'s is denoted $E$. If $j \colon X \setminus E \hookrightarrow X$ is the inclusion, then write $\mc{P}_{\omega} (E)$ for the sheaf $j_{\idot} (\mc{P}_{\omega} |_{X \setminus E})$.   

\begin{defn}
	To each $v \in \Gamma(U_i,\mc{P}_{\omega})$, the associated Dunkl-Opdam operator is 
	\begin{equation}\label{eq:DOop}
	D_{v} = v + \sum_{(w,Z) \in \mc{S}(X)} \frac{2 \bc (w,Z)}{1 - \lambda_{w,Z}} i_{ \sigma (v)} (d \log f_{Z,i}) (w - 1),
	\end{equation}
	where $\lambda_{w,Z}$ is the eigenvalue of $w$ on each fiber of the conormal bundle of $Z$ in $X$. 
\end{defn}

The operator $D_v$ is a section of $\mc{P}_{\omega}(E) \o \kk W$ over $U_i$. The $\Gamma(U_i,\mc{O}_X \o \kk W)$-submodule of $\mc{P}_{\omega}(E) \o \kk W$ generated by all Dunkl-Opdam operators $\{ D_v \ | \ v \in \Gamma(U_i,\mc{P}_{\omega}) \}$ and the subspace $\Gamma(U_i,\mc{O}_X \o \kk W)$ is denoted $\Gamma(U_i, \mc{F}^1)$. The modules $\Gamma(U_i, \mc{F}^1)$ glue to form a sheaf $\mc{F}^1$ in the $W$-equivariant Zariski topology on $X$. The anchor map $\sigma : \mc{P}_{\omega}(E) \o \kk W \rightarrow \Theta_X(E) \o \kk W$ restricts to a map $\mc{F}^1 \rightarrow \Theta_X \o \kk W$ which fits into a short exact sequence
\begin{equation}\label{eq:ses}
0 \rightarrow \mc{O}_X \o \kk W \rightarrow \mc{F}^1 \stackrel{\sigma}{\longrightarrow} \Theta_X \o \kk W \rightarrow 0.
\end{equation}

\begin{defn}
	The subsheaf of algebras of $j_{\idot} (\dd_{\omega}(X \setminus E) \rtimes W)$ generated by $\mc{F}^1$ is called the \textit{sheaf of Cherednik algebras} associated to $W$, $\omega$ and $\bc$. It is denoted $\sH_{\omega,\bc}(X,W)$.   
\end{defn}

The global sections of $\sH_{\omega,\bc}(X,W)$ are denoted $\H_{\omega,\bc}(X,W)$. The space of parameters for the family $\sH$ is $\mathbb{H}^2(X,\Omega^{1,2}_X)^W \oplus \mc{S}(X)^W$. 

\begin{example}
	The following example shows that $\mathbb{H}^2(X,\Omega^{1,2}_X)^W$ need not be finite-dimensional when $X$ is not affine. Let $X = \kk^2 \setminus \{ (0,0) \}$, acted on by $W = (\Z/\ell \Z)^2$ in the obvious way. Let $\mr{DR}_2$ be the de Rham complex truncated at degree 2 (ending with closed forms). We have a short exact sequence $0 \to \Omega^{1,2}_X \to \mr{DR}_2 \to \mc{O}_X \to 0$ of complexes. Now, the cohomology of $\mr{DR}_2$ is $H^{\idot}(X,\kk)$ in degrees $\le 2$ and we get a $W$-equivariant sequence 
	\[
	H^1(X,\kk) \to H^1(X,\mc{O}_X) \to \mathbb{H}^2(X,\Omega^{1,2}_X) \to H^2(X,\kk).	
	\]
	Since $H^1(X,\kk) = H^2(X,\kk) = 0$, this implies that $\mathbb{H}^2(X,\Omega^{1,2}_X)$ is finite-dimensional if and only if $H^1(X,\mc{O}_X)$ is finite-dimensional. Similarly, $\mathbb{H}^2(X,\Omega^{1,2}_X)^W$ is finite-dimensional if and only if $H^1(X,\mc{O}_X)^W$ is finite-dimensional. But ${H}^1(X,\mc{O}_X)$ is spanned by $x^iy^j$, for $i,j \le -1$, when $X = \kk^2\setminus \{ (0,0) \}$; see \cite[III, Exercise~4.3]{Hartshorne}.
	
One can construct the corresponding sheaves of twisted differential operators as follows. Let us take the usual sheaves of differential operators on $\kk^{\times} \times \kk$ and $\kk \times \kk^{\times}$. We have to glue them on $\kk^{\times} \times \kk^{\times}$, preserving the order filtration. Let us glue by preserving functions and mapping
\[
\frac{\partial}{\partial x} \mapsto \frac{\partial}{\partial x} + \frac{\partial F}{\partial x}, \quad \frac{\partial}{\partial y} \mapsto \frac{\partial}{\partial y} + \frac{\partial F}{\partial y},
\] 
where $F=(xy)^{-1}P(x^{-1},y^{-1})$. If $F_1 \neq F_2$ are Laurent polynomials of this form then one can check that the corresponding sheaves of twisted differential operators are not isomorphic as filtered algebras. 

When $F$ is $W$-invariant, that is, $F=(xy)^{-\ell}P(x^{-\ell},y^{-\ell})$, the sheaf of twisted differential operators can be deformed to sheaves of Cherednik algebras.   
\end{example}

\subsection{}\label{sec:PBW} There is a natural order filtration $\F^{\idot}$ on $\sH_{\omega,\bc}(X,W)$, defined in one of two ways. Either one defines $\F^{\idot}$ to be the restriction to $\sH_{\omega,\bc}(X,W)$ of the order filtration on $j_{\idot} (\dd_{\omega}(X \setminus E) \rtimes W)$, or, equivalently, by giving elements in $\F^1$ degree at most one, with $D \in \F^1$ having degree one if and only if $\sigma(D) \neq 0$, and then defining the filtration inductively by setting $\F^i = \F^{1}  \F^{i-1}$. By definition, the filtration is exhaustive. Let $p \colon T^* X \rightarrow X$ be the projection map. The second author has shown in Theorem 2.11 of \cite{ChereSheaf} that the algebras $ \sH_{\omega,\bc}(X,W)$ are a flat deformation of $\dd(X) \rtimes W$ over the base $\mathbb{H}^2(X,\Omega^{1,2}_X)^W \oplus \mc{S}(X)^W$. Equivalently, the PBW property holds for Cherednik algebras: 

\begin{thm}\label{thm:PBW}
	$\gr_{\F} \sH_{\omega,\bc}(X,W) \cong p_{\idot} \mc{O}_{T^* X} \rtimes W$. 
\end{thm}

Theorem \ref{thm:PBW} implies that, for any affine $W$-stable open set $U \subset X$, the algebra of global sections $\Gamma(U, \sH_{\omega,\bc}(X,W))$ has finite global dimension; namely, its global dimension is bounded by $2 \dim X$. Moreover, it implies that $\sH_{\omega,\bc}(X,W)$ is a coherent sheaf of algebras. A $\sH_{\omega,\bc}(X,W)$-module $\ms{M}$ can be regarded as an $\mc{O}_{X/W}$-module. Equation \eqref{eq:DOop} implies that each section of $\sH_{\omega,\bc}(X,W)$ acts on $\ms{M}$ as a locally $\mathrm{ad} (\mc{O}_{X/W})$-nilpotent endomorphism. Therefore, we get a morphism $\sH_{\omega,\bc}(X,W) \rightarrow \dd_{\mc{O}_{X/W}}(\ms{M})$, where the latter is Grothendieck's sheaf of differential operators on $\ms{M}$. This morphism preserves filtrations. When $\omega = [\mc{L}]$, for some equivariant line bundle $\mc{L}$, the algebra $\sH_{\omega,\bc}(X,W)$ acts on $\mc{L}$ and the morphism $\sH_{\omega,\bc}(X,W) \to \dd_{\mc{O}_{X/W}}(\mc{L})$ is an embedding. 

\subsection{} Throughout, a $\sH_{\omega,\bc}(X,W)$-module will always mean a \textit{left} $\sH_{\omega,\bc}(X,W)$-module that is quasi-coherent over $\mc{O}_X$. The category of all $\sH_{\omega,\bc}(X,W)$-modules is denoted $\Qcoh{\sH_{\omega,\bc}(X,W)}$ and the full subcategory of all modules coherent over $\sH_{\omega,\bc}(X,W)$ is $\Coh{\sH_{\omega,\bc}(X,W)}$. The categories $\Qcoh{\sH_{\omega,\bc}(X,W)}$ and $\Coh{\sH_{\omega,\bc}(X,W)}$ are abelian. 

\subsection{Singular support}

Since the group $W$ acts symplectically on $T^* X$, the quotient scheme $(T^* X)/W$ has a natural Poisson structure. It has finitely many symplectic leaves by \cite[Proposition~2.9]{BonnafeAutomorphisms}. We recall that a (reduced) locally closed subvariety $Z \subset (T^* X) / W$ is said to be \textit{coisotropic} (resp. \textit{isotropic}) if the smooth locus of $Z \cap \mc{L}$ is coisotropic (resp. isotropic) in $\mc{L}$, for each symplectic leaf $\mc{L}$ of $(T^* X) / W$. Let $\bpi \colon T^* X \to (T^* X)/W$ be the quotient map. The following result is implicit in \cite{ThompsonHolI}. 

\begin{lem}\label{lem:isoclosed}
	Let $K$ be a (reduced) locally closed subvariety of $(T^* X)/W$. 
	\begin{enumerate}
		\item[(i)] $K$ is isotropic if and only if $\bpi^{-1}(K)$ is isotropic in $T^* X$. 
		\item[(ii)] If $K$ is isotropic then so too is its closure.  
	\end{enumerate}
\end{lem} 

\begin{proof}
	Part (i). We assume that $K$ is isotropic in $(T^* X)/W$. Let $K' = \bpi^{-1}(K)$. We must show that there is a dense open subset of $K'$ such that $\omega |_{T_y K'} = 0$ for all $y$ in this open set. Here $\omega$ is the symplectic form on $T^* X$. We may assume that $K$ is irreducible. Then there exists a unique leaf $\mc{L}$ such that $K \cap \mc{L}$ is open dense in $K$ and hence $K' \cap \mc{L}'$ is open dense in $K'$, where $\mc{L}' := \bpi^{-1}(\mc{L})$. Shrinking $K$ if necessary, we can assume that $K \cap \mc{L}$ is smooth. Hence $\omega_{\mc{L}} |_{T_k (K \cap \mc{L})} = 0$ for all $k \in K \cap \mc{L}$, where $\omega_{\mc{L}}$ is the symplectic form on $\mc{L}$. The restriction $\bpi |_{\mc{L}'} \colon \mc{L}' \to \mc{L}$ is \'etale and, by definition, 
	\begin{equation}\label{eq:leafomega}
		\omega_{\mc{L}} (d \bpi(-), d \bpi(-)) = \omega |_{\mc{L}'}.
	\end{equation}
	This implies that $\omega |_{T_y K'} = \omega |_{T_y (K' \cap \mc{L}')} = 0$ for all $y \in K' \cap \mc{L}'$. Hence $K'$ is isotropic in $T^* X$. 
	
	Assume now that $K'$ is isotropic in $T^* X$ and that $\mc{L}$ is now an arbitrary leaf in $(T^* X) / W$. Again, shrinking $K$ if necessary, we may assume that the intersection $K \cap \mc{L}$ is smooth. Then $K' \cap \mc{L}'$ is smooth since $\bpi |_{\mc{L}'}$ is \'etale. Then \cite[Proposition 1.3.30]{CG} says that $\omega |_{T_y (K' \cap \mc{L}')} = 0$ for all $y \in K' \cap \mc{L}'$. Equation~\eqref{eq:leafomega} now implies that $\omega_{\mc{L}} |_{T_k (K \cap \mc{L})} = 0$ for all $k \in K \cap \mc{L}$.      

	Part (ii). By part (i), it suffices to show that $\bpi^{-1}(\overline{K}) = \overline{\bpi^{-1}(K)}$ is isotropic in $T^* X$, given that $\bpi^{-1}(K)$ is isotropic. This is a consequence of \cite[Proposition 1.3.30]{CG}.
\end{proof}

Assume that $X$ is affine. Let $\iota_i$ be a linear splitting of the principal symbol map 
$$
\sigma_i : \mc{F}^i \rightarrow \gr_{\mc{F}}^i \H_{\omega,\bc}(X,W) = \mc{O}(T^* X)_i \o \kk W.
$$
Let $C = \mc{O}(T^*X)^W \subset Z(\mc{O}(T^* X) \rtimes W)$, and write $\iota: C \rightarrow \H_{\omega,\bc}(X,W)$ for the linear map induced by $\iota_i$. The map $\iota$ satisfies properties (i) and (ii) of \cite[Section 2.1]{BernsteinLosev} (with $d = 1$). Therefore, associated to any finitely generated $\H_{\omega,\bc}(X,W)$-module $M$ is its singular support $\SS (M)$, a closed subset of $(T^* X)/W$. This set does not depend on the choice of good filtration.

More generally, if $X$ is not affine, we can cover $X$ by $W$-stable affine open subsets $U_i$. If $\ms{M}$ is a coherent $\sH_{\omega,\bc}(X,W)$-module and $M_i = \Gamma(U_i,\ms{M})$, then the singular supports $\SS(M_i)$ agree on overlaps $U_i \cap U_j$, and hence define a global variety $\SS(\ms{M})$ in $(T^* X)/W$. 

If $\ms{M}$ is finitely generated, then  for each leaf $\mc{L} \subset (T^* X)/W$, the intersection $\mc{L} \cap \SS(\ms{M})$ is a coisotropic subvariety of $\mc{L}$ since Gabber's Theorem \cite{Gabber} says that the radical ideal defining $\SS(\ms{M})$ is involutive. That is, $\SS(\ms{M})$ is coisotropic.

\begin{defn}
	A coherent $\sH_{\omega,\bc}(X,W)$-module $\ms{M}$ is called \textit{holonomic} if $\SS(\ms{M})$ is an isotropic subvariety of $(T^* X) / W$. 
\end{defn}

The full subcategory of $\Coh{\sH_{\omega,\bc}(X,W)}$ consisting of holonomic modules is denoted $\Hol{\sH_{\omega,\bc}(X,W)}$.

\begin{example}
	If $\ms{M}$ is a $\sH_{\omega,\bc}(X,W)$-module coherent over $\mc{O}_X$ then $\ms{M}$ is holonomic. Indeed, giving $\ms{M}$ the (good) filtration with $\mc{F}_{-1}\ms{M} = 0$ and $\mc{F}_i \ms{M} = \ms{M}$ for all $i \ge 0$ shows that $\SS(\ms{M})$ is contained in the zero section $X/W$ of $(T^* X)/W$. 
\end{example} 

Fixing a good filtration $\mc{F}_{\idot}$ on a coherent module $\ms{M}$, we may consider $\mr{gr}_{\mc{F}} \ms{M}$ as a module over either $\mc{O}_{(T^*X)/W}$ or $\mc{O}_{T^*X}$. This defines singular support $\SS(\ms{M}) \subset (T^* X)/W$ as above, or $\SS_X(\ms{M}) \subset T^* X$. Then $\SS_X(\ms{M}) = \bpi^{-1}(\SS(\ms{M}))$ and Lemma~\ref{lem:isoclosed} implies that $\ms{M}$ is holonomic if and only if $\SS_X(\ms{M})$ is isotropic in $T^* X$.  

Similarly, the support $\Supp \ms{M}$ of an $\H_{\omega,\bc}(X,W)$-module is a subset of $X/W$, closed if $\ms{M}$ is coherent. If we instead wish to consider the support of $\ms{M}$ thought of as a sheaf on $X$ then we write $\Supp_X \ms{M}$ for this. We have $\Supp_X \ms{M} = \pi^{-1}(\Supp \ms{M})$.

\subsection{Stabilizer stratification}\label{sec:stabstrata}

If $X$ is a smooth $W$-variety then it admits a \emph{stabilizer stratification}, which we now describe. A subgroup $\Pa \subset W$, which occurs as the isotropy group of a point of $X$, is called a \emph{parabolic} subgroup of $W$ (with respect to its action on $X$). For each parabolic subgroup $\Pa$ of $W$, we let $\St_0$ be a connected component of $\{ x \in X \, | \, W_x = \Pa \}$. Then we say that $\St := W(\St_0)$ is a \textit{stratum} of $X$. This defines a partition of $X$ into smooth, locally closed $W$-connected sets. The image of these strata in $X/W$ defines a partition of the latter by smooth locally closed and connected subvarieties. 

The stratification of $X$ induces a filtration on the category $\Qcoh{\sH_{\omega,\bc}(X,W)}$ of (quasi-coherent) modules, and by restriction on the full subcategories of coherent, respectively holonomic, modules.

\begin{defn}\label{defn:SerreYXsubcat}
	For any locally closed $W$-stable subset $Y\subset X$, the category $\Qcoh{\sH_{\omega,\bc}(X,W)}_Y$ (resp. $\Coh{\sH_{\omega,\bc}(X,W)}_Y$, $\Hol{\sH_{\omega,\bc}(X,W)}_Y$) denotes the category of all (resp., coherent, holonomic) $\sH_{\omega,\bc}(X,W)$-modules $\ms{M}$ with $\Supp_X \ms{M} \subset \overline{Y}$.
\end{defn}

\section{Melys morphisms}\label{sec:Melys}

In this section, we introduce melys morphisms and show that pull-back of modules is well-defined for melys morphisms. 

\subsection{The main result}

Let $\varphi \colon Y \to X$ be a $W$-equivariant morphism with $X,Y$ smooth. For each parameter $\bc$, let $\mc{S}_{\bc}(X)$ denote the set of pairs $(w,Z)$ in $\mc{S}(X)$ with $\bc(w,Z) \neq 0$. The following notion was introduced in \cite{ProjAffine}, but for a much more restrictive class of morphisms.

\begin{defn}\label{defn:melys}
	Let $\varphi \colon Y \to X$ be a $W$-equivariant morphism. 
	\begin{enumerate}
		\item $\varphi$ is \textit{$\bc$-melys} if $\varphi^{-1}(Z) \subset Y^w$ for each $(w,Z) \in \mc{S}_{\bc}(X)$.
		\item $\varphi$ is \textit{strongly $\bc$-melys} if it is $\bc$-melys and $Y_i \cap \varphi^{-1}(Z) \neq \emptyset$ implies $Y_i^w \subsetneq Y_i$ for each $(w,Z) \in \mc{S}_{\bc}(X)$ and each connected component $Y_i$ of $Y$. 
	\end{enumerate}
\end{defn}

If $Y$ is connected then every $\bc$-melys morphism can be factored as a strongly $\bc$-melys morphism $Y \to X^K$ followed by a closed embedding $X^K \hookrightarrow X$, for some normal subgroup $K \subset W$ acting trivially on $Y$.

\begin{example}\label{ex:melysbasic}
	We list two key examples of melys morphisms. 
	\begin{enumerate}
	\item If $\bc = 0$ then every equivariant morphism is $\bc$-melys.
	\item If $Y \subset X$ is any $W$-stable smooth locally closed subset then $\varphi \colon Y \hookrightarrow X$ is $\bc$-melys. 
	\end{enumerate}
	In particular, both open embeddings and closed embeddings are $\bc$-melys. 
\end{example}

\begin{example}\label{ex:melysbasic2}
	An example of a melys morphism that is not strongly melys can be constructed as follows. Take a collection of reflection hypersurfaces $(w_1,Z_1), \ds, (w_k,Z_k) \in \mc{S}_{\bc}(X)$, closed under conjugation by $W$, with $Y = \bigcap_i Z_i$ non-empty and smooth. Then $Y$ will be a $W$-stable closed subset of $X$ and the embedding $Y \hookrightarrow X$ is $\bc$-melys but not strongly $\bc$-melys. 
	
For an example of a morphism that is not $\bc$-melys, take any $X$ with $\mc{S}_{\bc}(X) \neq \emptyset$ and let $\varphi \colon Y = X \times X \to X$ be projection onto the second factor. If $W$ acts diagonally on $Y$ then $\varphi$ is not $\bc$-melys.   
\end{example}

Given a $\bc$-melys morphism $\varphi$, we define $\varphi^* \bc$ on $\mc{S}(Y)$ as follows. Let $(w,Z) \in \mc{S}(Y)$. If there exists $Z'$ (necessarily unique) with $(w,Z') \in \mc{S}_{\bc}(X)$ such that $Z$ is a connected component of $\varphi^{-1}(Z')$ then we set 
\begin{equation}\label{eq:bcpullback}
    (\varphi^* \bc)(w,Z) := \frac{1 - \lambda_{w,Z}}{1 - \lambda_{w,Z'}}\mr{val}_{Z}(\varphi^* (f_{Z'})) \bc(w,Z')
\end{equation}
to be (roughly speaking) multiplication by the order of vanishing of $\varphi^* (f_{Z'})$ along $Z$. Here, if $r = \mr{val}_{Z}(\varphi^* (f_{Z'}))$ then $\lambda_{w,Z'} = \lambda_{w,Z}^r$. If no such $Z'$ exists then we set $(\varphi^* \bc)(w,Z) = 0$. For any $W$-equivariant $\mc{O}_X$-module $\ms{M}$, let $\varphi^* \ms{M} = \mc{O}_Y \o_{\varphi^{-1} \mc{O}_X} \varphi^{-1} \ms{M}$ be the usual pull-back in the category of quasi-coherent sheaves.

The following result forms the basis of the remainder of the paper. 
	 
\begin{thm}\label{thm:melyspullbackgeneral}
	If $\varphi$ is $\bc$-melys then $\varphi^* \sH_{\omega,\bc}(X,W)$ is a $\sH_{\varphi^* \omega,\varphi^* \bc}(Y,W)$-$\varphi^{-1} \sH_{\omega,\bc}(X,W)$-bimodule. Hence, for any $\sH_{\omega,\bc}(X,W)$-module $\ms{M}$, the $\mc{O}_Y$-module $\varphi^* \ms{M}$ is a $\sH_{\varphi^* \omega,\varphi^* \bc}(Y,W)$-module. 
\end{thm} 

\begin{proof}
	Let $Y^{\circ} = Y \setminus \left(\bigcup_{Z} Z \right)$, where the union is over all $Z$ with $(w,Z) \in \mc{S}_{\varphi^* \bc}(Y)$. We will first show that, locally on $Y^{\circ}$, the sheaf $\mc{O}_{Y^{\circ}} \o_{\mc{O}_X} \sH_{\omega,\bc}(X,W)$ admits a $\dd_{\varphi^*\omega}(Y^{\circ}) \rtimes W$-module structure. Then we check that this glues to a global module structure. Finally, we check that $\mc{O}_{Y} \o_{\mc{O}_X} \sH_{\omega,\bc}(X,W)$ is an $ \sH_{\varphi^* \omega,\varphi^* \bc}(Y,W)$-submodule of $\mc{O}_{Y^{\circ}} \o_{\mc{O}_X} \sH_{\omega,\bc}(X,W)$. 
	
	Let $p \in Y$ and $q = \varphi(p)$. We replace $X$ by some good neighbourhood of $q$ (still denoted $X$) and replace $Y$ by some good neighbourhood of $p$ contained in $\varphi^{-1}(X)$. This way, we may assume that both $X$ and $Y$ are affine varieties whose reflection hypersurfaces are all principal. Fix $(w,Z) \in \mc{S}(X)$. There exists a non-trivial root of unity $\lambda_{w,Z}$ such that $I(Z) = \langle f \in \mc{O}(X) \, | \, w(f) = \lambda_{w,Z} f \rangle$. In particular, since $I(Z)$ is principal, we may assume that $w(f_Z) = \lambda_{w,Z} f_Z$. 
	
	 Since $Y$ need not be connected, we consider each connected component $Y_i$ in turn. Relative to $Y_i$, we categorize the reflection hypersurfaces $(w,Z)$ in $\mc{S}_{\bc}(X)$ into three types
	\begin{enumerate}
		\item[(I)] $Y_i \subset \varphi^{-1}(Z)$ (and hence $Y_i^w = Y_i$ by the melys condition),
		\item[(II)] $\varphi^{-1}(Z) \cap Y_i = \emptyset$; or 
		\item[(III)] $\varphi^{-1}(Z) \cap Y_i$ is a union of connected components of $Y_i^w$, each one of codimension one in $Y_i$.  
	\end{enumerate}
	Let us justify why only these three cases can occur. If we are not in (I) or (II) then $\varphi^{-1}(Z) \cap Y_i$ is a non-empty but proper subset of $Y_i$. By the $\bc$-melys condition, $\varphi^{-1}(Z) \cap Y_i \subset Y_i^w$. Hence $Y_i^w \neq \emptyset$. Since $Y_i$ is a connected component of $Y$, either $w(Y_i) = Y_i$ or $w(Y_i) \cap Y_i = \emptyset$. Thus, we must have $w(Y_i) = Y_i$. Moreover, since we are not in (I), $\varphi^*(f_Z) |_{Y_i} \neq 0$. Then $w (\varphi^*(f_Z) |_{Y_i}) = \lambda_{w,Z}( \varphi^*(f_Z) |_{Y_i})$ implies that $w$ acts non-trivially on $\mc{O}(Y_i)$. In other words, $Y_i^w$ is a proper closed subset of $Y_i$. On the other hand, $\varphi^{-1}(Z) \cap Y_i$ is a closed subset of $Y_i$ of codimension at most one since $Z$ is a hypersurface. This forces us to be in situation (III). 
	
	As explained in \cite[Section~2.2]{BBJantzen} (see also \cite[Lemma~A.2.2.]{ProjAffine}), there exists an action map $\gamma \colon \Gamma(Y,\mc{P}_{\varphi^* \omega}) \to \Gamma(Y,\varphi^* \mc{P}_{\omega})$ making $\varphi^* \ms{N}$ into a $\dd_{\varphi^* \omega}(Y)$-module for every $\dd_{\omega}(X)$-module $\ms{N}$. The morphism $\gamma$ is $W$-equivariant. In particular, if $\Pa$ is the pointwise stabilizer of $Y_i$ in $W$, then by averaging over $\Pa$ we may assume that $\gamma(v) = \sum_j g_j \o u_j$, with $u_j \in \Gamma(X,\mc{P}_{\omega})^{\Pa}$, for all $v\in \Gamma(Y_i,\mc{P}_{\varphi^* \omega})$. For any $u \in \Gamma(X,\mc{P}_{\omega})^{\Pa}$ and $(w,Z)$ of type (I), we have $w(u) = u$ since $Y_i^w = Y_i$ (melys condition) and hence $w \in \Pa$. Then $w(\sigma(u)(f_Z)) = \lambda_{w,Z} \sigma(u)(f_Z)$ implying that $\sigma(u)(f_Z) \in I(Z)$. This means that $\sigma(u)(f_Z)f_Z^{-1} \in \mc{O}(X)$. If $(w,Z)$ is of type (II) then $\varphi^*(f_Z) |_{Y_i}$ is an invertible function and hence $1 \o \sigma(u)(f_Z)f_Z^{-1} \in \mc{O}(Y_i) \o_{\mc{O}(X)} \mc{O}(X)$. We deduce that if
	\begin{equation*}
	h(u) := \sum_{\stackrel{(w,Z) \in \mc{S}_{\bc}(X)}{\textrm{in (I) or (II)}}} \frac{2 \bc(w,Z)}{1 - \lambda_{w,Z}} \frac{\sigma(u)(f_Z)}{f_Z}(w - 1)
	\end{equation*}
	then $1 \o h(u) \in \mc{O}(Y_i) \o_{\mc{O}(X)} (\mc{O}(X) \o \kk W)$. 
	
	Finally, we note that if $(w,Z)$ is of type (III) then $\varphi^*(f_Z)$ is invertible on the (principal) open set $Y_i^{\circ}$ of $Y_i$. Therefore, if $\gamma(v) = \sum_j g_j \o u_j$ with $u_j \in \Gamma(X,\mc{P}_{\omega})^{\Pa}$ then 
	\begin{equation}\label{eq:Pvarphiomega}
		\sum_j g_j \o \left( D_{u_j}- h(u_j) -  \sum_{\stackrel{(w,Z) \in \mc{S}_{\bc}(X)}{\textrm{in (III)}}} \frac{2 \bc(w,Z)}{1 - \lambda_{w,Z}} \frac{\sigma(u_j)(f_Z)}{f_Z}(w - 1)  \right) = \sum_j g_j \o u_j 
	\end{equation}
	belongs to $\mc{O}(Y_i^{\circ}) \o_{\mc{O}(X)} \H_{\omega,\bc}(X,W)$ and we define a left action of the algebra $\dd_{\varphi^* \omega}(Y_i^{\circ})$ on $\mc{O}(Y_i^{\circ}) \o_{\mc{O}(X)} \H_{\omega,\bc}(X,W)$ via the expression on the left of \eqref{eq:Pvarphiomega}. The action map
	\[
	\dd_{\varphi^* \omega}(Y^{\circ}) \times \left( \mc{O}(Y^{\circ}) \o_{\mc{O}(X)} \H_{\omega,\bc}(X,W) \right) \to \mc{O}(Y^{\circ}) \o_{\mc{O}(X)} \H_{\omega,\bc}(X,W)
	\] 
	is $W$-equivariant and hence the action of $\dd_{\varphi^* \omega}(Y^{\circ})$ on $\mc{O}(Y^{\circ}) \o_{\mc{O}(X)} \H_{\omega,\bc}(X,W)$ extends to an action of $\dd_{\varphi^* \omega}(Y^{\circ}) \rtimes W$ on $\mc{O}(Y^{\circ}) \o_{\mc{O}(X)} \H_{\omega,\bc}(X,W)$. 
	
	To show that this action globalizes (dropping the local assumptions on $X$ and $Y$), it suffices to note that the map $\gamma$ is a map of sheaves $\mc{P}_{\varphi_* \omega} \to \varphi^* \mc{P}_{\omega}$. Therefore, the expressions on the right hand side of \eqref{eq:Pvarphiomega} agree on overlaps. This implies that the expressions on the left hand side of \eqref{eq:Pvarphiomega} agree on overlaps (even though the choice of $f_Z$ depends on the local charts).  
	
	Finally, we check that this construction gives an action of $\sH_{\varphi^* \omega,\varphi^* \bc}(Y,W)$	on $\varphi^* \sH_{\omega,\bc}(X,W)$. Since $ \sH_{\omega,\bc}(X,W)$ is flat over $\mc{O}_X$, the inclusion $\mc{O}_Y \hookrightarrow \mc{O}_{Y^{\circ}}$ tensors up to an inclusion $\varphi^* \sH_{\omega,\bc}(X,W) \hookrightarrow \mc{O}_{Y^{\circ}} \o_{\mc{O}_X} \sH_{\omega,\bc}(X,W)$. The right hand side is a $\sH_{\varphi^* \omega,\varphi^* \bc}(Y,W)$-module since this algebra is a subalgebra of $\dd_{\varphi^* \omega}(Y^{\circ}) \rtimes W$. Therefore, we just need to check locally that $\varphi^* \sH_{\omega,\bc}(X,W)$ is a $\sH_{\varphi^* \omega,\varphi^* \bc}(Y,W)$-submodule. 
	
As above, if $f \in \mc{O}(Y), r \in \H_{\omega,\bc}(X,W)$ and $v \in \Gamma(Y_i,\mc{P}_{\varphi_* \omega})$, then, applying equation \eqref{eq:Pvarphiomega} and using the key definition \eqref{eq:bcpullback}, we have
	\begin{align*}
		D_{v} \cdot(f \o r) & = [D_{v},f] \o r + \sum_j f g_j \o u_j r + \sum_{(w,Z) \in \mc{S}_{\varphi^* \bc}(Y)} f \frac{2 (\varphi^* \bc)(w,Z)}{1 - \lambda_{w,Z}} \frac{\sigma(v)(f_Z)}{f_Z}(w - 1) \o r \\
		& = [D_{v},f] \o r + \sum_j f g_j \o (D_{u_j} - h(u_j)) r \in \mc{O}(Y) \o_{\mc{O}(X)} \H_{\omega,\bc}(X,W),
	\end{align*}
	since $\varphi^* \bc$ is defined so that the equality
	\[
	\sum_{\stackrel{(w,Z) \in \mc{S}_{\varphi^* \bc}(Y)}{Z \subset Y_i}} \frac{2 (\varphi^* \bc)(w,Z)}{1 - \lambda_{w,Z}} (d \log f_Z) (w - 1) = \sum_{\textrm{(III)}} \frac{2 \bc(w,Z')}{1 - \lambda_{w,Z'}} \varphi^*(d \log f_{Z'}) (w - 1) 
	\]
holds in $\Omega_{Y^{\circ}_i}^1 \o \kk W$. 
\end{proof}

\begin{remark}\label{rem:nopullbackmelys}
	One cannot expect a $\sH_{\varphi^* \omega,\bc'}(Y,W)$-module structure on the pullback $\varphi^* \ms{M}$ of a $\sH_{\omega,\bc}(X,W)$-module $\ms{M}$ (for some $\bc'$) if we drop the melys condition. For instance, consider the case where $X = \h$ and $W$ acts as a complex reflection group. If we take $Y = \h \times \h$ with diagonal $W$-action and $\varphi$ projection onto the second factor then $\mc{S}(Y)$ is empty and so $\H_{\bc'}(Y,W) = \dd(Y) \rtimes W$. But we can choose $\bc$ such that $\H_{\bc}(X,W)$ has modules of GK-dimension less than $\dim \h$. This would give rise to (non-zero) modules over $\dd(Y) \rtimes W$ of GK-dimension less than $2 \dim \h$, a contradiction. 
\end{remark}

Each reflection hypersurface $Z \subset X$ is a smooth divisor. We denote by $\omega_Z$ the class of the associated line bundle $\mc{O}_X(Z)$ in $\mathbb{H}^2(X,\Omega^{1,2}_X)$. The class of the canonical divisor $K_X$ in $\mathbb{H}^2(X,\Omega^{1,2}_X)$ is denoted $\omega^{\mr{can}}$. For a parameter $\bc \in \mc{S}(X)^W$, define $\overline{\bc}(w,Z) := \bc(w^{-1},Z)$. Then it was shown in \cite[Lemma~4.1]{ThompsonHolI} that there is an isomorphism of sheaves of algebras
\[
\sH_{\omega,\bc}(X,W)^{\op} \cong \sH_{\overline{\omega},\overline{\bc}}(X,W), 
\] 
where\footnote{In the main isomorphism of \cite[Lemma~4.1]{ThompsonHolI}, the expression $\omega^{\mr{can}} - \nu + \sum 2 \bc(w,Z) \omega_Z$ should be corrected to $\omega^{\mr{can}} - \nu + \sum 2 \overline{\bc}(w,Z) \omega_Z$.}
\begin{equation}\label{eq:omegabarrightleft}
	\overline{\omega} := \omega^{\mr{can}} - \omega + \sum_{(w,Z) \in \mc{S}(X)} 2 \overline{\bc}(w,Z) \omega_Z. 
\end{equation}
As explained in \cite[Section~4.2]{ThompsonHolI}, tensoring by the canonical bundle $K_X$ gives an equivalence between left $\sH_{\omega,\bc}(X,W)$-modules and right $\sH_{\overline{\omega},\overline{\bc}}(X,W)$-modules. Therefore, just as for $\dd$-modules, one can associate to any $\bc$-melys morphism the transfer bimodules
$$
\sH_{Y \to X} := \varphi^* \sH_{\omega,\bc}(X,W) \quad \textrm{over} \quad (\sH_{\varphi^*\omega,\varphi^*\bc}(Y,W),\varphi^{-1} \sH_{\omega,\bc}(X,W)),
$$
and
$$
\sH_{X \leftarrow Y} := K_Y \o_{\mc{O}_Y} \varphi^* \sH_{\overline{\omega},\overline{\bc}}(X,W) \o_{\varphi^{-1}\mc{O}_X} \varphi^{-1} K_X^{-1} 
$$
over $(\varphi^{-1} \sH_{\omega,\bc}(X,W),\sH_{\varphi^*\omega,\varphi^*\bc}(Y,W))$.

\begin{remark}\label{rem:melysop}
	If $\varphi \colon Y \to X$ is $\bc$-melys then one can check that $\overline{\varphi^* \omega } = \varphi^* \overline{\omega}$ and $\overline{\varphi^* \bc} = \varphi^* \overline{\bc}$. 
\end{remark}

 We define the direct and inverse image functors $\varphi_0 \colon \Qcoh{\sH_{\varphi^*\omega,\varphi^*\bc}(Y,W)} \to \Qcoh{\sH_{\omega,\bc}(X,W)}$ and $\varphi^0 \colon \Qcoh{\sH_{\omega,\bc}(X,W)} \to \Qcoh{\sH_{\varphi^*\omega,\varphi^*\bc}(Y,W)}$ by 
 \begin{align*}
\varphi_0(\ms{M}) & := \varphi_{\idot}(\sH_{X \leftarrow Y} \o_{\sH_{\varphi^*\omega,\varphi^*\bc}(Y,W)} \ms{M}) \\
\varphi^0 (\ms{N}) & := \sH_{Y \to X} \o_{\varphi^{-1} \sH_{\omega,\bc}(X,W)} \varphi^{-1} \ms{N} \\
& \ (\cong \varphi^{*} \ms{N} \textrm{as an $\mc{O}_Y$-module.}) & 
\end{align*}

\begin{example}\label{ex:invariantfunctionmelys}
	Let $X,Y,Z$ be smooth $W$-spaces with $W$ acting trivially on $Z$. If  $g \colon X \to Y$ is $\bc$-melys and $f \colon X \to Z$ any equivariant morphism then $g \times f \colon X \to Y \times Z$ is $\bc$-melys. In particular, taking $X = Y$ and $g = \mr{Id}$ gives a $\bc$-melys morphism $X \to X \times Z$. This gives an action of the monoidal category $\Coh{\dd_Z}$ on $\Coh{\sH_{\omega,\bc}(X,W)}$; the module $\ms{A} \in \Coh{\dd_Z}$ acts by $\ms{M} \mapsto (\mr{Id} \times f)^0(\ms{M} \boxtimes \ms{A})$.    
\end{example}

Since we work over smooth quasi-compact and separated schemes, \cite[Proposition~2.2.7]{Globallocallyfreeres} says that every quasi-coherent $\mc{O}_X$-module is a quotient of a locally free $\mc{O}_X$-module, where we may choose the locally free module to be of finite rank if the $\mc{O}_X$-module is coherent. This means that many of the results on $\dd$-modules apply verbatim in our case without the usual assumption (\cite[Assumption~1.4.19]{HTT} or \cite[VI, Section~4]{BorelDmod}) that $X$ be quasi-projective. In particular, any (coherent) $\sH_{\omega,\bc}(X,W)$-module admits a resolution by (finite rank) locally free $\sH_{\omega,\bc}(X,W)$-modules and a resolution by (coherent) locally projective $\sH_{\omega,\bc}(X,W)$-modules of length at most $2 \dim X$; compare with \cite[Corollary~1.4.20]{HTT}. 

Therefore we may extend the functors $\varphi^0, \varphi_0$ to various derived categories. Assume that $X,Y$ are $W$-connected (so equidimensional). Let $ D^b_{\mathrm{qc}}(\sH_{\omega,\bc}(X,W))$ denote the bounded derived category of  $\sH_{\omega,\bc}(X,W)$-modules with quasi-coherent cohomology. Taking finite locally projective resolutions, we define $\varphi_+ \colon D^b_{\mathrm{qc}}(\sH_{\varphi^*\omega,\varphi^*\bc}(Y,W)) \to D^b_{\mathrm{qc}}(\sH_{\omega,\bc}(X,W))$ and $\varphi^! \colon D^b_{\mathrm{qc}}(\sH_{\omega,\bc}(X,W)) \to D^b_{\mathrm{qc}}(\sH_{\varphi^*\omega,\varphi^*\bc}(Y,W))$ by 
$$
\varphi_+(\ms{M}) = R \varphi_{\idot}(\sH_{X \leftarrow Y} \o_{\sH_{\varphi^*\omega,\varphi^*\bc}(Y,W)}^{{L}} \ms{M})
$$
and
$$
\varphi^!(\ms{N}) = \sH_{Y \to X} \o_{\varphi^{-1} \sH_{\omega,\bc}(X,W)}^{{L}} \varphi^{-1} \ms{N} [\dim Y - \dim X].
$$
so that $\varphi^!(\ms{N}) = {L} \varphi^{0} (\ms{N})[\dim Y - \dim X]$.

\begin{remark}
	A priori it is not clear that $\varphi_0$ preserves quasi-coherence (resp. $\varphi_+$ preserves complexes with quasi-coherent cohomology). However, the usual proof in the case of $\dd$-modules, by using a spectral sequence to reduce to the case where $X$ and $Y$ are affine, works verbatim for sheaves of Cherednik algebras. See \cite[VI, Section 5.1]{BorelDmod} for a detailed proof in the case of $\dd$-modules. 
\end{remark}

\begin{lem}
	If $\varphi \colon Y \to X$ is strongly $\bc$-melys and $\psi \colon U \to Y$ is $\varphi^*\bc$-melys then $\varphi \circ \psi$ is $\bc$-melys. If, in addition, $\psi$ is strongly $\varphi^* \bc$-melys then $\varphi \circ \psi$ is strongly $\bc$-melys. 
\end{lem}

\begin{proof}
	Being strongly melys can be stated as saying that for each $(w,Z) \in \mc{S}_{\bc}(X)$, the pre-image $\varphi^{-1}(Z)$ is a (disjoint) union of reflection hypersurfaces $\bigcup Z'$ with $(w,Z') \in \mc{S}(Y)$. By definition, we have $(w,Z') \in \mc{S}_{\varphi^* \bc}(Y)$ for all $Z'$ occurring in this union. The two statements of the lemma follow immediately from this fact since 
	\[
	(\varphi \circ \psi)^{-1}(Z) = \bigcup \psi^{-1}(Z').
	\]
\end{proof}

Whenever the composition of two melys morphisms $\varphi,\psi$ is again melys, we have $(\varphi \circ \psi)^! \cong \psi^! \circ \varphi^!$ and $(\varphi \circ \psi)_+ \cong \varphi_+ \circ \psi_+$. The proof of these statements is identical to the $\dd$-module case (see \cite[Proposition~1.5.11, Proposition~1.5.21]{HTT}), the key fact being that multiplication is an isomorphism 
\[
\sH_{Z \to X} \cong \sH_{Z \to Y} \o^{L}_{\psi^{-1}  \sH_{\varphi* \omega, \varphi^* \bc}(Y,W)} \psi^{-1} \sH_{Y \to X}.
\]

Let $W(\bc)$ be the subgroup of $W$ generated by all $w$ such that $(w,Z) \in \mc{S}_{\bc}(X)$. Then $\sH_{\omega,\bc}(X,W(\bc))$ is a subsheaf of $\sH_{\omega,\bc}(X,W)$, the latter being free of rank $|W/W(\bc)|$ over the former. We may consider the forgetful functor 
\[
\mr{For}_X \colon \Qcoh{\sH_{\omega,\bc}(X,W)} \to \Qcoh{\sH_{\omega,\bc}(X,W(\bc))}.
\]
Note that objects of $\Qcoh{\sH_{\omega,\bc}(X,W)}$ are sheaves on $X/W$ whilst those of $\Qcoh{\sH_{\omega,\bc}(X,W(\bc))}$ are sheaves on $X/W(\bc)$, so $\mr{For}_X(\ms{M}) = \eta^{-1} \ms{M}$ as sheaves, where $\eta \colon X / W(\bc) \to X/W$ is the quotient. The following is immediate. 

\begin{lem}\label{lem:forgetcommutepushpull}
	We have commutative diagrams
	$$
	\begin{tikzcd}
	 \Qcoh{\sH_{\varphi^*\omega,\varphi^*\bc}(Y,W)} \ar[r,"\varphi_0"] \ar[d,"\mr{For}_Y"'] &  \Qcoh{\sH_{\omega,\bc}(X,W)} \ar[d,"\mr{For}_X"] \\
	 \Qcoh{\sH_{\varphi^*\omega,\varphi^*\bc}(Y,W(\bc))} \ar[r,"\varphi_0"] & \Qcoh{\sH_{\omega,\bc}(X,W(\bc))},
	\end{tikzcd}
	$$
	$$
	\begin{tikzcd}
	 \Qcoh{\sH_{\varphi^*\omega,\varphi^*\bc}(Y,W)} \ar[d,"\mr{For}_Y"'] &  \Qcoh{\sH_{\omega,\bc}(X,W)} \ar[l,"\varphi^0"] \ar[d,"\mr{For}_X"] \\
	 \Qcoh{\sH_{\varphi^*\omega,\varphi^*\bc}(Y,W(\bc))} & \Qcoh{\sH_{\omega,\bc}(X,W(\bc))} \ar[l,"\varphi^0"],
	\end{tikzcd}
	$$
	and $\SS_X(\ms{M}) = \SS_X(\mr{For}_X(\ms{M}))$ for any coherent $\sH_{\omega,\bc}(X,W)$-module $\ms{M}$. 
	
	The analogous commutative diagrams exist when considering the derived functors $\varphi_+, \varphi^!$.  
\end{lem}

\subsection{Completions}

The completion of the local ring $\mc{O}_{X,x}$ at its maximal ideal is denoted $\widehat{\mc{O}}_{X,x}$. Let $\widehat{X}_x = \mathrm{Spec}\, \widehat{\mc{O}}_{X,x}$. For each $x \in X$, with stabilizer $W_x$, let $\bc_x$ denote the restriction of $\bc$ to the set of pairs $(w,Z)$ with $x \in Z$.

\begin{lem}\label{lem:completelocalCA}
	If $\ms{M}$ is a $\sH_{\omega,\bc}(X,W)$-module then $\widehat{\mc{O}}_{X,x} \o_{\mc{O}_X} \ms{M}$ is a $\sH_{\bc_x}(\widehat{X}_x,W_x)$-module.
\end{lem}

See \cite{ChereSheaf} for the precise definition of $\sH_{\bc_x}(\widehat{X}_x,W_x)$. 

\begin{proof}
	Since Dunkl operators are $\mathrm{ad}(\mc{O}_{X/W})$-nilpotent, the multiplication on $\sH_{\omega,\bc}(X,W)$ extends to $\widehat{\mc{O}}_{X/W,\overline{x}} \o_{\mc{O}_{X/W}} \sH_{\omega,\bc}(X,W)$ and the latter algebra is Morita equivalent to the algebra $\widehat{\mc{O}}_{X,x} \o_{\mc{O}_X} \sH_{\omega,\bc}(X,W_x) = \sH_{\bc_x}(\widehat{X}_x,W_x)$ in such a way that if $\ms{M}$ is a $\sH_{\omega,\bc}(X,W)$-module then the space $\widehat{\mc{O}}_{X,x} \o_{\mc{O}_X} \ms{M}$ is a $\sH_{\bc_x}(\widehat{X}_x,W_x)$-module; see \cite{BE} and \cite{ChereSheaf}. 
\end{proof}	
	
If $\ms{M}$ is a coherent $\sH_{\omega,\bc}(X,W)$-module then each filtered piece in a good filtration on $\ms{M}$ is coherent over $\mc{O}_X$. This implies that 
\begin{equation}\label{eq:completesingularsupport}
\SS_{\widehat{X}_{x}}(\widehat{\mc{O}}_{X,x} \o_{\mc{O}_X} \ms{M}) = \widehat{X}_{x} \underset{X}{\times} \SS_X(\ms{M}),
\end{equation}
where $\widehat{\mc{O}}_{X,x} \o_{\mc{O}_X} \ms{M}$ is a module over $\sH_{\bc_x}(\widehat{X}_x,W_x)$.  

The $\bc$-melys map $\varphi \colon Y \to X$ restricts to a $\bc_{\varphi(y)}$-melys morphism $\widehat{\varphi} \colon \widehat{Y}_{y} \to \widehat{X}_{\varphi(y)}$ on the formal neighbourhood of $y$ in $Y$. Let $\bc_y = \widehat{\varphi}^* \bc_{\varphi(y)}$. 

\begin{lem}\label{lem:completepullbackbi}
Let $\varphi \colon Y \to X$ be a $\bc$-melys morphism and choose $y \in Y^W, x = \varphi(y)$.  
\begin{enumerate}
	\item[(i)] As $\sH_{\bc_y}(\widehat{Y}_y,W)$-$\sH_{\bc_x}(\widehat{X}_x,W)$-bimodules,
	$$
	\widehat{\mc{O}}_{Y,y} \o_{\mc{O}_Y} \sH_{Y \to X} \cong \sH_{\widehat{T_y Y} \to \widehat{T_x X}}
	$$
	and similarly for $\sH_{X \leftarrow Y}$.
	\item[(ii)] If $\widehat{\mc{O}}_{Y,y} \cong \mc{O}_{Y,y} \o_{\mc{O}_{X,x}} \widehat{\mc{O}}_{X,x}$ then in addition,
	$$
	\widehat{\mc{O}}_{Y,y} \o_{\mc{O}_Y} \sH_{Y \to X} \cong \sH_{Y \to X} \o_{\varphi^{-1} \mc{O}_X} \widehat{\mc{O}}_{X,x},
	$$
 and similarly for $\sH_{X \leftarrow Y}$. 
\end{enumerate} 
\end{lem} 
	
\begin{proof}
	Part (i). The statements are local so we assume $X,Y$ affine. Moreover, since Picard algebroids are trivial on formal neighbourhoods, we may also assume that $\omega = 0$. Note that the ring homomorphism $\mc{O}(X) \to \widehat{\mc{O}}_{Y,y}$ factors through $\widehat{\mc{O}}_{X,x}$. In particular, $\widehat{\mc{O}}_{Y,y}$ is a $\widehat{\mc{O}}_{X,x}$-module. This implies the second isomorphism in the chain of isomorphisms: 
	\begin{align*}
			\widehat{\mc{O}}_{Y,y} \o_{\mc{O}_Y} \sH_{Y \to X} & = \widehat{\mc{O}}_{Y,y} \o_{\mc{O}(Y)} (\mc{O}(Y) \o_{\mc{O}(X)} \H_{\bc}(X,W)) \\
			& \cong \widehat{\mc{O}}_{Y,y} \o_{\mc{O}(X)} \H_{\bc}(X,W) \\
			& \cong (\widehat{\mc{O}}_{Y,y} \o_{\widehat{\mc{O}}_{X,x}} \widehat{\mc{O}}_{X,x}) \o_{\widehat{\mc{O}}_{X,x}} (\widehat{\mc{O}}_{X,x}  \o_{\mc{O}(X)} \H_{\bc}(X,W)) \\
			& \cong (\widehat{\mc{O}}_{Y,y} \o_{\widehat{\mc{O}}_{X,x}} \widehat{\mc{O}}_{X,x}) \o_{\widehat{\mc{O}}_{X,x}} \H_{\bc}(\widehat{X}_x,W) \\
			& \cong \widehat{\mc{O}}_{Y,y} \o_{\widehat{\mc{O}}_{X,x}} \H_{\bc}(\widehat{X}_x,W) = \sH_{\widehat{T_y Y} \to \widehat{T_x X}}.
	\end{align*}
	The proof for $\sH_{X \leftarrow Y}$ and part (ii) are similar.
\end{proof}

\'Etale morphisms that are $\bc$-melys are especially well-behaved, as the lemma below shows. In particular, an \'etale $\bc$-melys morphism is automatically strongly $\bc$-melys. 

\begin{lem}\label{lem:etalebimodules}
	Let $\varphi \colon Y \to X$ be a $\bc$-melys \'etale morphism between smooth $W$-varieties. Then there exists an embedding $\varphi^{-1} \sH_{\omega,\bc}(X,W) \hookrightarrow \sH_{\varphi^* \omega,\varphi^* \bc}(Y,W)$ such that 
	\begin{equation}\label{eq:etalebi1}
			\sH_{Y\to X} \cong \sH_{\varphi^* \omega,\varphi^* \bc}(Y,W)
	\end{equation}
	as $(\sH_{\varphi^* \omega,\varphi^* \bc}(Y,W),\varphi^{-1} \sH_{\omega,\bc}(X,W))$-bimodules and 
\begin{equation}\label{eq:etalebi2}
		\sH_{X \leftarrow Y} \cong \sH_{\varphi^* \omega,\varphi^* \bc}(Y,W)
\end{equation}
	as $(\varphi^{-1} \sH_{\omega,\bc}(X,W), \sH_{\varphi^* \omega,\varphi^* \bc}(Y,W))$-bimodules.  
\end{lem}

\begin{proof}
	In the case of $\dd$-modules, these statements are well-known; see \cite[Theorem~2.2]{CoutinhoEtale} for a detailed proof. In particular, in the notation of the proof of Theorem~\ref{thm:melyspullbackgeneral}, they hold for $\dd$-modules relative to $\varphi \colon Y^{\circ} \to X^{\circ}$. It follows from the proof of Theorem~\ref{thm:melyspullbackgeneral} that the isomorphism $\dd_{Y^{\circ}\to X^{\circ}} \cong \dd_{\varphi^* \omega}(Y^{\circ}) \otimes \kk W$ of $(\dd_{\varphi^* \omega}(Y^{\circ}) \rtimes W,\varphi^{-1} \dd_{\omega}(X^{\circ}) \rtimes W)$-bimodules restricts to \eqref{eq:etalebi1} and similarly for \eqref{eq:etalebi2}. 
\end{proof}

\subsection{Reduction to trivial stabilizer}\label{sec:redtrivstab}

Let $N \subset W$ be a subgroup and $U_0 \subset X$ an open subset such that
\begin{equation}\label{eq:Eetale1}
	\textrm{$n(U_0) \subset U_0$ for all $n \in N$; and}
\end{equation}
\begin{equation}\label{eq:Eetale2}
	\textrm{$W_x \subset N$ for all $x \in U_0$.}
\end{equation}
If $\mr{Fun}_N(W,\sH_{\omega,\bc}(U_0,N))$ denotes the sheaf of $N$-equivariant functions $W \to \sH_{\omega,\bc}(U_0,N)$, then following \cite[Section~3.2]{BE} we write  
$$
Z(W,N,\sH_{\omega,\bc}(U_0,N)) := \End_{\sH_{\omega,\bc}(U_0,N)}(\mr{Fun}_N(W,\sH_{\omega,\bc}(U_0,N))).
$$
This is a matrix algebra over $\sH_{\omega,\bc}(U_0,N)$. Hence there is an equivalence
\begin{equation}\label{eq:matrixequivZ}
	\Qcoh{Z(W,N,\sH_{\omega,\bc}(U_0,N))} \cong \Qcoh{\sH_{\omega,\bc}(U_0,N)}.
\end{equation}
The smooth, but disconnected, variety $V := W \times_N U_0$ carries an action of $W$. Let $\varphi \colon V \to X$ be the map $\varphi(w,u) = w(u)$. 

\begin{prop}\label{prop:etalemelysiso}
	$\varphi$ is \'etale and strongly $\bc$-melys. Moreover, there is an embedding   
	\begin{equation}\label{eq:etalefactoriso}
		\varphi^{-1} \sH_{\omega,\bc}(X,W) \hookrightarrow \sH_{\varphi^*\omega,\varphi^*\bc}(V,W) \cong Z(W,N,\sH_{\varphi^*\omega,\varphi^*\bc}(U_0,N))
	\end{equation}
	of sheaves of algebras on $V / W = U_0/N$. 
\end{prop}

\begin{proof}
	We fix left coset representatives $w_0, \ds, w_k$ of $N$ in $W$ and let $U_i = w_i(U_0)$ so that $V = \bigsqcup_{i= 0}^k U_i$ and $U := \bigcup_{i = 0}^k U_i$. The fact that $\varphi$ is \'etale is immediate from the fact that it is an open embedding on each connected component $U_i$ of $V$. If $(w,Z) \in \mc{S}(X)$ and $U_i \cap Z \neq \emptyset$ then $U_i^{w} \neq U_i$ and $U_i \cap Z$ is a reflection hypersurface in $U_i$; in fact, $U_i \cap Z \neq \emptyset$ only if $w \in w_i N w_i^{-1}$. Therefore, if $Z \cap U \neq \emptyset$ then $\varphi^{-1}(Z) = \bigsqcup_i (U_i \cap Z)$ is a disjoint union (with multiplicity one) of reflection hypersurfaces in $V$. Thus, $\varphi$ is strongly $\bc$-melys. An element of $\mc{S}(V,W)$ is of the form $(w,Z_i)$ where $w \in W$ and $Z_i$ is a connected component of $U_i^w$ of codimension one. Then $(\varphi^*\bc)(w,Z_i) = \bc(w,\overline{Z}_i)$, the closure taken in $X$. 
	
The embedding of \eqref{eq:etalefactoriso} follows from Lemma~\ref{lem:etalebimodules} and the final isomorphism is well-known; see \cite{BE}, \cite[Section 2.3]{LosevHecke} and \cite{Wilcox}. 
\end{proof}

\begin{remark}\label{rem:Lunaopen}
	Typically, one chooses $x \in X$ with $\Pa = W_x$, $N = N_W(\Pa)$ and $U_0$ some $N$-stable open neighbourhood of $x$ with $W_{u} \subset \Pa$ for all $u \in U_0$. It is a consequence of Luna's slice theorem \cite{Luna} that $U_0$ can be chosen so that the image $U$ of $\varphi$ is affine. Then $\varphi \colon V \to U$ is surjective and \'etale, which means that it is finite and faithfully flat. 
\end{remark}

In order to use Lemma~\ref{lem:completepullbackbi} to understand the behaviour of singular support under various functors, it is necessary to reduce to a neighbourhood of a point $x \in X^W$. This is accomplished by the following application of Proposition~\ref{prop:etalemelysiso}. 

\begin{lem}\label{lem:localSSfixed}
	Let $x \in X$. There exists a $W_x$-stable affine neighbourhood $U$ of $x$ such that $\ms{M} |_U$ is a $\sH_{\omega,\bc}(U,W_x)$-module for any $\sH_{\omega,\bc}(X,W)$-module $\ms{M}$. Moreover, if $\ms{M}$ is coherent then  
	$$
	\SS_U(\ms{M} |_U) = \SS_X(\ms{M}) |_U.
	$$ 
\end{lem}

\begin{proof}
	We may assume $X$ affine and $W$-connected. Let $U$ be an affine neighbourhood of $x$ such that the stabilizer of every point in $U$ is contained in $W_x$. Shrinking $U$ if necessary, we assume it is $W_x$-stable. Let $V= W \times_{W_x} U$ and $\varphi \colon V \to X$ the natural map. Then we are in the situation of Proposition~\ref{prop:etalemelysiso}, implying that $\varphi$ is $\bc$-melys and \'etale and 
	\[
	\H_{\omega,\bc}(X,W) \hookrightarrow \mc{O}(V) \o_{\mc{O}(X)} \H_{\omega,\bc}(X,W) \cong Z(W,W_x,\H_{\omega,\bc}(U,W_x)).
	\]
If $e_U \in Z(W,W_x,\H_{\omega,\bc}(U,W_x))$ is the idempotent that satisfies $e_U(f)(w) = f(w)$ for $w \in W_x$ and $e_U(f)(w) = 0$ otherwise then $e_U \cdot (\varphi^0 \ms{M}) \cong \ms{M}|_U$ as a $\mc{O}_U$-module, and it carries an action of 
	\[
	e_U Z(W,W_x,\H_{\omega,\bc}(U,W_x)) e_U \cong \H_{\omega,\bc}(U,W_x). 
	\]
If $\ms{M}$ is coherent and we fix a good filtration $\mc{F}_{\idot}$ on $\ms{M}$ then $\varphi^* \mc{F}_{\idot}$ is a good filtration on $\varphi^0 \ms{M}$. If $\mc{G}_{\idot} := e_U \cdot  \varphi^* \mc{F}_{\idot}$ is the associated filtration on $e_U \cdot (\varphi^0 \ms{M})$ then  
\[
\gr_{\mc{G}} (e_U \cdot (\varphi^0 \ms{M})) = e_U \cdot (\gr_{\varphi^* \mc{F}} (\varphi^0 \ms{M})) = (\gr_{\mc{F}} \ms{M}) |_U
\]
meaning that $\SS_U(\ms{M} |_U) = \SS_X(\ms{M}) |_U$.   
\end{proof} 

\subsection{Reduction to the formal neighbourhood of a point.}

We note applications of reduction to the formal neighbourhood of a point. We say that the sheaf $\sH_{\omega,\bc}(X,W)$ is \textit{simple} if there does not exist any non-zero sheaf of ideals $\mc{J} \lhd \sH_{\omega,\bc}(X,W)$. Note that this does \textit{not} imply that $\Gamma(X,\sH_{\omega,\bc}(X,W))$ is a simple ring if $X$ is not affine.

\begin{prop}\label{prop:simplealgebraH}
Assume that $X$ is irreducible and $W$ acts effectively on $X$. The set of parameters $(\omega,\bc) \in \mathbb{H}^2(X,\Omega^{1,2}_X)^W \oplus \mc{S}(X)^W$ where the sheaf of algebras $\sH_{\omega,\bc}(X,W)$ is not simple is contained in a countable union of affine hyperplanes. 
\end{prop}

\begin{proof}
For each $x \in X$, consider the set $\mc{S}_x = \{ (w,Z) \in \mc{S}(X) \, | \, x \in Z \}$. Since $\mc{S}(X)$ is a finite set, the equivalence relation $x \sim x'$ iff $\mc{S}_x = \mc{S}_{x'}$ defines a finite partition $X = \bigsqcup_{i \in I} X_i$ of $X$ into locally closed subsets, where $X_i = \{ x \in X \, | \, \mc{S}_{x} = \mc{S}_i \}$. Restriction of parameters defines an injective linear map $p \colon \mc{S}(X) \to \bigoplus_{i \in I} \mc{S}_i$. Let $W_i = \langle w \, | \, (w,Z) \in \mc{S}_i \textrm{ for some $Z$} \rangle$, so that $W_i \subset W_x$ for all $x \in X_i$. Then $p$ restricts to an injection $\mc{S}(X)^W \to \bigoplus_{i \in I} \mc{S}_i^{W_i}$. Composing with the projection $\mathbb{H}^2(X,\Omega^{1,2}_X)^W \oplus \mc{S}(X)^W \to \mc{S}(X)^W$ gives a linear map 
	\[
	p \colon \mathbb{H}^2(X,\Omega^{1,2}_X)^W \oplus \mc{S}(X)^W \to \bigoplus_{i \in I} \mc{S}_i^{W_i}. 
	\]
    Our assumptions ensure $T_x X$ is a faithful $W_x$-representation for each $x \in X$. If we choose $x_i \in X_i$ and set $\mf{h}_i = T_{x_i} X$ then we let $C_i \subset \mc{S}_i^{W_i}$ denote the set of parameters $\bc$ such that $\H_{\bc_{x_i}}(\h_i,W_i)$ is not a simple ring. The result \cite[Theorem~1.4.2]{LosevSRAComplete} says that $C_i$ is contained in a countable union of affine hyperplanes, not containing $0$. 
	
	If $\mc{J} \lhd \sH_{\omega,\bc}(X,W)$ is a proper ideal then we choose an $i$ and $x_i \in X_i$ such that $x_i$ is contained in the support of $\sH_{\omega,\bc}(X,W) / \mc{J}$. Then Lemma~\ref{lem:completelocalCA} implies that $\H_{\bc_{x_i}}(\widehat{\h}_i,W_{x_i})$ is not simple. Lemma~\ref{lem:simpleiffcompletesimple}(ii) then says that $\H_{\bc_{x_i}}({\h}_i,W_{x_i})$ is not simple and hence, by Lemma~\ref{lem:simpleiffcompletesimple}(iii), $\H_{\bc_{x_i}}(\h_i,W_i)$ is not a simple ring. Hence $(\omega,\bc)$ is contained in $p^{-1}(\cdots \times \mc{S}_{i-1}^{W_{i-1}} \times C_i \times \mc{S}_{i+1}^{W_{i+1}} \times \cdots)$. Thus, the set of parameters $(\omega,\bc)$ where the algebra $\sH_{\omega,\bc}(X,W)$ is not simple belongs to the set 
	\begin{equation}\label{eq:nonsimpleset}
		\bigcup_{i \in I} p^{-1}(\cdots \times \mc{S}_{i-1}^{W_{i-1}} \times C_i \times \mc{S}_{i+1}^{W_{i+1}} \times \cdots).
	\end{equation}
	Since $0 \notin C_i$ and $p$ is linear, the set \eqref{eq:nonsimpleset} is a (possibly empty) countable union of affine hyperplanes.     
\end{proof}

If $\ms{M}$ is an $\sH_{\omega,\bc}(X,W)$-module and $e$ the trivial idempotent in $\kk W$ then $e \ms{M}$ is a representation of the spherical subalgebra $e \sH_{\omega,\bc}(X,W)e$. We say that $(\omega,\bc)$ is \textit{aspherical} if there exists a non-zero sheaf $\ms{M}$ such that $e \ms{M} = 0$.

\begin{prop}\label{prop:asphericalvalueshyper}
Assume that $X$ is irreducible and $W$ acts effectively on $X$. The set of aspherical values is contained in a finite union of affine hyperplanes in $\mathbb{H}^2(X,\Omega^{1,2}_X)^W \oplus \mc{S}(X)^W$.
\end{prop}

\begin{proof}
	We continue with the setup in the proof of Proposition~\ref{prop:simplealgebraH}. Assume that $\ms{M}$ is a non-zero module with $e \ms{M} = 0$. We choose $i \in I$ and $x_i \in X_i$ such that $\ms{M}_{x_i} \neq 0$. Then Lemma~\ref{lem:completelocalCA} says that $\widehat{\mc{O}}_{X,x_i} \o_{\mc{O}_X} \ms{M}$ is a non-zero $\H_{\bc_{x_i}}(\widehat{\h}_i,W_{x_i})$-module and if $e(i) \in \kk W_{x_i}$ is the trivial idempotent then 
	\[
	e(i)(\widehat{\mc{O}}_{X,x_i} \o_{\mc{O}_X} \ms{M}) = (e \ms{M})_{x_i} = 0. 
	\]
	Hence $\bc_{x_i}$ belongs to the set of aspherical values for $\H_{\bc_{x_i}}(\widehat{\h}_i,W_{x_i})$. Since $W_{x_i}$ acts faithfully on $\h_i$, Lemma~\ref{lem:asphericalhyperplanes} and Lemma~\ref{lem:asphericalcomplete} say that the set of aspherical values for $\H_{\bc_{x_i}}(\widehat{\h}_i,W_{x_i})$ is a finite union of affine hyperplanes not containing $0$. The proposition follows just as in the proof of Proposition~\ref{prop:simplealgebraH}.
\end{proof}

\begin{remark}
In Proposition~\ref{prop:simplealgebraH}, we expect the set of parameters where the algebra $\sH_{\omega,\bc}(X,W)$ is not simple is a countable union of affine hyperplanes. Similarly, in Proposition~\ref{prop:asphericalvalueshyper}, the set of aspherical parameters is expected to be a finite union of affine hyperplanes; see Remark~\ref{rem:preciseapsherical}.  
\end{remark}

If the parameter $(\omega,\bc)$ is not aspherical then the sheaves $\sH_{\omega,\bc}(X,W)$ and $e \sH_{\omega,\bc}(X,W)e$ are Morita equivalent.

\subsection{Finite-dimensional modules}

Abusing terminology, we say that an irreducible module $\ms{M}$ is \textit{finite-dimensional} if $\Supp \ms{M} \subset X/W$ is a single point and $\dim_{\kk} \Gamma(X,\ms{M}) < \infty$. If $X$ is affine then the irreducible finite-dimensional modules are precisely the finite-dimensional irreducible $\H_{\omega,\bc}(X,W)$-modules. Recall from Section~\ref{sec:stabstrata} that $X$ has a finite stratification by stabilizer type. Our classification of irreducible holonomic modules implies: 

\begin{cor}
	For any given $(\omega,\bc)$, there are only finitely many finite-dimensional irreducible $\sH_{\omega,\bc}(X,W)$-modules up to isomorphism and they are all supported on the zero-dimensional strata of $X$.  
\end{cor}

\begin{proof}
	Let $\ms{M}$ be an irreducible finite-dimensional module. If $\Supp_X \ms{M} = W \cdot p$ then $\ms{M} = \bigoplus_{q \in W \cdot p} \ms{M}_q$, with $\ms{M}_q$ the subsheaf of sections supported at $q$. Let $\Pa$ be the stabilizer of $p$ and assume that $W \cdot p$ is contained in the stratum $\St$. Since $\ms{M}_p$ is finite-dimensional, $\widehat{\mc{O}}_{X,p} \o_{\mc{O}_X} \ms{M} = \ms{M}_p$ is a finite-dimensional $\H_{\bc_p}(\widehat{X}_p,\Pa)$-module by Lemma~\ref{lem:completelocalCA}. By restriction, it becomes a $\H_{\bc_p}(\h,\Pa)$-module, where $\h = T_p X$. We have $\dim \h^{\Pa} = \dim \St$. If $\H_{\bc_p}(\h,\Pa)$ admits a finite-dimensional module then we must have $\h^{\Pa} = 0$. Thus, $\dim \St = 0$. This means that $\St = W \cdot p$ and $\St_0 = \{ p \}$ is a parallelizable closed subset of $X$. Since any irreducible finite-dimensional module is automatically holonomic, it follows directly from Theorem~\ref{thm:mainclassificationnRCA} that there are only finitely many up to isomorphism.    
\end{proof}

\subsection{Jacquet functor}

Let $\mc{M}_{\mr{Coh}({\mc{O}})}(\sH_{\omega,\bc}(X,W))$ denote the category of all $\sH_{\omega,\bc}(X,W)$-modules that are coherent over $\mc{O}_X$. Fix $x \in X$. Lemma~\ref{lem:localSSfixed} says that there exists a neighbourhood $U$ of $x$ such that $\ms{M} |_U$ is a $\sH_{\omega,\bc}(U,W_x)$-module for any $\sH_{\omega,\bc}(X,W)$-module $\ms{M}$. If $\h = T_x X$ then 
\[
\widehat{\mc{O}}_{U,x} \o_{\mc{O}_U} \sH_{\omega,\bc}(U,W_x) \cong \H_{\bc_x}(\widehat{\h},W_x).
\]
Thus, for any $\ms{M} \in \mc{M}_{\mr{Coh}({\mc{O}})}(\sH_{\omega,\bc}(X,W))$, the space $\widehat{\mc{O}}_{U,x} \o_{\mc{O}_U} (\ms{M} |_U)$ is a $\H_{\bc_x}(\widehat{\h},W_x)$-module, finitely generated over $\widehat{\kk[\h]}_0$. Let $\eu \in \H_{\bc_x}(\h,W_x)$ be the Euler operator, as defined in \eqref{eq:eulerelement}. Given a $\H_{\bc_x}(\h,W_x)$-module $M$, we write 
\[
E(M) = \{ m \in M \, | \, (\eu - a)^N (m) = 0 \textrm{ for some $a \in \kk$ and $N > 0$} \}
\]
for the space of $\eu$-locally finite vectors in $M$. It is a $\H_{\bc_x}(\h,W_x)$-submodule of $M$. As explained in \cite[Section~2.4]{BE}, the space 
\[
\mc{J}_x(\ms{M}) := E(\widehat{\mc{O}}_{U,x} \o_{\mc{O}_U} (\ms{M} |_U))
\] 
of $\eu$-locally finite vectors in $\widehat{\mc{O}}_{U,x} \o_{\mc{O}_U} (\ms{M} |_U)$ is a $\H_{\bc_x}(\h,W_x)$-module in category $\mc{O}_{\bc_x}(\h)$. 

\begin{lem}\label{lem:Jacqeutfunctor}
	The Jacquet functor $\mc{J}_x \colon \mc{M}_{\mr{Coh}({\mc{O}})}(\sH_{\omega,\bc}(X,W)) \to \mc{O}_{\bc_x}(\h)$ is exact. Moreover, $\mc{J}_x(\ms{M}) = 0$ if and only if $x \notin \Supp_X \ms{M}$. 
\end{lem}

\begin{proof}
	As shown in \cite[Theorem~2.3]{BE}, the functor $E$ is an equivalence between the category of $\H_{\bc_x}(\widehat{\h},W_x)$-modules that are finitely generated over $\widehat{\kk[\h]}_0$ and category $\mc{O}_{\bc_x}(\h)$. In particular, it is exact and conservative. This implies the statements of the lemma. 
\end{proof}

\section{Smooth morphisms and closed embeddings}

As in the case of $\dd$-modules, we consider next the case of smooth morphisms and closed embeddings. 

\subsection{Smooth morphisms}

If $\varphi \colon Y \to X$ is a morphism between smooth varieties then we have the standard correspondence of cotangent bundles
	\begin{equation}\label{eq:cotangentmap}
T^* Y \stackrel{\rho}{\longleftarrow} Y \underset{X}{\times} T^* X \stackrel{\varpi}{\longrightarrow} T^* X,
\end{equation}
where $\varpi$ is projection and $\rho(y,x,v) = (y, (d_y \varphi)^*(v))$. If $\varphi$ is $W$-equivariant then this descends to a correspondence
$$
(T^* Y)/W \stackrel{\rho/W}{\longleftarrow} (Y \underset{X}{\times} T^* X)/W \stackrel{{\varpi}/W}{\longrightarrow} (T^* X)/W.
$$
If $y \in Y$ and $x = \varphi(y)$ then we write $\widehat{\varphi} \colon \widehat{Y}_y \to \widehat{X}_x$ for the induced map on formal neighbourhoods.

\begin{lem}\label{lem:formalcotangentcorrespondence}
	Let $\Lambda \subset T^* X$ be a closed subvariety. Then 
	$$
	\widehat{Y}_y \underset{Y}{\times} \rho(\varpi^{-1}(\Lambda)) \cong \rho_{\widehat{\varphi}}(\varpi^{-1}_{\widehat{\varphi}}(\widehat{X}_{x} \underset{X}{\times} \Lambda)).
	$$
\end{lem}

\begin{proof}
	Let $\mc{A}_X = \Sym_{\mc{O}_X} \Theta_X$. The $\mc{O}_Y$-module $\varphi^* \mc{A}_X$ carries an $\mc{O}_Y$-linear action of $\Theta_Y$, making it into a $(\mc{A}_Y = \Sym_{\mc{O}_Y} \Theta_Y)$-module. In other words, it defines a (quasi-coherent) $\mc{O}$-module on $T^* Y$. More generally, we may regard $\mc{O}_{\Lambda}$ as an $\mc{A}_X$-module and 
	\[
	\varphi^* \mc{A}_X \o_{\varphi^{-1} \mc{A}_X} \varphi^{-1} \mc{O}_{\Lambda} \cong \varphi^* \mc{O}_{\Lambda}
	\]
	is a quasi-coherent sheaf on $T^* Y$. Then $\rho(\varpi^{-1}(\Lambda))$ is the support of this sheaf. This implies that $\widehat{Y}_y \underset{Y}{\times} \rho(\varpi^{-1}(\Lambda))$ is the support of $\widehat{\mc{O}}_{Y,y} \o_{\mc{O}_Y} \varphi^* \mc{O}_{\Lambda}$. Just as in the series of identifications in the proof of Lemma~\ref{lem:completepullbackbi}, we have 
	\begin{equation}\label{eq:sheafcompleterhophi}
			\widehat{\mc{O}}_{Y,y} \o_{\mc{O}_Y} \varphi^* \mc{O}_{\Lambda} \cong \widehat{\mc{O}}_{Y,y} \o_{\widehat{\mc{O}}_{X,x}} (\widehat{\mc{O}}_{X,x} \o_{\mc{O}_X} \mc{O}_{\Lambda}).
	\end{equation}
	Here $\widehat{\mc{O}}_{X,x} \o_{\mc{O}_X} \mc{O}_{\Lambda}$ is a module over the (free, of rank $\dim X$) sheaf $\mc{A}_{\widehat{X}_{x}}$ of continuous vector fields on the formal disc $\widehat{X}_{x}$ via the identification $\mc{A}_{\widehat{X}_{x}} \o_{\mc{A}_X} \mc{O}_{\Lambda} = \widehat{\mc{O}}_{X,x} \o_{\mc{O}_X} \mc{O}_{\Lambda}$. The support of the right hand side of \eqref{eq:sheafcompleterhophi} is given by $\rho_{\widehat{\varphi}}(\varpi^{-1}_{\widehat{\varphi}}(\widehat{X}_{x} \underset{X}{\times} \Lambda))$, which implies the statement of the lemma.   	
\end{proof}

\begin{lem}\label{lem:isotropicsmoothmorphism}
	If $\varphi \colon Y \to X$ is a smooth surjective morphism and $\Lambda \subset T^*X$, then $\rho(\varpi^{-1}(\Lambda))$ is isotropic in $T^* Y$ if and only if $\Lambda$ is isotropic in $T^* X$. 
\end{lem}

\begin{proof}
	Working \'etale locally on $Y$, as in \cite[Theorem~41.13.1]{stacks}, we may assume $Y = X \times Z$ with $\varphi$ projection on $X$. Then $\rho(\varpi^{-1}(\Lambda)) = \Lambda \times T^*_Z Z$ and the claim is obvious. 
\end{proof}
  
The proof of the following result is similar to that for $\dd$-modules \cite[Theorem~2.4.6]{HTT}. 

\begin{thm}\label{thm:smoothSSmelys}
If $\varphi$ is a smooth $\bc$-melys morphism and $\ms{M}$ a coherent $\sH_{\omega,\bc}(X,W)$-module then 
\begin{enumerate}
	\item[(i)] $\mc{H}^i(\varphi^! (\ms{M})) = 0 $ for $i \neq -(\dim Y - \dim X)$, 
	\item[(ii)] $\varphi^0 (\ms{M})$ is coherent over $\sH_{\varphi^* \omega,\varphi^*\bc}(Y,W)$; and 
	\item[(iii)] $\SS(\varphi^0 \ms{M}) = (\rho/W) ({\varpi}/W)^{-1}(\SS(\ms{M}))$.
\end{enumerate} 
Hence, if $\varphi$ is surjective then $\ms{M}$ is holonomic if and only if $\varphi^0 \ms{M}$ is holonomic. 
\end{thm}

\begin{proof}
	By definition, $\varphi^!(\ms{M}) = {L} \varphi^{0} (\ms{M})[\dim Y - \dim X]$ and hence
	\[
	\mc{H}^i(\varphi^! (\ms{M})) \cong \mc{H}^{i + \dim Y - \dim X}({L} \varphi^{0} (\ms{M})).
	\]
	Therefore, part (i) follows from the fact that $\mc{H}^i({L} \varphi^{0} (\ms{M})) = 0 $ for $i \neq 0$ because $\mc{O}_Y$ is flat over $\varphi^{-1} \mc{O}_X$. 
	
	Part (ii). As in the proof of \cite[Proposition~1.5.13]{HTT}, it suffices to show that the canonical map $\sH_{\varphi^* \omega,\varphi^* \bc}(Y,W) \to \sH_{Y \to X}$ is surjective. This can be checked in a formal neighbourhood of each point of $Y$. Let $\widehat{\varphi} \colon \widehat{Y}_y \to \widehat{X}_x$ denote the completion of $\varphi$ and set $\h = T_y Y$, $\mf{k} = T_x X$. Then Cartan's Lemma identifies $\widehat{Y}_y = \widehat{T_y Y}_0$ and $\widehat{X}_x = \widehat{T_x X}_0$ $W$-equivariantly, where $W$ acts linearly on the tangent spaces. Under this identification, the morphism $\widehat{\varphi}$ equals $\widehat{p}$, where $p \colon \h \twoheadrightarrow \mf{k}$ is projection with kernel $K := \ker (d_y \varphi)$. In this case, $\bc_x$-melys means that $W(\bc_x)$ acts trivially on $K$. We may assume $W = W(\bc_x)$. Then $\H_{\bc_x}(\h,W) = \dd(K) \boxtimes \H_{\bc_x}(\mf{k},W)$ and surjectivity is clear. 	
	
	Part (iii). The commutativity of the following diagram 
\[
	    	\begin{tikzcd}
		T^* Y \ar[d,"\pi_Y"'] & \ar[l,"\rho"'] Y \underset{X}{\times} T^* X \ar[r,"\varpi"] & T^* X \ar[d,"\pi_X"] \\
		(T^* Y)/W & \ar[l,"{\rho}/W"'] (Y \underset{X}{\times} T^* X)/W \ar[r,"{\varpi}/W"] & (T^* X)/W
	\end{tikzcd}
    \]
	implies that it suffices to show that $\SS_Y(\varphi^0 \ms{M}) = \rho(\varpi^{-1}(\SS_X(\ms{M})))$. Thinking of both varieties as schemes over $Y$, it suffices to show that the equality holds in the preimage of the formal neighbourhood of each point of $Y$. By Lemma~\ref{lem:localSSfixed}, we may assume that $y \in Y^W$. Let $x = \varphi(y)$. We wish to compute the singular support of $\sH_{Y \to X} \o_{\varphi^{-1}\sH_{\omega,\bc}(X,W)} \varphi^{-1}(\ms{M})$ in a formal neighbourhood of $y$. By equation \eqref{eq:completesingularsupport} and  Lemma~\ref{lem:completepullbackbi}(i), 
	\begin{align*}
		\widehat{Y}_y \underset{Y}{\times} \SS_Y(\varphi^0 \ms{M}) & = \SS_{\widehat{Y}_{y}}(\widehat{\mc{O}}_{Y,y} \o_{\mc{O}_Y} \varphi^0 \ms{M}) \\
		& = \SS_{\widehat{Y}_{y}}(\widehat{\varphi}^0 (\widehat{\mc{O}}_{X,x} \o_{\mc{O}_X} \ms{M})).
	\end{align*}
	
	 By Lemma~\ref{lem:forgetcommutepushpull}, we may assume $W = W(\bc)$. If we let $M$ be a finitely generated $\H_{\bc}(\mf{k}, W)$-module such that $\widehat{\mc{O}}_{X,x} \o_{\mc{O}_X} \ms{M} \cong \widehat{\mc{O}}_{\mf{k},0} \o_{\mc{O}_{\mf{k}}} M$, then $p^0 M = \mc{O}(K) \boxtimes M$ and it is trivial that $\SS_{\h}(p^0 M) = \rho_p(\varpi^{-1}_p(\SS_{\mf{k}}(M)))$ and hence
	$$
	\SS_{\widehat{\h}_{0}}(\widehat{p}^0 (\widehat{\mc{O}}_{\mf{k},0} \o_{\mc{O}_{\mf{k}}} M)) = \rho_{\widehat{p}}(\varpi^{-1}_{\widehat{p}}(\SS_{\widehat{\mf{k}}_{0}}(\widehat{\mc{O}}_{\mf{k},0} \o_{\mc{O}_{\mf{k}}} M))).
	$$
	We deduce that
	$$
	\SS_{\widehat{Y}_{y}}(\widehat{\varphi}^0 (\widehat{\mc{O}}_{X,x} \o_{\mc{O}_X} \ms{M})) = \rho_{\widehat{\varphi}}(\varpi^{-1}_{\widehat{\varphi}}(\SS_{\widehat{X}_{x}}(\widehat{\mc{O}}_{X,x} \o_{\mc{O}_{X}} \ms{M}))).
	$$
	Finally, we note that $\SS_{\widehat{X}_{x}}(\widehat{\mc{O}}_{X,x} \o_{\mc{O}_{X}} \ms{M}) = \widehat{X}_{x} \underset{X}{\times} \SS_X(\ms{M})$ by \eqref{eq:completesingularsupport} again. Combining these gives
	$$
	\widehat{Y}_y \underset{Y}{\times} \SS_Y(\varphi^0 \ms{M}) = \rho_{\widehat{\varphi}}(\varpi^{-1}_{\widehat{\varphi}}(\widehat{X}_{x} \underset{X}{\times} \SS_X(\ms{M}))),
	$$
	and we deduce from Lemma~\ref{lem:formalcotangentcorrespondence} that 
	$$
	\widehat{Y}_y \underset{Y}{\times} \SS_Y(\varphi^0 \ms{M}) = \widehat{Y}_y \underset{Y}{\times} \rho (\varpi^{-1}(\SS_X(\ms{M}))).
	$$
	Since this holds for all $y \in Y$, we must have $\SS_Y(\varphi^0 \ms{M}) = \rho (\varpi^{-1}(\SS(\ms{M})))$. 

By Lemma~\ref{lem:isoclosed}, $\ms{M}$ is holonomic if and only if $\SS_X(\ms{M})$ is isotropic in $T^* X$; the analogous statement holds for $\varphi^0 \ms{M}$. Therefore, the final claim follows directly from the equality $\SS_Y(\varphi^0 \ms{M}) = \rho(\varpi^{-1}(\SS_X(\ms{M})))$ and Lemma~\ref{lem:isotropicsmoothmorphism}.
\end{proof}

In particular, we deduce:

\begin{cor}\label{cor:holEtLoc}
	Suppose $\varphi \colon Y \to X$ is a $\bc$-melys finite \'etale morphism between smooth $W$-varieties. Then a $\sH_{\omega,\bc}(X,W)$-module $\ms{M}$ is holonomic if and only if $\varphi^0 \ms{M}$ is a holonomic $\sH_{\varphi^* \omega,\varphi^* \bc}(Y,W)$-module.
\end{cor}

We note for later the following results. 

\begin{prop}\label{prop:etalepushpullsummand}
	Suppose $\varphi \colon Y \to X$ is a $\bc$-melys finite \'etale morphism between smooth $W$-varieties. Then the adjunction $\mr{Id} \to \varphi_0 \circ \varphi^0$ is a split embedding. This means that $\ms{M}$ is a direct summand of $\varphi_0(\varphi^0(\ms{M}))$ for any $\sH_{\omega,\bc}(X,W)$-module $\ms{M}$.
\end{prop}

\begin{proof}
Since $\varphi$ is finite, it is affine and we identify $\sH_{\varphi^* \omega,\varphi^* \bc}(Y,W)$ with $\varphi_{\idot} \sH_{\varphi^* \omega,\varphi^* \bc}(Y,W)$; that is, we consider $\sH_{\varphi^* \omega,\varphi^* \bc}(Y,W)$ as a sheaf on $X/W$. It follows from Lemma~\ref{lem:etalebimodules} that the adjunction $\ms{M} \to \varphi_0 (\varphi^0 \ms{M})$ is the morphism 
\[
	\ms{M} \to \sH_{\varphi^* \omega,\varphi^*\bc}(Y,W) \o_{\sH_{\omega,\bc}(X,W)} \ms{M} \cong \mc{O}_Y \o_{\mc{O}_X} \ms{M}.
\]
Since $\varphi$ is finite and flat, there is a ($\mc{O}_X$-linear) trace map $\Tr \colon \mc{O}_Y \to \mc{O}_X$ (again, considering $\mc{O}_Y$ as a sheaf on $X/W$) which we may assume is the identity on $\mc{O}_X$ after dividing by the degree of $\varphi$. This splits the injection $\mc{O}_X \to \mc{O}_Y$. Let $\mc{J}$ denote the kernel of $\Tr$. One can check directly that $\Tr$ is $W$-equivariant and is a morphism of $\dd_X$-modules. Therefore, we just need to check locally that if $D_v \in \sH_{\omega,\bc}(X,W)$ is a Dunkl operator then $D_v$ maps $\mc{J} \o_{\mc{O}_X} \ms{M}$ to itself. If $V$ is a $W$-orbit in $X$ and $U = \varphi^{-1}(V)$ then we may replace $\mc{O}_X$ (resp. $\mc{O}_Y$) by the semi-local ring $\mc{O}_{X,V}$ (resp. $\mc{O}_{Y,U}$). It follows from the proof of Lemma~\ref{lem:goodneighbourhood} that each $Z$ is principal in $V$. If $v$ is a derivation of $\mc{O}_{X,V}$, $f \in \mc{J}$ and $m \in \ms{M} |_V$ then 
\begin{align*}
D_v (f \o m) & = [D_v,f] \o m + f \o D_v (m) \\
& = v(f) \o m + \sum_{(w,Z) \in \mc{S}(X)} \frac{2 \bc(w,Z)}{1-\lambda_{w,Z}} \frac{v(f_Z)}{f_Z} (w(f) - f) \o m + f \o D_v(m). 
\end{align*}
As noted already, one can check directly that for $g \in \mc{O}_Y$, $\Tr(v(g)) = v(\Tr(g))$, and hence $v(f) \in \mc{J}$ if $f \in \mc{J}$. Now, let $a = \frac{v(f_Z)}{f_Z} (w(f) - f)$. Note that $v(f_Z) \in \mc{O}_{X,V}$ and $w(f)  - f \in \mc{J}$ since $\mc{J}$ is $W$-stable. Thus, $f_Z a = v(f_Z)(w(f) - f)$ belongs to $\mc{J}$. Even though $\mc{O}_{X,V}$ need not be a domain because we cannot assume $X$ is connected, $f_Z \in \mc{O}_{X,V}$ is not a zero-divisor because its zero set is a proper closed subset of a connected component of $X$. Thus, $0 = \Tr(f_Z a) = f_Z \Tr(a)$ implies that $\Tr(a) = 0$ and thus $a \in \mc{J}$. It follows that $D_v(f \o m)$ belongs to $(\mc{J} \o_{\mc{O}_{X}} \ms{M}) |_V$. 
\end{proof} 

\begin{lem}\label{lem:etalepushforward}
	Suppose $\varphi \colon Y \to X$ is a $\bc$-melys finite \'etale morphism between smooth $W$-varieties. If $\ms{M}$ is a holonomic $\sH_{\varphi^* \omega,\varphi^* \bc}(Y,W)$-module then $\varphi_0 \, \ms{M}$ is a holonomic $\sH_{\omega,\bc}(X,W)$-module.
\end{lem}

\begin{proof}
Since $\varphi$ is finite, it is affine. The statement is local on $X$, therefore we may assume that $X$, and hence $Y$, is affine. Then, there is an inclusion
	\[
	\H_{\omega,\bc}(X,W) \hookrightarrow \mc{O}(Y) \o_{\mc{O}(X)} \H_{\omega,\bc}(X,W) \cong \H_{\varphi^* \omega,\varphi^* \bc}(Y,W), 
	\]
    with $\H_{\varphi^* \omega,\varphi^* \bc}(Y,W)$ a finite projective $\H_{\omega,\bc}(X,W)$-module. This implies that if $\mc{F}_{\idot}$ is a good filtration on $M = \Gamma(Y,\ms{M})$ then it is also a good filtration when considered as a $\H_{\omega,\bc}(X,W)$-module. Hence the singular support $\SS_Y(\ms{M})$ is the support of $\gr_{\mc{F}} M$ as a $\mc{O}(T^*Y)$-module and the singular support $\SS_X(\varphi_0 \ms{M})$ is the support of $\gr_{\mc{F}} M$ as a $\mc{O}(T^*X)$-module. In the case $\varphi$ is \'etale, in diagram \eqref{eq:cotangentmap} $\rho$ is an isomorphism and $\varpi$ is finite. Then the lemma follows from the fact that if $\Lambda \subset T^* Y$ is isotropic then so too is $\varpi(\rho^{-1}(\Lambda))$ and that $\SS_X(\varphi_0 \ms{M}) = \varpi(\rho^{-1}(\SS_Y(\ms{M})))$. 
\end{proof}

\subsection{Kashiwara's Theorem}

In this section we consider closed embeddings. If $Y \subset X$ is a (smooth, $W$-stable) closed subvariety and $i$ the corresponding closed embedding, then for a $\sH_{\omega,\bc}(X,W)$-module $\ms{M}$, we define $i^{\natural} \ms{M} = \mc{H}om_{i^{-1} \sH_{\omega,\bc}(X,W)}(\sH_{X \leftarrow Y},i^{-1} \ms{M})$, which is a left $\sH_{i^*\omega,i^*\bc}(Y,W)$-module. As an $\mc{O}_Y$-module, 
\[
i^{\natural} \ms{M} \cong \mc{H}om_{i^{-1} \mc{O}_X}(\mc{O}_Y,i^{-1} \ms{M}) \o_{\mc{O}_Y} (i^{*} K_X \o K_Y^{-1}).
\]
Since $\sH_{X \leftarrow Y}$ is supported on $Y$ as a left $i^{-1} \sH_{\omega,\bc}(X,W)$-module, \cite[Proposition~1.5.14 and Proposition~1.5.16]{HTT} go through verbatim, showing that for $\ms{M}^{\idot} \in D^b_{\mathrm{qc}}(  \sH_{\omega,\bc}(X,W))$
	$$
	i^! \ms{M}^{\idot} \cong R \mc{H}om_{i^{-1} \sH_{\omega,\bc}(X,W)}(\sH_{X \leftarrow Y}, i^{-1} \ms{M}^{\idot}) \cong R i^{\natural} \ms{M}^{\idot}. 
	$$
If $U$ is the complement to $Y$ in $X$, write $j$ for the open embedding $U \hookrightarrow X$. Note that $(j^!,j_+)$ is a pair of adjoint functors. The subsheaf $\Gamma_Y(\ms{M})$ consisting of all sections of a $\sH_{\omega,\bc}(X,W)$-module $\ms{M}$ supported on $Y$ is a $\sH_{\omega,\bc}(X,W)$-submodule. The functor $\Gamma_Y$ extends to an endofunctor $R \Gamma_Y \colon D^b(\sH_{\omega,\bc}(X,W)) \to D^b(\sH_{\omega,\bc}(X,W))$. 

\begin{lem}\label{lem:localcohomology}
	Let $Y$ be a closed, $W$-stable subset of $X$ and $U = X \setminus Y$, with $i \colon Y \hookrightarrow X$ and $j \colon U \hookrightarrow X$ the corresponding inclusions. 
	\begin{enumerate}
		\item[(i)] There is an exact triangle 
		$$
		R \Gamma_Y(\ms{M}^{\idot}) \to \ms{M}^{\idot} \to j_+ j^!(\ms{M}^{\idot}) \stackrel{[1]}{\longrightarrow} \cdots \qquad \forall \, \ms{M}^{\idot} \in D^b_{\mathrm{qc}}(  \sH_{\omega,\bc}(X,W)).
		$$
		\item[(ii)] If $Y$ is smooth then the above triangle becomes 
		$$
		i_0 i^{!} (\ms{M}^{\idot}) \to \ms{M}^{\idot} \to j_+ j^!(\ms{M}^{\idot}) \stackrel{[1]}{\longrightarrow} \cdots
		$$
	\end{enumerate}
\end{lem}

\begin{proof}
	The proofs of both parts are identical to the $\dd$-module case, see \cite[Proposition~1.7.1]{HTT}, since the key fact used in these proofs is \cite[Lemma~1.5.17]{HTT} saying that $\mc{O}_Y \o_{i^{-1} \mc{O}_X} i^{-1} R j_* \ms{K}^{\idot} = 0$ for any $\ms{K}^{\idot} \in D^b(\mc{O}_U)$. 
\end{proof}

\begin{lem}
	If $\ms{M}^{\idot}$ is supported on $U$ then $\ms{M}^{\idot} \to j_+ j^!(\ms{M}^{\idot})$ is an isomorphism. 
\end{lem}

\begin{proof}
	The claim follows from Lemma~\ref{lem:localcohomology}(i) if we can argue that $R \Gamma_Y(\ms{M}^{\idot}) \cong 0$. 
	
	By induction on the number of non-zero cohomology groups of $\ms{M}^{\idot}$, we may assume $\ms{M}^{\idot} = \ms{M}$ is concentrated in degree zero (with support contained in $U$). If $\mc{I} \lhd \mc{O}_X$ is the ideal defining $Y$ then 
	\[
	R^i \Gamma_Y(\ms{M})_x = \lim_{m \to \infty} \Ext_{\mc{O}_{X,x}}^i(\mc{O}_{X,x} / \mc{I}^m_x,\ms{M}_x).
	\]
	Since $Y \cap U = \emptyset$, we have $\Ext_{\mc{O}_{X,x}}^i(\mc{O}_{X,x} / \mc{I}^m_x,\ms{M}_x) = 0$ for all $x \in X$. Thus, $R \Gamma_Y(\ms{M}^{\idot}) \cong 0$.
\end{proof} 

In order to state a version of Kashiwara's theorem for sheaves of Cherednik algebras, we recall the notion of $\bc$-regular subvariety introduced in \cite{ThompsonHolI}. First, the reader is reminded that there is a notion of \textit{regular parameters} $\bc$ for rational Cherednik algebras; see Definition~\ref{defn:bcregularRCA} of the appendix. 

\begin{defn}
	Let $Y$ be a $W$-stable, smooth, locally closed subvariety of $X$. For each $y \in Y$ define 
	\[
	W_{y,\perp} = \langle w\,  | \, (w,Z) \in \mc{S}_{\bc}(X) \textrm{ and } Y \subset Z \rangle \subset W \quad \textrm{and} \quad V_y := T_y X / T_y Y.
	\] 
	The group $W_{y,\perp}$ acts linearly on the vector space $V_y$ and hence it makes sense to talk of regular parameters on the set of reflections in $(V_y,W_{y,\perp})$. The subvariety $Y$ is said to be $\bc$-regular if $\bc |_{(V_y,W_{y,\perp})}$ is regular for all $y \in Y$. 
\end{defn}

One should think of $\bc$-regular as meaning that the restriction of $\bc$ to the subgroup $W_{\perp}$ of $W$ acting on the perpendicular direction to $Y$ in $X$ is regular.

\begin{example}
	If $Y$ is a closed ($W$-stable) smooth subvariety such that $Y \not\subset Z$ for any $(w,Z) \in \mc{S}_{\bc}(X)$ then it is $\bc$-regular. In this case, we say $Y$ is \textit{strongly $\bc$-regular}. 
\end{example}

Recall that the embedding of any smooth closed subvariety is a $\bc$-melys morphism.

\begin{example}\label{ex:invariantfunctionmelys2}
	Let $f \colon X \to \mathbb{A}^1$ be a $W$-invariant function and define the closed embedding $\phi = (\mr{Id},f) \colon X \hookrightarrow X \times \mathbb{A}^1$. Then $\phi$ is strongly $\bc$-regular. If $\ms{M}$ is a $\sH_{\omega,\bc}(X,W)$-module and $\ms{N}$ a $\dd(\mathbb{A}^1)$-module, then $\ms{M} \o_{f^{-1} \mc{O}_{\mathbb{A}^1}} f^{-1} \ms{N} = \phi^0(\ms{M} \boxtimes \ms{N})$ is a $\sH_{\omega,\bc}(X,W)$-module. Given a derivation $v$, the corresponding Dunkl operator $D_v$ acts as
	$$
	D_v(m \otimes n) = (D_v m) \otimes n + v(f) m \otimes \partial_t n,
	$$
	the group $W$ acts trivially on the second factor, and $\mc{O}_X$ acts in the natural way. This construction will play an important role in Section~\ref{sec:bfunctionRCA}.
\end{example}
	
It was shown in \cite[Theorem 4.5]{ThompsonHolI} that Kashiwara's Theorem generalises to the setting of sheaves of Cherednik algebras. We recall a version of the result here, sketching the parts of the proof not explicitly covered in \cite{ThompsonHolI}.  

\begin{thm}\label{thm:Kash}
	Let $Y \subset X$ be a smooth closed subvariety. 
	\begin{enumerate}
		\item[(i)] The functor $i_0$ is exact, preserves coherence and,  for any $\ms{M} \in \Coh{\sH_{i^* \omega,i^* \bc}(Y,W)}$, 
		\begin{equation}\label{eq:SSclosedembedding}
			\SS_X(i_0 (\ms{M})) = \varpi(\rho^{-1}(\SS_Y(\ms{M}))).
		\end{equation}
	\item[(ii)] Assume that $Y \subset X$ is $\bc$-regular. The functor 
	\[
	i^{\natural} = \mc{H}^0(i^!) \colon \Qcoh{\sH_{\omega,\bc}(X,W)}_Y \rightarrow \Qcoh{\sH_{i^* \omega,i^* \bc}(Y,W)}
	\]
	is an equivalence with quasi-inverse $i_0$. It restricts to an equivalence 
	$$
	i^{\natural} \colon \Coh{\sH_{\omega,\bc}(X,W)}_Y \rightarrow \Coh{\sH_{i^* \omega,i^* \bc}(Y,W)},
	$$
	again with quasi-inverse $i_0$.
	\end{enumerate}
\end{thm}

\begin{proof}
(i) We first note that $i_0$ preserves coherence for closed embeddings because $\mc{H}_{X \leftarrow Y}$ is a coherent left $\sH_{\omega,\bc}(X,W)$-module. In fact, it is locally a cyclic module.
	
If $0 \to \ms{M}_1 \to \ms{M}_2 \to \ms{M}_3 \to 0$ is a short exact sequence in $\Qcoh{\sH_{i^* \omega,i^* \bc}(Y,W)}$ then we can check in the formal neighbourhood $\widehat{X}_y$ of each $y \in Y$ in $X$ that 
\begin{equation}\label{eq:ioses}
	0 \to i_0(\ms{M}_1) \to i_0(\ms{M}_2) \to i_0(\ms{M}_3) \to 0
\end{equation}
is exact on the right. Let $y \in Y$ and $\widehat{i} \colon \widehat{Y}_y \to \widehat{X}_y$. If $\ms{N} = i_0(\ms{M})$ then Lemma~\ref{lem:completepullbackbi}(ii) says that $\widehat{\mc{O}}_{X,y} \o_{\mc{O}_X} \ms{N} = \widehat{i}_0(\widehat{\mc{O}}_{Y,y} \o_{\mc{O}_Y} \ms{M})$. Thus, tensoring \eqref{eq:ioses} by $\widehat{\mc{O}}_{X,y}$ shows that it suffices to check that $\widehat{i}_0$ is exact. Again, by Lemma~\ref{lem:completepullbackbi}, it suffices to check that $i_0$ is exact for a linear closed embedding (corresponding to the embedding $T_y Y \to T_y X$). 
	
	Let $\h$ be a $W$-module and $\mf{k}$ a submodule; we write $i$ for the embedding of $\mf{k}$ in $\h$. If $M = \Gamma(\mf{k},\ms{M})$ then $i_0(M) = K_{\h/\mf{k}}^{-1} \o_{\kk[\h]}\H_{\overline{\bc}}(\h,W) \o_{\H_{\bc}(\mf{k},W)} M$. For both coherence and singular support, we may assume $W = W(\bc)$. In this case, we may choose a $W$-stable decomposition $\h = \mf{k} \oplus \mf{k}'$ with $W = W_{\mf{k}} \times W_{\mf{k}'}$, which gives rise to a factorization $\H_{\bc}(\h,W) = \H_{\bc}(\mf{k},W_{\mf{k}}) \o \H_{\bc}(\mf{k}',W_{\mf{k}'})$ since $W$ acts on $\h$ as a complex reflection group. Then  
	\[
	i_0(M) = K_{\mf{k}'}^{-1} \o_{\kk[\mf{k}']} (\H_{\overline{\bc}}(\mf{k}',W_{\mf{k}'}) / \H_{\overline{\bc}}(\mf{k}',W_{\mf{k}'}) \kk[\mf{k}']_+) \o_{\kk} M.
	\] 
	In particular, it is clear that $i_0$ is exact and that $\SS(i_0 (\ms{M})) = \varpi(\rho^{-1}(\SS(\ms{M})))$. The fact that the equality $\SS(i_0 (\ms{M})) = \varpi(\rho^{-1}(\SS(\ms{M})))$ holds in the special case $i \colon \mf{k} \hookrightarrow \h$ implies that the equality holds (as stated in \eqref{eq:SSclosedembedding}) for arbitrary $i \colon Y \hookrightarrow X$ follows exactly as in the proof of Theorem~\ref{thm:smoothSSmelys}.

	(ii) Exactly as in the first half of the proof of \cite[Theorem 4.5]{ThompsonHolI}, we have adjunctions $\mr{Id} \to i^{\natural} i_0$ and $i_0 i^{\natural} \to \mr{Id}$ and it suffices to check in a formal neighbourhood of each $y \in Y$ that both of these are isomorphisms. Since $i_0$ and $i^{\natural}$ are compatible, in an appropriate sense, with the morphism $\varphi$ of Proposition~\ref{prop:etalemelysiso}, arguing as in the proof of Lemma~\ref{lem:localSSfixed} allows us to reduce to the case where $y$ is fixed by $W$.
		
	Since $\mc{O}_Y$ is finitely presented over $\mc{O}_X$ and $\widehat{\mc{O}}_{X,y}$ is flat over $\mc{O}_X$, the canonical map 
	$$
	\widehat{\mc{O}}_{X,y} \o_{\mc{O}_X} \Hom_{\mc{O}_X}(\mc{O}_Y,\ms{N}) \to \Hom_{\widehat{\mc{O}}_{X,y}}(\widehat{\mc{O}}_{Y,y},\widehat{\mc{O}}_{X,y} \o_{\mc{O}_X} \ms{N})
	$$
	 is an isomorphism. Now, if $\ms{N} = i_0(\ms{M})$ then as we have seen in (i), $\widehat{\mc{O}}_{X,y} \o_{\mc{O}_X} \ms{N} = \widehat{i}_0(\widehat{\mc{O}}_{Y,y} \o_{\mc{O}_Y} \ms{M})$. Combining these shows that 
	 $$
	 \widehat{\mc{O}}_{X,y} \o_{\mc{O}_X} (\mr{Id} \to i^{\natural} i_0) \cong \mr{Id} \to \widehat{i}^{\natural} \, \widehat{i}_0.
	 $$
	 Since $\widehat{\mc{O}}_{X,y} \cong \widehat{\mc{O}}_{T_y X,0}$ and $\widehat{\mc{O}}_{Y,y} \cong \widehat{\mc{O}}_{T_y Y,0}$ such that $\widehat{i}$ gets identified with the completion of the linear embedding $T_y Y \to T_y X$, the above isomorphism reduces us to the case of the completion at zero of a linear embedding. The claim that $\widehat{i}^{\natural}$ and $\widehat{i}_0$ are equivalences in this case is explained in the proof of \cite[Theorem 4.5]{ThompsonHolI}. The proof that $i_0 i^{\natural} \to \mr{Id}$ is an isomorphism is similar. 
	  
Finally, since $i_0$ is an equivalence, which preserves coherence by part (i), it follows that its inverse $i^{\natural}$ also preserves coherence.  
\end{proof}

The functor $i^{\natural}$ restricts to an equivalence between the corresponding categories of holonomic modules. Indeed, as in the case of $\dd$-modules, we have:

\begin{cor}\label{cor:holoclosedpushforward}
	Let $i \colon Y \hookrightarrow X$ be a closed embedding, with $Y$ smooth, and choose $\ms{M}^{\idot}$ in $D^b_{\mr{Coh}}(\sH_{i^* \omega,i^* \bc}(Y,W))$. Then 
	\[
i_+ (\ms{M}^{\idot}) \in D^b_{\mr{Hol}}(\sH_{\omega,\bc}(X,W)) \quad \Leftrightarrow \quad \ms{M}^{\idot} \in D^b_{\mr{Hol}}(\sH_{i^* \omega,i^* \bc}(Y,W)).	
\] 
\end{cor}

\begin{proof}
	Since $i_0$ is exact by Theorem~\ref{thm:Kash}(i), it suffices to show that $i_0(\ms{M})$ is holonomic if and only if $\ms{M}$ is holonomic. For a closed embedding of smooth varieties $i \colon Y \hookrightarrow X$, a subvariety $\Lambda \subset T^* Y$ is isotropic if and only if $\varpi(\rho^{-1}(\Lambda))$ is isotropic in $T^*X$. Therefore the claim follows directly from equation \eqref{eq:SSclosedembedding}.  
\end{proof}

\begin{cor}
	If $X$ is affine and $Y \subset X$ is a $\bc$-regular closed subvariety then 
	$$
	\H_{i^* \omega,i^*\bc}(Y,W) = \mathbb{I}(\H_{\omega,\bc}(X,W)I) / \H_{\omega,\bc}(X,W) I
	$$
	where $I = I(Y)$ and $\mathbb{I}(J)$ is the idealizer in $\H_{\omega,\bc}(X,W)$ of the left ideal $J$. 
\end{cor}

\begin{proof}
	Under the equivalence of Theorem~\ref{thm:Kash}, the generator $\H_{i^*\omega,i^*\bc}(Y,W)$ of $\Lmod{\H_{i^*\omega,i^*\bc}(Y,W)}$ is sent to $i_0(\H_{i^*\omega,i^*\bc}(Y,W)) = \H_{\omega,\bc}(X,W) / \H_{\omega,\bc}(X,W) I$. The result follows since the right hand side is the endomorphism ring of $i_0(\H_{i^*\omega,i^*\bc}(Y,W))$. 
\end{proof} 

\subsection{Reduction to the linear case}

The following result implies that, locally around each $W$-fixed point in $X$, modules for the sheaf of Cherednik algebras may be thought of as modules over the rational Cherednik algebra $\H_{\bc}(\h,W)$, supported on a certain smooth closed subset. However, for this to be true we require $\omega \in \mathbb{H}^2 ( X,\Omega_X^{1,2})^W$ to be Zariski locally trivial, which is a condition we will impose later in the paper when applying Proposition~\ref{prop:goodembedding}. 

\begin{prop}\label{prop:goodembedding}
For each $x \in X^W$, there is an open neighbourhood $V$ of $x$ in $X$ and a linear representation $\h$ of $W$ such that $(V,x) \hookrightarrow (\h,0)$ is a strongly $\bc$-regular subscheme (for any $\bc$). 
\end{prop}

\begin{proof}
Let $\h'=T_xX$. Then there is a $W$-equivariant regular map $f \colon X\rightarrow \h'$ such that $df_x= \mathrm{Id}$. Namely take any (not necessarily $W$-equivariant) map with this property and average it over $W$. 
By replacing $X$ with a small enough $W$-stable affine neighbourhood of $x$, we can make sure that:
\begin{enumerate}
\item $f$ is \'etale;
\item $f$ separates points on each $W$-orbit, see Lemma \ref{lem:fseperate};
\item all reflection hypersurfaces of $X$ pass through $x$.
\end{enumerate} 
Indeed, $f$ is unramified by construction and hence \'etale because $X$ is smooth at $x$. 

 Let $(y_1,...,y_n)$ be some generators of $\mc{O}(X/W)$, and let $F: X\rightarrow \h'\times \mathbb{A}^n_{\kk}$ be the $W$-equivariant map given by $F=(f,y)$, where $y=(y_1,...,y_n)$ and $W$ acts trivially on the $y_i$. Then $F$ is injective (separates points of $X$), since $y$ separates $W$-orbits, and by (2) $f$ separates points inside each $W$-orbit. It is finite; in fact, already $y$ is finite, since by Hilbert-Noether $\mc{O}(X)$ is module-finite over $\mc{O}(X/W)$. Also, $d F_b$ is injective for all $b\in X$ since $d f_b$ is an isomorphism (because $f$ is \'etale by (1)). Thus, $F$ is a closed embedding \cite[II,  Lemma~7.4]{Hartshorne}. Moreover, all reflection hypersurfaces of $X=F(X)$ are intersections with $X$ of reflection hypersurfaces in $\h := \h'\times \mathbb{A}^n_{\kk}$.
\end{proof}

\begin{lem}\label{lem:fseperate}
If $f : X \rightarrow Y$ is a $W$-equivariant \'etale morphism and $x \in X$ with $\Stab_W(x) = W$ then there is some neighbourhood of $x$ such that $f$ separates points in each $W$-orbit. 
\end{lem}

\begin{proof}
Let $Q$ be the subset of $y\in X$ such that there exists $z\in X$ with $y=wz$ for some $w\in W$ and $f(y)=f(z)$ but $y\ne z$. Then $Q$ is constructible by Chevalley's Theorem (elimination of quantifiers) which says that the image of a morphism is constructible. Indeed, $Q$ is the image under the first projection of the set $Q'\subset X\times X$ of pairs $(y,z)$ such that $f(y)=f(z)$, $y=wz$ for some $w\in W$, but $y\ne z$, which is defined by finitely many equations and inequalities, so constructible. We wish to show that $x \in X \setminus \overline{Q}$. 

Since all data involved ($X,Y,f$, the $W$-action, and the point $x$) are of finite type, they are defined over some finitely generated subfield $k_0 \subset \kk$. After enlarging $k_0$, we may assume that $x$ is $k_0$-rational. Choose an embedding $k_0\hookrightarrow \C$. The morphism
$f_{\C}\colon X_{\C}\to Y_{\C}$ is still $W$-equivariant and \'etale at
$x_{\C}$. Notice that $Q_{\C}$ does not intersect a small $W$-invariant analytic neighbourhood $B$ of $x$, since for $y\in B$, one has $wy\in B$, and $f|_B$ is injective. Since $Q_{\C}$ is constructible and $W$-stable, this means that $x_{\C} \in X_{\C}\setminus \overline{Q_{\C}}$.  

Therefore, there exists an open neighborhood $U$ of $x$ disjoint from $Q$.
\end{proof}

\subsection{Duality}

Let $\mathrm{D}^b_{\mathrm{Coh}}(\sH_{\omega,\bc}(X,W))$ denote the derived category of bounded complexes of $\sH_{\omega,\bc}(X,W)$-modules with coherent cohomology. Recall from \eqref{eq:omegabarrightleft} that we have defined the parameters $\overline{\omega},\overline{\bc}$. As in \cite[Section~4.7]{ThompsonHolI}, we define 
$$
\D \colon \mathrm{D}^b_{\mathrm{Coh}}(\sH_{\omega,\bc}(X,W)) \to \mathrm{D}^b_{\mathrm{Coh}}(\sH_{\overline{\omega},\overline{\bc}}(X,W))
$$
by 
$$
\D (\ms{M}) := R \mc{H}om_{\sH_{\omega,\bc}(X,W)}(\ms{M},\sH_{\omega,\bc}(X,W) \o_{\mc{O}_X} K_X^{-1})[n],
$$
where $n:= \dim X$. Then $\D^2 \simeq \mr{Id}$ since \eqref{eq:omegabarrightleft} implies that $\overline{(\overline{\omega})} = \omega$ and $\overline{(\overline{\bc})} = \bc$. 

\begin{remark}
As for $\dd$-modules, we do not define $\D$ on the derived category of complexes with quasi-coherent cohomology since the dual of a quasi-coherent module will not be quasi-coherent in general.
\end{remark}

As explained in \cite[Proposition~2.1]{ThompsonHolI}, the following is a consequence of standard results on almost commutative algebras with regular associated graded algebra, as summarized in \cite[Appendix D.4]{HTT}. 

\begin{lem}\label{lem:ASregularcondition}
	Let $\ms{M}$ be a coherent $\sH_{\omega,\bc}(X,W)$-module. Then, 
	\begin{enumerate}
		\item[(i)] $\SS (\D(\ms{M})) = \SS(\ms{M})$,
		\item[(ii)] $\dim \SS(\mathcal{H}^i(\D(\ms{M}))) \le n- i$ for $i \in [-n,n]$ and $\mathcal{H}^i(\D(\ms{M})) = 0$ for $i \notin [-n,n]$, 
		\item[(iii)] $\mathcal{H}^i(\D(\ms{M})) = 0$ for $i < n - \dim \SS(\ms{M})$.
	\end{enumerate} 
Moreover, for any complex $\ms{M}^{\idot}$ with coherent cohomology, $\D(\ms{M}^{\idot}[k]) = \D(\ms{M}^{\idot})[-k]$, where $\ms{M}^{\idot}[k]^i = \ms{M}^{i+k}$ (and one changes the sign of the differential appropriately). 
\end{lem}

Let $\mathrm{D}^b_{\mathrm{Hol}}(\sH_{\omega,\bc}(X,W))$ denote the bounded derived category of complexes of $\sH_{\omega,\bc}(X,W)$-modules with holonomic cohomology. 

\begin{lem}
	The functor $\D$ restricts to a duality $\mathrm{D}^b_{\mathrm{Hol}}(\sH_{\omega,\bc}(X,W)) \to \mathrm{D}^b_{\mathrm{Hol}}(\sH_{\overline{\omega},\overline{\bc}}(X,W))$. 
	
	It is left exact with respect to the standard $t$-structure on $\mathrm{D}^b_{\mathrm{Hol}}(\sH_{\omega,\bc}(X,W))$. In particular, $\mc{H}^i(\D(\ms{M})) = 0$ for all $i < 0$ if $\ms{M}$ is a holonomic module.  
\end{lem}

\begin{proof}
	The first part follows from Lemma~\ref{lem:ASregularcondition}(i) and the second from Lemma~\ref{lem:ASregularcondition}(iii).
\end{proof}

\section{Coherent and holonomic extensions}\label{sec:coherentext}

In this section we show that coherent, resp. holonomic, modules admit coherent, resp. holonomic, extensions from $W$-invariant open subsets of $X$.

\subsection{Coherent extensions}

We begin with some elementary results, whose proofs follow closely the standard arguments given in the case of $\dd$-modules. First, let $U \subset X$ be a $W$-stable open subset and assume that $\ms{M}$ is a quasi-coherent $\mc{O}_X \rtimes W$-module. If we choose a coherent $\mc{O}_U \rtimes W$-submodule $\ms{F}$ of $\ms{M} |_U$ then it is well-known, \cite[Proposition 1.4.16]{HTT}, that we can choose a coherent $\mc{O}_X$-submodule $\ms{F}' \subset \ms{M}$ such that $\ms{F}' |_U = \ms{F}$. Replacing $\ms{F}'$ by the image of $\mc{O}_X \rtimes W \o_{\mc{O}_X} \ms{F}'$ under the action map if necessary, we may assume that $\ms{F}'$ is actually a coherent $\mc{O}_X \rtimes W$-submodule.      

\begin{lem}\label{lem:coherentOXWsubmodulegen}
	Let $\ms{M}$ be a coherent $\sH_{\omega,\bc}(X,W)$-module. Then there exists a coherent $\mc{O}_X \rtimes W$-submodule $\ms{F}$ of $\ms{M}$ that generates $\ms{M}$. That is, the action map $\sH_{\omega,\bc}(X,W) \o_{\mc{O}_X} \ms{F} \to \ms{M}$ is surjective. 
\end{lem}

\begin{proof}
	As usual, we cover $X$ by $W$-stable affine open subsets $U_i$ for $1 \le i \le k$. Then $\ms{M} |_{U_i}$ is the localization of a finitely generated $\H_{\omega,\bc}(U_i,W)$-module $M_i$. Let $G_i$ be a finitely generated $\mc{O}(U_i) \rtimes W$-submodule that generates $M_i$ over $\H_{\omega,\bc}(U_i,W)$. As noted above, there exists a coherent $\mc{O}_X \rtimes W$-submodule $\ms{F}_i$ of $\ms{M}$ whose restriction to $U_i$ equals $\widetilde{G}_i$. Then $\ms{F} = \sum_i \ms{F}_i$ is the required $\mc{O}_X \rtimes W$-submodule of $\ms{M}$.
\end{proof}

\begin{lem}\label{lem:coherentextend}
Let $U \subset X$ be a $W$-stable open subset and $\ms{M}$ a $\sH_{\omega,\bc}(X,W)$-module. Given a coherent submodule $\ms{N} \subset \ms{M} |_U$, there exists a coherent submodule $\ms{N}' \subset \ms{M}$ such that $\ms{N}' |_U = \ms{N}$. 
\end{lem}

\begin{proof}
By Lemma~\ref{lem:coherentOXWsubmodulegen}, we can choose a coherent $\mc{O}_U \rtimes W$-submodule $\ms{F}$ of $\ms{N}$ that generates this sheaf over $\sH_{\omega,\bc}(U,W)$. Let $\ms{F}'$ be a coherent $\mc{O}_X \rtimes W$-submodule of $\ms{M}$, whose restriction to $U$ equals $\ms{F}$. Then $\ms{N}'$ can be chosen to be the (coherent) $\sH_{\omega,\bc}(X,W)$-submodule of $\ms{M}$ generated by $\ms{F}'$.   
\end{proof}

\begin{lem} \label{lem:irredsupp}
Let $\ms{M}$ be an irreducible $\sH_{\omega,\bc}(X,W)$-module. Then the support of $\ms{M}$ over $\mc{O}_{X/W}$ is irreducible. 
\end{lem}

\begin{proof}

Let $Z$ be an irreducible component of the support of $\ms{M}$. We can show $Z = \Supp \ms{M}$ locally. Therefore, replacing $X$ by a $W$-stable affine open covering, we assume that $X$ is affine. Let $M$ be the space of global sections of $\ms{M}$ and let $\mf{p}$ be the prime ideal in $\mc{O}(X)^W$ defining $Z$. By Lemma~\ref{lem:coherentOXWsubmodulegen}, we may choose a finitely generated $\mc{O}(X)^W$-submodule $F$ of $M$ that generates $M$ over $\H_{\omega,\bc}(X,W)$. Then the $\mc{O}(X)^W$-submodule $N$ of $M$ consisting of all sections killed by some power of $\mf{p}$ is non-zero. Indeed, it follows from \cite[{Proposition~02CE}]{stacks} that $\mf{p}$ is an associated prime of $F$ and thus, by definition, the subspace of $F$ consisting of sections killed by $\mf{p}$ is non-zero. Since the adjoint action of $\mc{O}(X)^W$ on $\H_{\omega,\bc}(X,W)$ is nilpotent, $N$ is a $\H_{\omega,\bc}(X,W)$-submodule of $M$. We deduce that $M = N$, which implies that $Z$ equals the support of $\ms{M}$. 
\end{proof}

Let $X$ be affine and $U \subset X$ an affine open subset. Write $d(N)$ for the Gelfand-Kirillov dimension of a module $N$. This equals $\dim \SS(N)$. Recall that $n = \dim X$. 

\begin{thm}\label{thm:GKextension}
Let $M$ be a $\H_{\omega,\bc}(X,W)$-module and $N \subset M |_U$ a finitely generated submodule. There exists a finitely generated submodule $N' \subset M$ such that $N' |_U = N$ and $d(N') = d(N)$. 
\end{thm}

\begin{proof}
By Lemma~\ref{lem:coherentextend}, there exists some coherent extension $M' \subset M$ of $N$. Replacing $M$ by $M'$, we may assume without loss of generality that $M$ is coherent and $M |_U  = N$. Let $\ell = 2n - d(N)$. We take $L^{\idot} = \tau^{\ge \ell - n} \D (M)$. We will show that $N' := H^0(\D(L^{\idot}))$ is a submodule of $M$ with $d(N') = d(N)$ and $N' |_U = N$. 

First we show that $N'$ is an extension of $N$ with $d(N')  = d(N)$. Then we will show that $N'$ embeds in $M$. Notice that $L^{\idot} |_U \cong \tau^{\ge \ell - n} \D(N) \cong \D(N)$ since $H^i(\D(N)) = 0$ for all $i < \ell - n$ by Lemma~\ref{lem:ASregularcondition}(iii). Therefore $\D(L^{\idot}) |_U \cong N$ and hence $N' |_U \cong N$. Next, 
$$
\SS(N') \subset \SS(\D(L^{\idot})) = \SS(L^{\idot}).
$$
Since $M$ is coherent, Lemma~\ref{lem:ASregularcondition}(ii) says that $\dim \SS(H^i(\D(M))) \le n- i$ for $i \in [-n,n]$ and is zero otherwise. Hence $\dim \SS(\tau^{\ge \ell - n} \D(M)) \le 2n - \ell = d(N)$. On the other hand, since $N' |_U = N$, we clearly have $\dim \SS(N') \ge d(N)$. Thus, $d(N') = d(N)$. 

Therefore, we just need to show that $N'$ embeds in $M$. The proof is identical to that of \cite[Proposition~3.1.7]{HTT}, but we include it for the reader's benefit. Consider the distinguished triangle 
$$
K^{\idot} \rightarrow \D(M) \rightarrow L^{\idot} \stackrel{+1}{\longrightarrow}
$$
where $K^{\idot} = \tau^{< \ell - n} \D(M)$. Dualizing, we get the distinguished triangle 
$$
\D(L^{\idot}) \rightarrow M \rightarrow \D(K^{\idot}) \stackrel{+1}{\longrightarrow}
$$
hence we get an exact sequence $H^{-1}(\D(K^{\idot})) \rightarrow N' \rightarrow M$ and it suffices to show that $H^{-1}(\D(K^{\idot})) = 0$. We will show that 
\begin{equation}\label{eq:vanishone}
H^i(\D(\tau^{\ge -k} K^{\idot})) = 0, \quad \forall \ i < 0, k > n - \ell. 
\end{equation}
Then taking $k \gg 0$, we have $\tau^{\ge - k} K^{\idot} = K^{\idot}$. 

For $k > n - \ell$, we have $H^{-k}(K^{\idot}) \cong H^{-k}(\D(M))$ and hence $\dim \SS(H^{-k}(K^{\idot})) \le n + k$. This implies that $H^j(\D(H^{-k}(K^{\idot}))) = 0$ for $j < -k$ and hence
\begin{equation}\label{eq:vanishtwo}
H^{i}(\D(H^{-k}(K^{\idot})[k])) = H^{i}(\D(H^{-k}(K^{\idot}))[-k]) = H^{i-k}(\D(H^{-k}(K^{\idot}))) = 0, \quad \forall \ i < 0. 
\end{equation}

Now we prove (\ref{eq:vanishone}) by induction on $k$. When $k = n - \ell +1$, we have 
$$
\tau^{\ge -k} K^{\idot} = \tau^{ \ge \ell - n - 1} \tau^{< \ell - n} \D(M) = H^{-k}(K^{\idot})[k]
$$
and hence $H^i(\D(\tau^{\ge - k} K^{\idot})) = H^i(\D(H^{-k}(K^{\idot})[k])) = 0$. Assume $k \ge n - \ell + 2$. Then we apply $\D$ to the distinguished triangle
$$
H^{-k}(K^{\idot})[k] \rightarrow \tau^{\ge -k} K^{\idot} \rightarrow \tau^{\ge -(k-1)} K^{\idot} \stackrel{+1}{\longrightarrow}
$$
to get the distinguished triangle
$$
\D(\tau^{\ge -(k-1)} K^{\idot}) \rightarrow \D(\tau^{\ge -k} K^{\idot}) \rightarrow  \D(H^{-k}(K^{\idot})[k]) \stackrel{+1}{\longrightarrow}
$$
and the assertion follows from (\ref{eq:vanishtwo}) by induction. 
\end{proof}

The above result shows that a finitely generated module of given GK-dimension admits an extension of the same GK-dimension. Ideally, we would like to say that a holonomic module on $U$ admits a holonomic extension to $X$. If this is the case then we say that the extension property holds for holonomic modules. The remainder of Section~\ref{sec:coherentext} is devoted to describing cases where the extension property for holonomic modules can be shown to hold.

\subsection{Quantum Hamiltonian reduction}\label{sec:QHR}

In this section, we give a proof that the extension property holds for holonomic modules over algebras constructed by quantum Hamiltonian reduction. This is potentially applicable to algebras other than Cherednik algebras. 

Let $G$ be a connected reductive group and let $Z$ be a smooth affine $G$-variety. Let $\mu \colon T^* Z \to \mf{g}^*$ be the moment map for the induced Hamiltonian action on $T^* Z$. We fix a character $\chi \colon \mf{g} \to \kk$ and consider the category $\Coh{\dd,G,\chi}$ of coherent $(G,\chi)$-equivariant $\dd$-modules on $Z$; see \cite[\S~4.2]{BernsteinLosev} or \cite[\S~2.2]{BellBoosSemisimple}. We denote the quantum comoment map $\mu^* \colon \mf{g} \to \dd(Z)$. We can form the quantum Hamiltonian reduction 
\[
U_{\chi}(Z) = (\dd(Z) / \dd(Z) (\mu^* - \chi)(\mf{g}))^G.
\] 
The order filtration on $\dd(Z)$ induces a filtration $\mc{F}_{\idot}$ on $U_{\chi}(Z)$. Let $Y = \Spec \gr_{\mc{F}} \, U_{\chi}(Z)$. Note that there is a surjection of Poisson algebras $\mc{O}(\mu^{-1}(0) \git \, G) \to \gr_{\mc{F}} \, U_{\chi}(Z)$. As noted in \cite[\S~2.3]{BernsteinLosev}, the Poisson scheme $\mu^{-1}(0) \git \, G$ has finitely many symplectic leaves and hence so too does $Y$, since the latter is a closed Poisson subscheme. The algebra $U_{\chi}(Z)$ forms a sheaf on $Z \git \, G$. We say that a $U_{\chi}(Z)$-module $\ms{M}$ is holonomic if $\SS(\ms{M}) \cap \mc{L}$ is isotropic in $\mc{L}$ for all leaves $\mc{L}$ of $Y$.  

\begin{thm}\label{thm:QHRholonomic}
	The functor of Hamiltonian reduction and its adjoint $\Ham^{\perp}$ restrict to functors 
	\[
	\Ham \colon \Hol{\dd,G,\chi} \to \Hol{U_{\chi}(Z)}, \quad \Ham^{\perp} \colon \Hol{U_{\chi}(Z)} \to \Hol{\dd,G,\chi}. 
	\]
\end{thm}

\begin{proof}
	It is already explained in \cite[Section~4.2]{BernsteinLosev} that $\Ham^{\perp}$ sends holonomic modules to holonomic modules. Therefore, we must check that $\Ham(\ms{N})$ is a holonomic $U_{\chi}(Z)$-module if $\ms{N}$ is a holonomic $(G,\chi)$-equivariant $\dd$-module. If we fix a $G$-stable good filtration $\mc{F}_{\idot}$ of $\ms{N}$ then $\mc{F}_{\idot}^G$ is a good filtration on $\ms{M} = \Ham(\ms{N})$ and hence the singular support of $\ms{M}$ equals the image of $\SS(\ms{N}) \subset \mu^{-1}(0)$ under the quotient map $\pi \colon \mu^{-1}(0) \to \mu^{-1}(0) \git \, G$. Thus, it follows from Lemma~\ref{lem:isotropicimage} below that $\SS(\ms{M}) = \pi(\SS(\ms{N}))$ is isotropic in $\mu^{-1}(0) \git \, G$.   
\end{proof}

\begin{lem}\label{lem:isotropicimage}
	Let $S \subset \mu^{-1}(0)$ be a $G$-stable closed subvariety that is isotropic as a subvariety of $T^* Z$. Then $\pi(S)$ is isotropic in $\mu^{-1}(0) \git \, G$. 
\end{lem}

\begin{proof}
	Considering each irreducible component of $S$ in turn, we may assume that $S$ is irreducible. Let $C = \pi(S)$, an irreducible closed subset of $\mu^{-1}(0) \git \, G$. There exists a unique leaf $\mc{L}$ such that $\mc{L} \cap C$ is open dense in $C$. We must show that $\mc{L} \cap C$ is isotropic in $\mc{L}$. As explained in \cite[\S~2.2]{BernsteinLosev}, the symplectic leaves in $\mu^{-1}(0) \git \, G$ are precisely the connected components of the strata in the stabilizer stratification; see \cite[\S~2.2]{BernsteinLosev} for the definition of this stratification. Let $\mc{L}_0 = \pi^{-1}(\mc{L}) \subset \mu^{-1}(0)$. Note that $\mc{L}_0 \cap S$ is open dense in $S$. If $\omega$ is the symplectic form on $T^* Z$ and $\omega_{\mc{L}}$ the symplectic form on the leaf $\mc{L}$ then the symplectic slice theorem, as stated in Equation (2) of \cite{BernsteinLosev}, implies that $\pi |_{\mc{L}_0} \colon \mc{L}_0 \to \mc{L}$ is a smooth morphism such that $\omega |_{\mc{L}_0} = \pi^* \omega_{\mc{L}}$. Therefore, if we choose $s \in S_{\reg}$ such that $c = \pi(s) \in C \cap \mc{L}$ is a smooth point of the intersection and $v_1, v_2 \in T_c (C \cap \mc{L})$ then we can choose $w_1, w_2 \in T_s S$ such that $v_i = (d \pi)_s(w_i)$ and hence 
	\[
	\omega_{\mc{L}}(v_1,v_2) = \omega(w_1,w_2) = 0,
	\]
	because $S$ is isotropic. 
\end{proof}

As explained in \cite{BernsteinLosev}, the following is an immediate consequence of Theorem~\ref{thm:QHRholonomic}. 

\begin{cor}\label{cor:QHRfinitelength}
	Holonomic $U_{\chi}(Z)$-modules have finite length. 
\end{cor}

\begin{proof}
	Recall that the category of holonomic $\dd$-modules has finite length. Then the corollary follows from Theorem~\ref{thm:QHRholonomic}, together with the fact that $\Ham$ is an exact quotient functor. 
\end{proof}

\begin{cor}\label{cor:extensionQHR}
	Let $V \subset Z \git\,  G$ be an affine open set and $\ms{M}$ a holonomic $U_{\chi}(V)$-module. Then there exists a holonomic extension $\ms{M}'$ of $\ms{M}$ to $Z \git \, G$, which can be chosen to be irreducible if $\ms{M}$ is irreducible. 
\end{cor}

\begin{proof}
Let $V'$ be the preimage of $V$ in $Z$. Let $\ms{N} = \Ham^{\perp}(\ms{M})$, a holonomic $(G,\chi)$-equivariant $\dd$-module on $V'$. We claim that there exists a holonomic $(G,\chi)$-equivariant $\dd$-module $\ms{N}'$ on $Z$ extending $\ms{N}$. Indeed, if $j \colon V' \hookrightarrow Z$ is the open embedding then $j_0 \ms{N}$ is a holonomic $(G,\chi)$-equivariant $\dd$-module and $\Ham(j_0 \ms{N}) |_V \cong \ms{M}$. 

Now assume $\ms{M}$ is irreducible. For a connected group $G$, every subquotient of a $(G,\chi)$-equivariant $\dd$-module is $(G,\chi)$-equivariant. In particular, if $\ms{M}$ is irreducible, then there exists an irreducible subquotient $\ms{N}_1$ of $\ms{N}$ with $\Ham(\ms{N}_1) \cong \ms{M}$. Then the minimal extension $j_{*!} \mc{N}_1 \subset j_0 \ms{N}_1$ is also $(G,\chi)$-equivariant. We can take $\ms{M}' = \Ham(j_{*!} \mc{N}_1)$ to be our extension of $\ms{M}$; since $j_{*!} \mc{N}_1$ is irreducible, so too is $\ms{M}'$ because an exact quotient functor sends a simple object to either zero or a simple object. 
\end{proof}

\subsection{Holonomic extensions - linear case}

In this section we explain how the results of the previous section show that the extension property for holonomic modules holds for most complex reflection groups. Let $(\h,W)$ be a complex reflection group and $\bc \in \mc{S}(\h)^W$. 

\begin{lem}\label{lem:holonomicfinitegroupext}
	Let $W^{\circ}$ be the (normal) subgroup of $W$ generated by all reflections $s$ with $\bc(s) \neq 0$. Let $f \in \mc{O}(\h)^W$ be non-zero. 
    \begin{enumerate}
        \item[(i)] Holonomic $\H_{\bc}(\h,W)$-modules have finite length if and only if holonomic $\H_{\bc}(\h,W^{\circ})$-modules have finite length. 
        \item[(ii)] Let $M$ be an irreducible holonomic $\H_{\bc}(D(f),W)$-module and $M' \subset M |_{\H_{\bc}(D(f),W^{\circ})}$ an irreducible (holonomic) submodule. Then there exists an irreducible holonomic $\H_{\bc}(\h,W)$-extension of $M$ if and only if there exists a holonomic $\H_{\bc}(\h,W^{\circ})$-extension of $M'$. 
    \end{enumerate}
\end{lem}

\begin{proof}
	Note that $\H_{\bc}(\h,W^{\circ}) \subset \H_{\bc}(\h,W)$ and $\H_{\bc}(\h,W)$ is a free $\H_{\bc}(\h,W^{\circ})$-module of rank $|W/W^{\circ}|$. Using Lemma~\ref{lem:isoclosed}(i), one can show that if $N$ is a $\H_{\bc}(\h,W)$-module then it is holonomic if and only if $N |_{\H_{\bc}(\h,W^{\circ})}$ is holonomic. Part (i) follows from these two facts. 

Part (ii). Assume that $N'$ is an irreducible holonomic extension of the $\H_{\bc}(D(f),W^{\circ})$-module $M'$. Then $N'' = \Ind_{\H_{\bc}(\h,W^{\circ})}^{\H_{\bc}(\h,W)} N'$ is a holonomic $\H_{\bc}(\h,W)$-module such that $N'' [f^{-1}]$ surjects onto $M$ since the latter is irreducible. Moreover, $N''$ has finite length since $N'$ is irreducible. Therefore, there is an irreducible submodule $N$ of $N''$ such that $N[f^{-1}] \twoheadrightarrow M$. Since $N$ is irreducible and $N[f^{-1}] \neq 0$ the composition $N \to N[f^{-1}] \to M$ is non-zero and thus injective. In particular, $N[f^{-1}] \cong M$ since $M = M[f^{-1}]$. 

Conversely, if $N$ is an irreducible extension of $M$ then $N |_{\H_{\bc}(\h,W^{\circ})}$ will decompose into a direct sum $N_1 \oplus \cdots \oplus N_k$ of irreducible (not necessarily isomorphic) $\H_{\bc}(\h,W^{\circ})$-modules; see Remark~\ref{rem:Grzeszczuk}. Then $M = N[f^{-1}] = N_1[f^{-1}] \oplus \cdots \oplus N_k[f^{-1}]$, with each $N_i[f^{-1}]$ an irreducible $\H_{\bc}(D(f),W^{\circ})$-module, non-zero since $N$ has no $f$-torsion. One of these summands will give (up to isomorphism) the extension of $M'$.  
\end{proof}

We refer the reader to Definition~\ref{defn:bcregularRCA} of the appendix in order to recall what it means for a parameter $\bc$ for $(\h,W)$ to be regular. The parameter $\bc$ is said to be \textit{spherical} if for any $\H_{\bc}(\h,W)$-module $M$, $e M = 0$ implies that $M = 0$, where $e \in \kk W$ is the trivial idempotent. The following condition allows us to use the results on quantum Hamiltonian reduction in the previous section. 

\begin{condition}\label{cond:good1}
 For each irreducible factor $(\h_i,W_i)$ of $(\h,W)$ either:
\begin{enumerate}
    \item[(i)] $(\h_i,W_i) \cong (\kk^n,G(m,p,n))$ with $n \neq 2$ if $p > 1$ and $\bc |_{W_i}$ is spherical; or
    \item[(ii)] $\bc |_{W_i}$ is regular. 
\end{enumerate}
\end{condition}

\begin{thm}\label{thm:strongextension}
Assume $(\h,W,\bc)$ satisfies Condition~\ref{cond:good1}. Let $f \in \mc{O}(\h)^W$ be non-zero and $M$ an irreducible holonomic $\H_{\bc}(D(f),W)$-module. Then, there exists an irreducible holonomic $\H_{\bc}(\h,W)$-module $N$ with $N[f^{-1}] \cong M$. 
\end{thm}

\begin{proof}
First, we may embed $W$ into a larger complex reflection group $K$ such that if $G(m,p,n)$, with $n \neq 2$ if $p > 1$, is an irreducible factor of $W$ then the corresponding irreducible factor of $K$ is $G(m,1,n)$. Otherwise, the irreducible factors of $W$ and $K$ agree. Then $K$ also acts on $\h$ and $W \lhd K$. Moreover, the restriction that $n \neq 2$ if $p > 1$ ensures that $\bc$ extends to a parameter on $K$ such that $\bc(s) = 0$ for all reflections $s \in K \setminus W$; see \cite{CMpartitions}. Then, Lemma~\ref{lem:holonomicfinitegroupext} implies that we may replace $W$ by $K$. Moreover, note that Condition~\ref{cond:good1} also holds for $(\h,K,\bc)$. Thus, we may assume that for each irreducible factor $W_i$ of $W$, either $W_i \cong G(m,1,n)$ and $\bc |_{W_i}$ is spherical, or $\bc |_{W_i}$ is regular. 

We decompose $W = W_1 \times W_2$ acting on $\h = \h_1 \oplus \h_2$, where each irreducible factor in $W_1$ is of the form $G(m,1,n)$ with $\bc |_{W_1}$ spherical and $\bc |_{W_2}$ is regular. We note that being regular is a stronger condition than being spherical therefore if $(\h,W,\bc)$ satisfies Condition~\ref{cond:good1} then $\H_{\bc}(\h,W)$ is Morita equivalent to $e \H_{\bc}(\h,W) e$ and this equivalence preserves holonomicity. Therefore, it suffices to consider $e \H_{\bc}(\h,W) e$. Moreover, if $A = e_2 \H_{\bc}(\h_2,W_2) e_2$ then $e \H_{\bc}(\h,W) e \cong e_1 \H_{\bc}(\h_1,W_1) e_1 \otimes A$, where $e_i$ is the trivial idempotent in $\kk W_i$. 

Since every irreducible factor of $(\h_1,W_1)$ is isomorphic to $(\kk^n,G(m,1,n))$, for some $m,n \ge 1$, it is well-known \cite{OblomkovHC,GordonCyclicQuiver} that there exists a connected reductive group $G$ and representation $Z$ such that for each parameter $\bc$ there is a character $\chi$ and filtered isomorphism
	\[
	e_1 H_{\bc}(\h_1,W_1) e_1 \cong U_{\chi}(Z).
	\]
    Since $\bc$ is regular on $W_2$, every finitely generated $A$-module has GK-dimension at least $\dim \h_2$. Then it is shown in \cite[Proposition~3.7]{ThompsonHolI} that a module $M$ is holonomic if and only if its GK-dimension equals $\dim \h_2$. To extend the results of Section~\ref{sec:QHR} to this setting, we consider the category $\Hol{\dd \otimes A,G,\chi}$ of holonomic $(G,\chi)$-equivariant $\dd \otimes A$-modules, where $G$ acts trivially on $A$. Here holonomicity just means that the module has GK-dimension $\dim (Z \times \h_2)$. We think of $f$ as a $G$-invariant polynomial on $Z \times (\h_2/W_2)$ via the isomorphism $\mc{O}(Z \times \h_2)^{G \times W_2} \cong \mc{O}(\h_1 \oplus \h_2)^{W_1 \times W_2}$. Note that $\dd(Z) \otimes A$ has finite global dimension since $A$ is Morita equivalent to $\H_{\bc}(\h_2,W_2)$.  Then Theorem~\ref{thm:GKextension}, whose proof goes through verbatim for $\dd(Z) \otimes A \hookrightarrow (\dd(Z) \otimes A)[f^{-1}]$, implies that if $M'$ is an irreducible $(\dd(Z) \otimes A)[f^{-1}]$-module of GK-dimension $\dim (Z \times \h_2)$ then there exists a $\dd(Z) \otimes A$-module $N'$ of GK-dimension $\dim (Z \times \h_2)$ with $N'[f^{-1}] \cong M'$. Since $\dd(Z) \otimes A$ is an almost commutative algebra in the sense of \cite[\S~8.4.2]{MR}, it is standard (see the proof of \cite[Corollary~8.5.7]{MR}) that such modules have finite length. Therefore, we may assume that $N'$ is irreducible. In particular, $N' \hookrightarrow N'[f^{-1}] = M'$. Since the latter is $(G,\chi)$-equivariant and $G$ is connected, $N'$ is also $(G,\chi)$-equivariant.  

    Finally, if $M$ is an irreducible holonomic $e \H_{\bc}(\h,W) e[f^{-1}] \cong (U_{\chi}(Z) \otimes A)[f^{-1}]$-module then by surjectivity of Hamiltonian reduction, we can pick an irreducible $M' \in \Hol{(\dd \otimes A)[f^{-1}],G,\chi}$ with $\Ham(M') = M$. Then $N = \Ham(N')$ is our desired irreducible holonomic extension of $M$. 	
\end{proof}

\begin{prop}\label{prop:holonomiclengthlinear}
If Condition~\ref{cond:good1} holds then holonomic $\H_{\bc}(\h,W)$-modules have finite length.  
\end{prop}

\begin{proof}
We continue with the setting of the proof of Theorem~\ref{thm:strongextension}. The fact that the functor $\Ham$ of Hamiltonian reduction is an exact quotient functor implies that it suffices to show that modules in $\Hol{\dd \otimes A,G,\chi}$ have finite length. But, again, any finitely generated $\dd(Z) \otimes A$-module has GK-dimension at least $\dim (Z \times \h_2)$ and the modules in $\Hol{\dd \otimes A,G,\chi}$ all have GK-dimension exactly $\dim (Z \times \h_2)$. Since $\dd(Z) \otimes A$ is an almost commutative algebra in the sense of \cite[\S~8.4.2]{MR}, it is standard (see the proof of \cite[Corollary~8.5.7]{MR}) that such modules have finite length.
\end{proof}

\subsection{Holonomic extensions - global case} In this section, we describe the conditions under which we can show that a holonomic module on a principal open subset $D(f)$ of $X$ admits a holonomic extension to $X$. As stated in the introduction, we \textit{expect} that the results of this section hold for any parameters $(\omega,\bc)$. For each $x \in X$ and $\bc \in \mc{S}(X)^W$, let $W_x^{\circ} \subset W_x$ be the subgroup of the stabilizer group $W_x$ generated by all elements $s$ that act as reflections on $T_x X$.

\begin{condition}\label{cond:good2}
For each $x \in X$, the triple $(T_x X, W_x^{\circ},\bc |_{W_x^{\circ}})$ satisfies Condition~\ref{cond:good1}. 
\end{condition}

In this section, our goal is to prove:

\begin{thm}\label{thm:holonomicextension}
	Assume Condition~\ref{cond:good2} holds and that $\omega \in \mathbb{H}^2 ( X,\Omega_X^{1,2})^W$ is Zariski locally trivial. Let $f \in \mc{O}(X)^W$ and $M$ an irreducible holonomic $\H_{\omega,\bc}(D(f),W)$-module. Then, for each $x \in X$, there exists a $W$-stable affine open neighbourhood $V$ of $x$ in $X$ and an irreducible holonomic $\H_{\omega,\bc}(V,W)$-module $N$ with $N[f^{-1}] \cong M |_{D(f) \cap V}$.  
\end{thm}

\begin{proof}
First, we assume that $x \in X^W$. Replacing $X$ by a $W$-stable affine open neighbourhood of $x$ we may assume that $\omega = 0$. By Proposition~\ref{prop:goodembedding}, we can pick $V$ such that there exists a strongly $\bc$-regular closed embedding $i \colon (V,x) \hookrightarrow (\h, W)$, where $\h = T_x X \times \mathbb{A}^n$ for some $n \ge 0$ with $W$ acting trivially on $\mathbb{A}^n$. Since we have assumed Condition~\ref{cond:good2}, Condition~\ref{cond:good1} holds for $(\h,W,\bc)$. Pick $F \in \mc{O}(\h)^W$ whose restriction to $V$ equals $f$. Then $D(F) \cap V = D(f)$ and if $I \colon D(f) \hookrightarrow D(F)$ is the restriction of $i$ to $D(f)$ then there is a commutative diagram 
\begin{equation}\label{eq:commuteholopullback}
\begin{tikzcd}
 \Qcoh{\H_{\bc}(\h,W)}_{V} \ar[r,"i^{\natural}"] \ar[d,"( - )|_{D(F)}"'] & \Qcoh{\H_{\bc}(V,W)} \ar[d,"(-)|_{D(f)}"] \\
  \Qcoh{\H_{\bc}(D(F),W)}_{D(f)} \ar[r,"I^{\natural}"] & \Qcoh{\H_{\bc}(D(f),W)} \\
\end{tikzcd}
\end{equation}
and Theorem~\ref{thm:Kash} says that the horizontal functors are equivalences. By Corollary~\ref{cor:holoclosedpushforward}, the functors $i^{\natural}$ and $I^{\natural}$ preserve the full subcategories of holonomic modules. If $M$ is an irreducible holonomic $\H_{\bc}(D(f),W)$-module then $I_0(M)$ is an irreducible holonomic $\H_{\bc}(D(F),W)$-module because $I_0$ is an equivalence. Therefore, since Condition~\ref{cond:good1} holds for $(\h,W,\bc)$, Theorem~\ref{thm:strongextension} says that there exists an irreducible holonomic $\H_{\bc}(\h,W)$-module $N'$ such that $N'[F^{-1}] = I_0(M)$. Then $N = i^{\natural}(N')$ is an irreducible holonomic $\H_{\bc}(V,W)$-module and the commutativity of \eqref{eq:commuteholopullback} implies that $N[f^{-1}] \cong M$. 

Now we consider the general situation. Let $V'$ be an affine open neighbourhood of $x$ such that the stabilizer of every point in $V'$ is contained in $\Pa := W_x$ and $\omega |_{V'} = 0$. Replacing $V'$ by a smaller $\Pa$-stable affine open subset if necessary, we may assume by the previous paragraph that any irreducible holonomic $\H_{\bc}(D(f) \cap V',\Pa)$-module admits an irreducible holonomic extension to $V'$. As in Proposition~\ref{prop:etalemelysiso}, we consider the \'etale morphism $\varphi \colon W \times_{\Pa} V' \to V$, where $V \subset X$ is the image of $\varphi$. This fits into a commutative diagram
\[
\begin{tikzcd}
W \times_{\Pa} (V' \cap D(f)) \ar[r,"\psi"] \ar[d,hook] & D(f) \cap V \ar[d,hook] \\
W \times_{\Pa} V' \ar[r,"{\varphi}"] & V, 
\end{tikzcd}
\]
where $\psi$ is the restriction of $\varphi$ to $W \times_{\Pa} (V' \cap D(f))$. As noted in Remark~\ref{rem:Lunaopen}, we may assume that $V$ is affine and $\varphi$ finite, faithfully flat. Let $M_0 = M |_{D(f) \cap V}$. By Corollary~\ref{cor:holEtLoc}, $\psi^0 (M_0)$ is a holonomic module. Moreover, since $\psi$ is finite and $M$ irreducible, $\psi^0 (M_0)$ has finite length. By Proposition~\ref{prop:etalepushpullsummand}, the module $M_0$ is a summand of $\psi_0 ( \psi^0 (M_0))$. Since $\psi_0$ is exact and $\psi^0 (M_0)$ has finite length, there must therefore exist an irreducible subfactor $M_1$ of $\psi^0 (M_0)$ with $\psi_0 \, M_1 \twoheadrightarrow M_0$. Using the Morita equivalence of \eqref{eq:matrixequivZ} and the fact that irreducible holonomic $\H_{\bc}(D(f) \cap V',\Pa)$-modules admit irreducible holonomic extensions by the previous paragraph, we deduce that $M_1$ admits an irreducible holonomic extension $N_1$ to $W \times_{\Pa} V'$. Let $N_2 = \varphi_0 \, N_1$. By Lemma~\ref{lem:etalepushforward}, $N_2$ is holonomic. Moreover, since $\varphi$ is finite, $N_2$ has finite length. By construction, there is a surjection from $N_2[f^{-1}]$ to $M_0$. Since $M_0$ is irreducible and $N_2$ has finite length, there must be an irreducible subfactor $N$ of $N_2$ such that $N[f^{-1}]$ surjects onto $M_0$. Since $N$ is irreducible and $N[f^{-1}] \neq 0$, the composition $N \to N[f^{-1}] \to M_0$ is non-zero. Thus, it is an injection and $N[f^{-1}] = M_0$. 
\end{proof}

\begin{remark}
	Just as in the case of $\dd$-modules, if an irreducible holonomic extension $N$ exists then it is unique up to isomorphism. 
\end{remark}

\subsection{Length}

In this section we use Corollary~\ref{cor:QHRfinitelength} to show that if Condition~\ref{cond:good2} holds then all holonomic modules have finite length. 

\begin{prop}\label{prop:finLength}
Assume that Condition~\ref{cond:good2} holds and $\omega \in \mathbb{H}^2 ( X,\Omega_X^{1,2})^W$ is Zariski locally trivial. Then holonomic $\sH_{\omega,\bc}(X,W)$-modules have finite length, as objects in the category of all quasi-coherent $\sH_{\omega,\bc}(X,W)$-modules. 
\end{prop}

\begin{proof} 
	The statement is local on $X/W$, so it suffices to prove it for a suitable $W$-stable affine neighbourhood of any point $x\in X$. Thus, we assume that $\omega = 0$. Further, if the stabilizer of $x$ is $\Pa\subset W$ then let $U_0$ be an affine neighbourhood of $x$ such that for any $u\in U_0$ we have $W_u \subset \Pa$. As in Remark~\ref{rem:Lunaopen}, we may assume that there is a finite \'etale $W$-stable map $V' = W \times_{W'} U_0 \rightarrow U$. We are now in the situation of Proposition~\ref{prop:etalemelysiso} and $\H_{\bc}(V',W) =\mc{O}(U_0)^{\Pa}\otimes_{\mc{O}(U_0)^W}\H_{\bc}(U_0,W)$ is Morita equivalent to $\H_{\bc}(U_0,\Pa)$. Therefore, we may replace $X$ with $U_0$ and $W$ with $\Pa$. Thus it suffices to assume that the stabilizer of $x$ equals $W$. By Proposition \ref{prop:goodembedding}, there is some open neighbourhood $V$ of $x$ in $X$ and a $\bc$-melys closed embedding $V \hookrightarrow \h$. Thus, Kashiwara's Theorem \ref{thm:Kash} applies, and the category of $\H_{\bc}(X,W)$-modules gets identified with the category of $\H_{\bc}(\h,W)$-modules supported on $X$, in a way preserving holonomicity (Corollary~\ref{cor:holoclosedpushforward}). Moreover, since Condition~\ref{cond:good2} is assumed to hold for $(X,\bc)$, Condition~\ref{cond:good1} holds for $(\h,W,\bc)$. The result follows from Proposition~\ref{prop:holonomiclengthlinear}. 
\end{proof}

\begin{remark}\label{rem:holoHsmodfinlength}
	If $\kk$ is a field of characteristic zero that is not algebraically closed and $X$ is a variety over $\kk$, then we say that a $\sH_{\omega,\bc}(X,W)$-module $\ms{M}$ is holonomic if its extension of scalars to an algebraic closure $\overline{\kk}$ is holonomic over $\sH_{\omega,\bc}(X_{\overline{\kk}},W)$. Then Proposition~\ref{prop:finLength} implies that $\ms{M}$ has finite length. 
\end{remark}

\begin{remark}
	Ivan Losev has informed us that he has a proof of Proposition~\ref{prop:finLength} that works for all parameters. 
\end{remark}

In the case of rational Cherednik algebras $\H_{\bc}(\h,W)$, the following result was shown by Ginzburg \cite{Primitive}. 

\begin{cor}\label{cor:Hhasfinitelength}
	If Condition~\ref{cond:good2} holds and $\omega \in \mathbb{H}^2 ( X,\Omega_X^{1,2})^W$ is Zariski locally trivial then the sheaf $\sH_{\omega,\bc}(X,W)$ is finite length as a bimodule over itself.
\end{cor}

\begin{proof}
	Again, this is a local statement, therefore we assume that $X$ is affine. By \cite[Lemma 4.1]{ThompsonHolI}, there is a parameter $(\omega^{\mr{op}},\bc^{\mr{op}})$ such that $\H_{\omega,\bc}(X,W)^{\mr{op}} \cong \H_{\omega^{\mr{op}},\bc^{\mr{op}}}(X,W)$. The algebra $\H_{\omega,\bc}(X,W)$ is a holonomic module over 
    \[
    \H_{\omega,\bc}(X,W) \otimes \H_{\omega,\bc}(X,W)^{\mr{op}} \cong \H_{(\omega,\omega^{\mr{op}}),(\bc,\bc^{\mr{op}})}(X \times X, W \times W)
    \]
    because the order filtration on $\H_{\omega,\bc}(X,W)$ is a good filtration when the algebra is considered a bimodule, and the diagonal copy of $T^* X$ in $T^*X \times T^* X \cong T^* (X \times X)$ is isotropic (up to sign).  One can check that Condition~\ref{cond:good2} holds for $\bc$ if and only if it holds for $\bc^{\mr{op}}$. Therefore, the corollary follows from Proposition~\ref{prop:finLength}. 
\end{proof}

\textbf{Throughout the remainder of the article, we assume that:} 
\begin{quote}
    $\omega \in \mathbb{H}^2 ( X,\Omega_X^{1,2})^W$ is Zariski locally trivial and the conclusion of Theorem~\ref{thm:holonomicextension} and Proposition~\ref{prop:finLength} hold for all spaces $X,Y,\ds$, groups $W$, and functions $\bc$ satisfying the conventions of Section~\ref{sec:Dunkloperator}. 
\end{quote}

\section{$b$-functions}\label{sec:bfunctionRCA} 

In this section, we introduce the notion of $b$-functions for sheaves of Cherednik algebras, generalizing the theory for $\dd$-modules. This will allow us to prove the main result, that push-forward along an open embedding preserves holonomicity. Since this is a local statement, we assume that $X$ is affine. Let $f \in \mc{O}(X)^W$ be a non-zero non-unit and set $U = D(f)$, an affine open subset of $X$.

\subsection{Extension of scalars}

Throughout this section, $s$ will denote a formal variable. Let $\H := \H_{\omega,\bc}(X,W)$. We consider $\H$-modules defined over the field $\K := \kk(s)$. The algebraic closure of $\K$ is denoted $\overline{\K}$. If $A$ is a $\kk[s]$-algebra and $M$ a $\kk[s]$-module, then we write $M_A := M \o_{\kk[s]} A$. Similarly, if $A$ is a $\K$-algebra and $M$ a $\K$-vector space then $M_A := M \o_{\K} A$. If $\lambda \in \kk$ and $M$ a $\kk[s]$-module then $M_{\lambda} := M \o_{\kk[s]} \kk$, where $s$ acts on $\kk$ by evaluation at $\lambda$. 

For brevity, write $B = \mc{O}(T^* X) \rtimes W$. 

\begin{lem}\label{lem:cyclicassgr}
	If $I$ is a left ideal of $\H[s]$ and $\mc{Q}$ the quotient filtration on $\H[s]/I$, then the $B[s]$-module $\gr_{\mc{Q}}(\H[s]/I)$ is cyclic with generator the image of $1$. 
\end{lem}

\begin{proof}
	By definition of quotient filtration, the map $\H[s] \twoheadrightarrow \H[s] /I$ is strictly filtered. This means that the associated graded morphism $B[s] = \gr_{\mc{F}} \H[s] \to \gr_{\mc{Q}} (\H[s]/ I)$ is also surjective. 
\end{proof}

\begin{lem}\label{lem:specializefiltrationlambda}
Let $M$ be a finitely generated $\H[s]$-module and $\mc{G}_{\idot} M$ a good filtration on $M$. 
\begin{enumerate}
\item[(i)] For $\lambda$ (Zariski) generic in $\kk$, the module $M$ has no $(s-\lambda)$-torsion and $(\mc{G}_{\idot} M)_{\lambda}$ is a good filtration on $M_{\lambda}$.
\item[(ii)] $(\mc{G}_{\idot} M)_{\overline{\mathbb{K}}}$ is a good filtration on $M_{\overline{\mathbb{K}}}$. 
\end{enumerate} 
\end{lem}

\begin{proof}
	Note that each graded piece of $K := \gr_{\mc{G}} M$ is a finite $\mc{O}(X)[s]$-module and $K$ is finitely generated (by definition) over $B[s]$. Let $T$ be the $\kk[s]$-torsion submodule of $K$. We may pick a finite $\kk[s]$-submodule $N \subset T$ such that $N$ generates $T$ over $B[s]$. Then $N = \bigoplus_{i = 1}^k N_{\mu_i}$ for some $\mu_i \in \kk$, where $(s - \mu_i)$ acts nilpotently on $N_{\mu_i}$. Since $s$ is central in $B[s]$, we have $T = \bigoplus_{i = 1}^k B \cdot N_{\mu_i}$ with $(s - \mu_i)$ acting nilpotently on $B \cdot N_{\mu_i}$. It follows that if $g = \prod_{i = 1}^k (s - \mu_i)$ and $A = \kk[s][g^{-1}]$ then $T_A = 0$ and $K_A$ is free over $A$. Induction on $j$ with respect to the short exact sequences $0 \to \mc{G}_{j-1} M \to \mc{G}_j M \to K_j \to 0$ implies that $(\mc{G}_j M)_A$ is free over $A$. This proves the first part of (i). 

For any $\lambda \notin \{ \mu_1, \ds, \mu_k \}$ it follows that $(\gr_{\mc{G}} M)_{\lambda} = \gr_{\mc{G}} M_{\lambda}$. In other words, for all $i$ we can identify $(\mc{G}_i M)_{\lambda}$ with the image of $\mc{G}_i M$ in $M_{\lambda}$. Since $\gr_{\mc{G}} M$ is finitely generated, there exists a surjection $B[s]^{\oplus r} \twoheadrightarrow \gr_{\mc{G}} M$. Tensoring gives
	$$
	B^{\oplus r} \twoheadrightarrow (\gr_{\mc{G}} M)_{\lambda} = \gr_{\mc{G}} M_{\lambda},
	$$
	which implies that the filtration $\mc{G}_{\idot} M_{\lambda}$ is good. 
	
	The analogous claims for $M_{\overline{\mathbb{K}}}$ in (ii) are straightforward. 
\end{proof}

\begin{prop}\label{prop:holonomicleftspecialize}
	Let $I$ be a left ideal of $\H[s]$ such that $\H_{\K} / I_{\K}$ is holonomic. Then, for all but finitely many $\lambda \in \kk$, $\H / I_{\lambda}$ is holonomic.
\end{prop}

\begin{proof}
	Let $\mc{L}$ be a leaf. Then $\H_{\K} / I_{\K}$ holonomic means that 
	$$
	2 \dim \left(\SS (\H_{\overline{\mathbb{K}}} / I_{\overline{\mathbb{K}}}) \cap \mc{L}_{\overline{\mathbb{K}}} \right) \le \dim \mc{L}_{\overline{\mathbb{K}}}.
	$$ 
	We give $\H[s] / I$ the quotient filtration $\mc{Q}_{\idot}$ and let $\overline{1}$ denote the image of $1 \in \H[s]$ in $\gr_{\mc{Q}}(\H[s] / I)$. The element $\overline{1}$ is a generator of $\gr_{\mc{Q}}(\H[s] / I)$ by Lemma~\ref{lem:cyclicassgr}. Let $S \subset (T^* X)/W$ denote the support of $\overline{1} \in \gr_{\mc{Q}}(\H[s] / I)_{\overline{\mathbb{K}}}$. Then, $2 \dim \left(S \cap \mc{L}_{\overline{\mathbb{K}}} \right) \le \dim \mc{L}_{\overline{\mathbb{K}}}$. Therefore, there exist $D_1, \ds, D_k \in \H_{\overline{\mathbb{K}}}$ such that:
	\begin{enumerate}
		\item[(a)] $\sigma(D_i) \in \overline{\mathbb{K}}[T^* X]^W$, 
		\item[(b)] $\sigma(D_i) \cdot \overline{1} = 0$, and
		\item[(c)]$2 \dim \left( V(\sigma(D_1), \ds, \sigma(D_k)) \cap \mc{L}_{\overline{\mathbb{K}}} \right) \le \dim \mc{L}_{\overline{\mathbb{K}}}$.
	\end{enumerate}
    The coefficients of the $D_i$ lie in some finite extension $\mathbb{K} \subset \mathbb{L}$ of $\mathbb{K}$. By choosing a basis of ${\mathbb{L}}$ over $\K = \kk(s)$, we may assume that $D_i \in \H[s]$. Choose some $i$ and let $N = \deg D_i$. Then $\sigma(D_i) \cdot \overline{1} = 0$ implies that $(D_i + I(s)_{\overline{\mathbb{K}}}) \cap \mc{F}_{N-1} \H[s] \neq \emptyset$. Up to a unit in $\kk(s)$, this implies that $(D_i + I) \cap \mc{F}_{N-1} \H[s] \neq \emptyset$. Hence, $\sigma(D_i) \cdot \overline{1} = 0$ in $\gr_{\mc{Q}}(\H[s] / I)$. If $f_i = \sigma(D_i)$ then $f_i \in \mc{O}(T^* X)^W[s] = Z(B[s])$ with $f_i \cdot \overline{1} = 0$. Since $B[s] \cdot \overline{1} = \gr_{\mc{Q}}(\H[s] / I)$ by Lemma~\ref{lem:cyclicassgr}, we deduce that the $f_i$ annihilate $\gr_{\mc{Q}}(\H[s] / I)$. 
	
	Finally, consider the projection map $q \colon V(f_1, \ds, f_k) \cap (\mc{L} \times \mathbb{A}^1) \to \mathbb{A}^1$. Property (c) says that the generic fibre of $q$ has dimension at most $\frac{1}{2} \dim \mc{L}$. Since $\mathbb{A}^1$ is an irreducible variety, Grothendieck's generic flatness result \cite[{Lemma 051R}]{stacks} says that there exists an open, non-empty subset $U \subset \mathbb{A}^1$ such that $q^{-1}(U) \to U$ is flat. In particular, $\dim q^{-1}(\lambda) \le \frac{1}{2} \dim \mc{L}$ for all $\lambda \in U$. By right-exactness of tensor products, $(\H[s] / I)_{\lambda} = \H / I_{\lambda}$. Shrinking $U$ if necessary, Lemma~\ref{lem:specializefiltrationlambda}(i) says that $(\gr_{\mc{Q}} \H[s] / I)_{\lambda} = \gr_{\mc{Q}} \H/I_{\lambda}$ for all $\lambda \in U$. We deduce that the $f_i(\lambda)$ annihilate $\gr_{\mc{Q}} \H/I_{\lambda}$ and hence $\SS(\H/I_{\lambda}) \subset q^{-1}(\lambda)$. 	
	
	Since there are only finitely many symplectic leaves in $(T^* X)/W$, the claim follows. 
\end{proof}

\subsection{Twisting by an invariant function}

Let $l \in \kk$. Conjugation by the formal expression $f^{s+l}$ defines an automorphism $\varphi_{l}$ of $\H_{\omega,\bc}(U,W)[s]$ that acts trivially on $\mc{O}(U)[s]$ and $W$ and acts on Dunkl operators by 
	$$
	\varphi_l(D_v) = D_v + (s + l) \, i_{\sigma(v)}(d \log f).
	$$
For any $\H_{\omega,\bc}(U,W)[s]$-module $M$ we write suggestively $M f^{s+l}$ for the twist of $M$ by $\varphi_l$. Thus, 
$$
D \cdot (m f^{s+l}) := (\varphi_l(D) m) f^{s+l}, 
$$
for $D \in \H_{\omega,\bc}(U,W)[s]$ and $m \in M$. In particular, when $\omega = 0$, the free rank one $\mc{O}_{U,\K}$-module $\mc{O}_{U,\K} f^{s+l}$ is a $\H_{\omega,\bc}(U,W)_{\K}$-module.

We will mainly be interested in the case $l = 0$, but require $l \neq 0$ at one key step. The following is immediate. 

\begin{lem}\label{lem:fltof}
	Let $l \in \Z$ and let $M$ be a $\H_{\omega,\bc}(U,W)[s]$-module. Then $m \o f^{s+l} \mapsto f^l(m \o f^s)$ is an isomorphism $M f^{s+l} \iso M f^s$ of $\H_{\omega,\bc}(U,W)[s]$-modules. 
\end{lem} 

\begin{remark}
\begin{enumerate}
	\item Thinking of $f$ as a $W$-invariant function $U \to \kk^{\times}$, Example~\ref{ex:invariantfunctionmelys2} shows that $M \o_{\kk[t^{\pm1}]} \kk[t,t^{-1}][s] t^s$ is a $\H_{\omega,\bc}(U,W)[s]$-module when $M$ is a $\H_{\omega,\bc}(U,W)$-module. This construction agrees with $Mf^s$, as defined above.  
	\item More generally, if we are given a collection of invariant functions $f_1, \ds, f_k$ and set $U = (f_1 \cdots f_k \neq 0)$ then we can twist any $\H_{\omega,\bc}(U,W)[s_1, \ds, s_k]$-module by $f_1^{s_1} \cdots f_k^{s_k}$. 
\end{enumerate}	
\end{remark}

\begin{lem}\label{lem:holoMsf}
Let $M$ be a holonomic $\H_{\omega,\bc}(U,W)$-module. Then $M_{\K} f^s$ is a holonomic module over $\H_{\omega,\bc}(U,W)_{\K}$. 
\end{lem}

\begin{proof}
	This follows from the fact that the automorphism $\varphi$ of $\H_{\omega,\bc}(U,W)_{\K}$ preserves the order filtration. Therefore, if $\{ M_i \}$ is a good filtration of $M$ then the $M_{i,\K} f^s$ form a good filtration of $M_{\K} f^s$. This implies that 
	\begin{equation}\label{eq:ssLwidetilde}
	\SS \left(M_{\overline{\mathbb{K}}} f^s\right) = \SS(M)_{\overline{\mathbb{K}}},
	\end{equation}
	which implies the statement of the lemma. 
\end{proof}

\begin{prop}\label{prop:holosextf}
	Let $N$ be an irreducible holonomic $\H_{\omega,\bc}(U,W)_{\K}$-module. Then, for each $x \in X$ there exists an affine open neighbourhood $V$ of $x$ such that $N |_{U \cap V}$ admits an irreducible holonomic extension to $V$. That is, there exists an irreducible holonomic $\H_{\omega,\bc}(V,W)_{\K}$-module $L$ such that $N |_{U \cap V} \cong L[f^{-1}]$. 
\end{prop}

\begin{proof}
	Let $M = N_{\overline{\mathbb{K}}}$. Then, by definition (see Remark~\ref{rem:holoHsmodfinlength}), $M$ is a holonomic $\H_{\omega,\bc}(U,W)_{\overline{\mathbb{K}}}$-module. Therefore, Proposition~\ref{prop:finLength} says that $M$ has finite length. Let $M_0 \subset M$ be an irreducible submodule. Replacing $X$ by a suitable $V$ and $U$ by $U \cap V$, Theorem~\ref{thm:holonomicextension} says there is an irreducible holonomic extension $M_1$ of $M_0 |_{U \cap V}$ to $V$. Replace $N$ by $N |_{U \cap V}$ and $M_0 |_{U \cap V}$ by $M_0$. Since $M_1$ is finitely presented and the extension $\mathbb{K}\subset \overline{\mathbb{K}}$ is algebraic, there exists some finite extension $\mathbb{K}\subset \mathbb{L}$ and $\H_{\omega,\bc}(V,W)_{\mathbb{L}}$-submodule $M_2$ of $M_1$ such that $M_1 = (M_2)_{\overline{\mathbb{K}}}$. We note that $M_2$ is an irreducible $\H_{\omega,\bc}(V,W)_{\mathbb{L}}$-module, for if it were not then picking a proper submodule $S$ of $M_2$ would give a proper $\H_{\omega,\bc}(V,W)_{\overline{\mathbb{K}}}$-submodule $S_{\overline{\mathbb{K}}}$ of $M_1$ because the extension $\mathbb{L}\subset \overline{\mathbb{K}}$ is faithfully flat. 
    
    We claim that $M_2$ has finite length over $\H_{\omega,\bc}(V,W)_{\K}$. Indeed, $M_2 |_{\H_{\omega,\bc}(V,W)_{\K}}$ is finitely generated since $[\mathbb{L}:\K]$ is finite and $M_2$ is finitely generated over $\H_{\omega,\bc}(V,W)_{\mathbb{L}}$. So if $M_2$ is not finite length then it has an infinite descending chain of submodules. Choosing $T \subset M_2 |_{\H_{\omega,\bc}(V,W)_{\K}}$ such that $M_2 / T$ has length greater than $[\mathbb{L}:\K]$ would contradict the fact that $T_{\mathbb{L}}$ must surject onto the irreducible module $M_2$. 
    
    Therefore, we can choose an irreducible $\H_{\omega,\bc}(V,W)_{\K}$-submodule $L$ of $M_2$. Then 
    \[
    L[f^{-1}] \hookrightarrow M_1[f^{-1}] |_{\H_{\omega,\bc}(U \cap V,W)_{\K}} = M|_{\H_{\omega,\bc}(U \cap V,W)_{\K}}.
    \]
    We note that there exists an (infinite) index set $I$ such that
    \[
    M |_{\H_{\omega,\bc}(U \cap V,W)_{\K}} \cong N^{\oplus I}. 
    \]
    Therefore, there exists a projection $M |_{\H_{\omega,\bc}(U \cap V,W)_{\K}} \to N$ such that the composition $L[f^{-1}] \hookrightarrow M \to N$ is non-zero. Since $L$ is irreducible, $L[f^{-1}]$ is an irreducible $\H_{\omega,\bc}(U \cap V,W)_{\K}$-module and hence the morphism $L[f^{-1}] \to N$ is injective. On the other hand, $N$ is irreducible so it is surjective. It follows that $L$ is holonomic, and thus the required extension.  
\end{proof}

Let $j : U \hookrightarrow X$ be the open embedding.

\begin{prop}\label{prop:Msholnomic}
Let $M$ be a holonomic $\H_{\omega,\bc}(U,W)$-module. 
\begin{enumerate}
    \item[(i)] The $\H_{\omega,\bc}(X,W)_{\K}$-module $j_0 (M_{\K} f^s)$ is holonomic.
    \item[(ii)] For any $m \in M$ there exists $D \in \H_{\omega,\bc}(X,W)[s]$ and $0 \neq b(s) \in \kk[s]$ such that 
$$
D(f m \otimes f^{s}) = b(s) m \otimes f^s.
$$
\end{enumerate}
\end{prop}

\begin{proof}
Both claims are local on $X$. This is clear for (i). For (ii), assume that there exist $g_1, \ds, g_k \in \mc{O}(X)$, with $X = D(g_1) \cup \cdots \cup D(g_k)$, and $D_i \in \H_{\omega,\bc}(D(g_i),W)[s], 0 \neq b_i(s) \in \kk[s]$, such that $D_i(f m \otimes f^{s}) = b_i(s) m \otimes f^s$. Replacing $g_i$ by some power, we may assume that $g_i D_i \in \H_{\omega,\bc}(X,W)[s]$. There exist $h_1, \ds, h_k$ such that $1 = h_1 g_1 + \cdots + h_k g_k$. Pick $b(s) \neq 0$ to be the least common multiple of the $b_i(s)$, so that $b(s) = c_i(s) b_i(s)$ for all $i$. Then, 
\begin{align*}
    b(s) m \otimes f^s & = \sum_{i = 1}^k h_i g_i b(s) m \otimes f^s \\
    & = \sum_{i = 1}^k h_i c_i(s) g_i D_i(f m \otimes f^s),
\end{align*}
and hence $D(f m \otimes f^{s}) = b(s) m \otimes f^s$ for $D = \sum_{i = 1}^k h_i c_i(s) g_i D_i$. 

Thus, we shrink $X$ so that all relevant irreducible holonomic $\H_{\omega,\bc}(U,W)_{\K}$-modules admit irreducible extensions to $X$ by Proposition~\ref{prop:holosextf}. Lemma~\ref{lem:holoMsf} says that $M_{\K} f^s$ is holonomic. Thus, it admits a finite composition series
$$
\{ 0 \} = M_0 \subset M_1 \subset \cdots \subset M_{\ell} = M_{\K} f^s
$$
by Remark~\ref{rem:holoHsmodfinlength}. 

Our proof of the proposition will be by induction on the length ${\ell}$ of $M_{\K} f^s$. Since $M_1$ is irreducible, Proposition~\ref{prop:holosextf} says that we can choose an irreducible holonomic extension $L_1 \subset j_0 M_1$ of $M_1$. Let $m \otimes f^s \in j_0 M_1$. Then $(j_0 M_1) / L_1$ is supported on $X \setminus U$. Therefore there exists $k \gg 0$ such that $f^k m \otimes f^s \in L_1$. Since $L_1$ is irreducible, 
$$
\H_{\omega,\bc}(X,W)_{\K} \cdot (f^k m \otimes f^s) =  \H_{\omega,\bc}(X,W)_{\K} \cdot (f^{k+1} m \otimes f^s) = \cdots .
$$
In particular, $f^k m \otimes f^s \in \H_{\omega,\bc}(X,W)_{\K} \cdot (f^{k+1} m \otimes f^s)$. Hence $D_0(s) \cdot (f^{k+1} m \otimes f^s) = f^k m \otimes f^s$ for some $D_0(s)$. This implies that $D_0(s+k) \cdot (f m \otimes f^s) = m \otimes f^s$. Clearing denominators, we get the required identity. Next we claim that $j_0 M_1 = L_1$. In particular $j_0 M_1$ is holonomic. Assume otherwise, and choose $m \otimes f^s \in j_0 M_1 \setminus L_1$. Then $f^k m \otimes f^s \in L_1$. We know that there exists $D_1 \in \H_{\omega,\bc}(X,W)_{\K}$ such that $D_1(f^k m \otimes f^s) = f^{k-1} m \otimes f^s \in L_1$ and $D_2$ such that $D_2(f^{k-1} m \otimes f^s) = f^{k-2} m \otimes f^s$ etc. Eventually we see that $m \otimes f^s \in L_1$, contradicting our initial assumptions. 

Now we assume that $j_0 M_{i-1}$ is holonomic and the identity holds for every $m \otimes f^s \in j_0 M_{i-1}$. We have a short exact sequence 
$$
0 \rightarrow j_0 M_{i-1} \rightarrow j_0 M_i \rightarrow j_0(M_i / M_{i-1}) \rightarrow 0.
$$
Choose a holonomic extension $L_i' \subset j_0(M_i / M_{i-1})$ of the irreducible module $M_i / M_{i-1}$ and let $L_i$ be its preimage in $j_0 M_i$. Then $L_i / j_0 M_{i-1} \cong L_i'$ implies that $L_i$ is holonomic. Moreover, we have a commutative diagram
$$
\begin{tikzcd}
0 \ar[r] & M_{i-1} \ar[d,equal] \ar[r] & L_i |_U \ar[r] \ar[d] & M_{i} / M_{i-1} \ar[r] \ar[d,equal] & 0 \\
0 \ar[r] & M_{i-1} \ar[r] & M_i \ar[r] & M_{i} / M_{i-1} \ar[r] & 0, 
\end{tikzcd}
$$
where the rows are exact, which implies that $L_i |_U = M_i$. Now take $m \otimes f^s \in j_0 M_i$. Again, there exists $k \gg 0$ such that $f^k m \otimes f^s \in L_i$. Since $L_i$ has finite length, the sequence 
$$
\H_{\omega,\bc}(X,W)_{\K} \cdot (f^k m \otimes f^s) \supset  \H_{\omega,\bc}(X,W)_{\K} \cdot (f^{k+1} m \otimes f^s) \supset \cdots 
$$
must eventually stabilize. In other words, 
$$
\H_{\omega,\bc}(X,W)_{\K} \cdot (f^{k+l} m \otimes f^s) =  \H_{\omega,\bc}(X,W)_{\K} \cdot (f^{k + l +1} m \otimes f^s) \quad \textrm{for some $l \gg 0$.}
$$
In particular, there exists $D_0 \in \H_{\omega,\bc}(X,W)_{\K}$ such that $D_0(f^{k + l +1} m \otimes f^s) = f^{k + l} m \otimes f^s$. Clearing denominators, we may write this as 
\[
D_1(f^{k + l +1} m \otimes f^s) = b_1(s) f^{k + l} m \otimes f^s
\]
for some $D_1 \in \H_{\omega,\bc}(X,W)[s]$ and $b_1 \in \kk[s]$ non-zero. By Lemma~\ref{lem:fltof}, this is equivalent to 
\begin{equation}\label{eq:Dbkleq}
	D_1(f m \o f^{s + k + l}) = b_1(s) m \otimes f^{s+k+l}.
\end{equation} 
Writing $D_1 = D_1(s)$ to show that $D_1$ is a polynomial in $s$ with coefficients in $\H_{\omega,\bc}(X,W)$, we make the formal substitution $t = s + k + l$ in \eqref{eq:Dbkleq} to get:   
$$
D_1(t - k - l)(f m \otimes f^t) = b_1(t - k - l) (m \otimes f^t);
$$
equivalently, $D(f m \otimes f^s) = b(s) (m \otimes f^s)$ for $b(s) = b_1(s-k-l) \neq 0$ and $D = D_1(s - k - l)$.

Next we show that $j_0 M_i$ is holonomic. In fact, the same argument shows that $j_0 M_i = L_i$. Notice also that we have proved that $\ell(j_0 M_i) = \ell(M_i) (= i)$. This completes the proof of the proposition. 
\end{proof}

\begin{cor}\label{cor:bfunction}
Let $N$ be a holonomic $\H_{\omega,\bc}(U,W)$-module and $u \in N$. Then there exist $D \in \H_{\omega,\bc}(X,W)[s]$ and $0 \neq b(s) \in \kk[s]$ such that 
$$ 
D_k(f^{k+1} u) = b(k) f^k u, \quad \forall \ k \in \Z, 
$$
where $D_k$ is the specialization of $D$ at $s = k$. 
\end{cor}

As noted in the introduction, if we take $\omega = 0$ and consider the $\H_{\bc}$-module $\mc{O}(X)$, then for any $W$-invariant function $f$ there exists $D \in \H_{\bc}(X,W)[s]$ and $0 \neq b(s) \in \kk[s]$ such that 
\begin{equation}\label{eq:bfunch}
	D(f^{s+1}) = b(s) f^s. 
\end{equation} 
Just as for $\dd$-modules, there is a unique monic polynomial $b(s) \in \kk[s]$ of minimal degree such that \eqref{eq:bfunch} holds. We call this polynomial $b$ the \textit{$b$-function} of $f$.  

Assume now that $(\h,W)$ is a complex reflection group. The \textit{discriminant} of $W$ is the homogeneous element $\delta \in \kk[\h]^W$ whose non-vanishing locus $(\delta \neq 0)$ equals the locus in $\h$ where $W$ acts freely. 

\begin{question}
	For each $\bc \in \mc{S}(\h)^W$, what are the roots of the $b$-function associated to $\delta$? 
\end{question}

For $W$ crystallographic, the roots of the $b$-function of $\delta$, with respect to the algebra $\dd(\h/W)$, were computed by Opdam \cite{OpdamApplications}. In equation \eqref{eq:bfunch}, one may assume without loss of generality that $D \in e \H_{\bc}(\h,W) e[s]$. The (faithful) action of $e \H_{\bc}(\h,W) e$ on $\kk[\h]^W$ factors through $\dd(\h/W)$. Thus $e \H_{\bc}(\h,W) e[s] \subset \dd(\h/W)[s]$, which implies that the $b$-function of $\delta$ with respect to $\dd(\h/W)$ divides the $b$-function with respect to $\H_{\bc}(\h,W)$.

\begin{example}
	Let $X = \mathbb{A}^1$ with linear action of $\Z / {\ell} \Z = \langle g \rangle$. We describe the $b$-function of $f = x^{\ell}$. If $e_0, \ds, e_{\ell-1} \in \kk[\Z/ \ell \Z]$ are the primitive orthogonal idempotents then there exist $\kappa_0, \ds, \kappa_{\ell-1} \in \kk$ depending linearly on $\bc(g), \bc(g^2), \ds, \bc(g^{\ell-1})$ such that the main defining relation of $\H_{\bc}$ can be written as
	\[
	[y,x] = 1 + \ell \sum_{i = 0}^{\ell-1} (\kappa_{i+1} - \kappa_i) e_i,  
	\]
	where $\kappa_{\ell} := \kappa_0$. Setting $\kappa_0 = 0$, the natural action of $\kk[x] \rtimes \Z/ \ell \Z$ on $\kk[x,x^{-1}]$ extends to an action of $\H_{\bc}$, with $y \cdot x^i = (i + \ell \kappa_{i} ) x^{i-1}$, and hence  
	$$
	y^{\ell} x^{\ell(s+1)} = \prod_{i = 0}^{\ell-1}(\ell(s+1) - i + \ell \kappa_{-i}) x^{\ell s}. 
	$$
	Applying the substitution $j = \ell-i$, the $b$-function for $x^{\ell}$ is $\prod_{j = 1}^{\ell}\left(s+ \frac{j}{\ell} + \kappa_{j}\right)$, where the subscripts $\kappa_j$ are taken modulo $\ell$.
\end{example} 

The $b$-function was introduced so that we can prove the following, which is one of the main results of the article. 

\begin{thm}\label{thm:preshol}
Let $N$ be a holonomic $\H_{\omega,\bc}(U,W)$-module. Then $j_0 N$ is holonomic. 
\end{thm}

\begin{proof}
An easy inductive argument using the fact that $j_0$ is exact means that we can assume $N$ is irreducible. First we show that $j_0 N$ is coherent. Choose $0 \neq u \in N$. We claim that $j_0 N = \H_{\omega,\bc}(X,W) \cdot \{ f^{-k} u \  | \ k \in \mathbb{N} \}$. Take $u' \in j_0 N$. Since $N$ is irreducible, there exists $D' \in \H_{\omega,\bc}(U,W)$ such that $D' u = u'$. Now $D' = D f^{-k}$ for some $D \in \H_{\omega,\bc}(X,W)$ and $k \gg 0$. Hence $u' \in \H_{\omega,\bc}(X,W) \cdot \{ f^{-k} u \  | \ k \in \N \}$. Thus, it suffices to show that there is some $k \gg 0$ such that $f^{-l} u \in \H_{\omega,\bc}(X,W) \cdot f^{-k} u$ for all $l > k$. Let $b(s) \in \kk[s]$ be the polynomial (depending on $u$) satisfying the statement of Corollary~\ref{cor:bfunction}. Then Corollary \ref{cor:bfunction} shows that $f^{-l} u \in \H_{\omega,\bc}(X,W) \cdot f^{-k} u$ provided none of the elements in $-k - \N$ are roots of the polynomial $b$.

It remains to show that $j_0 N$ is holonomic. Notice that we have shown that it is generated by the single element $f^{-k} u$. Let $I$ denote the annihilator in $\H_{\omega,\bc}(X,W)[s]$ of $f^{-k} u \otimes f^s$, so that $\H_{\omega,\bc}(X,W)[s] \cdot (f^{-k} u \otimes f^s) \cong \H_{\omega,\bc}(X,W)[s] / I$. By Proposition \ref{prop:Msholnomic}(i), $(\H_{\omega,\bc}(X,W)[s] / I)_{\K}$ is a holonomic module. We deduce from Proposition~\ref{prop:holonomicleftspecialize} that there is some $l \gg 0$ such that $\H_{\omega,\bc}(X,W) / I_{-l}$ is holonomic. Lemma~\ref{lem:specializefiltrationlambda}(i) says that $l$ can be chosen so that $\H_{\omega,\bc}(X,W)[s] / I$ has no $(s + l)$-torsion. This implies that the annihilator of $f^{-k-l}u$ in $\H$ is $I_{-l}$. Thus,  
$$
\H_{\omega,\bc}(X,W) / I_{-l} \cong \H_{\omega,\bc}(X,W) \cdot (f^{-k-l}u) = j_0 N, 
$$
as required. 
\end{proof}

\subsection{Consequences of Theorem~\ref{thm:preshol}} As in the case of $\dd$-modules, Theorem~\ref{thm:preshol} easily generalizes to the case where $X$ is not necessarily affine and $U$ not the complement of a divisor. Let $U$ be an arbitrary $W$-stable open subset of $X$ and $j : U \hookrightarrow X$ the corresponding embedding.

\begin{thm}\label{thm:jpushforwardholo}
	 If $\ms{N}$ is a holonomic $\sH_{\omega,\bc}(U,W)$-module then the cohomology of the derived pushforward $j_+ \ms{N}$ is holonomic.   
\end{thm}

\begin{proof}
	The statement is local on $X$, therefore we may assume that $X$ is affine. Since $U$ is the preimage of its image in $X/W$, we may fix a finite covering $D(f_1), \ds, D(f_k)$ of $U$ by principal open sets, with each $f_i$ invariant. Then the cohomology of $j_+ \ms{N}$ is the cohomology of the \v Cech complex
	\begin{equation}\label{eq:Cech}
	    0 \rightarrow \bigoplus_{i = 1}^k (j_i)_0 (\ms{N} |_{D(f_i)}) \rightarrow \cdots
	\end{equation}
	It follows from Theorem \ref{thm:preshol} that the latter has holonomic cohomology. 
\end{proof}

By the usual spectral sequence argument, \Cref{thm:jpushforwardholo} implies that the derived functor $j_+ \colon D^b_{\mr{qc}}(\sH_{\omega,\bc}(U,W)) \to D^b_{\mr{qc}}(\sH_{\omega,\bc}(X,W))$ restricts to a functor $D^b_{\mr{Hol}}(\sH_{\omega,\bc}(U,W)) \to D^b_{\mr{Hol}}(\sH_{\omega,\bc}(X,W))$. Therefore, we may define $j_! \colon D^b_{\mr{Hol}}(\sH_{\omega,\bc}(U,W)) \to D^b_{\mr{Hol}}(\sH_{\omega,\bc}(X,W))$ by 
\[
j_! = \D_X \circ j_+ \circ \D_U. 
\]
It is left adjoint to $j^!$. We could similarly define $j^+ = \D_U \circ j^! \circ \D_X$, but $j^+ \cong j^!$. Finally, this allows us to define the minimal extension $j_{!+}$ as the image of the adjunction $j_! \to j_+$. 

\begin{cor}\label{cor:simpleholonomicsocle}
	If $\ms{N}$ is an irreducible (non-zero) holonomic $\sH_{\omega,\bc}(U,W)$-module then there is a unique (up to isomorphism) irreducible holonomic extension of $\ms{N}$ to $X$; it is the socle of the module $j_0 \ms{N}$. 
\end{cor}

\begin{proof}
	First note that $\ms{N} \neq 0$ implies that $j_0 \ms{N} \neq 0$ because $(j_0 \ms{N}) |_U = \ms{N}$. Since $j_0 \ms{N}$ is holonomic by \Cref{thm:jpushforwardholo}, it has finite length and hence its socle is non-zero. Let $\ms{N}'$ be an irreducible submodule of $j_0 \ms{N}$. 
    
    The module $j_0 \ms{N}$ has no non-zero submodules supported on $X \setminus U$. This can be checked locally on $X$, therefore assume $X$ affine. Then $j_0 \ms{N}$ is a submodule of the first non-zero term in the complex \eqref{eq:Cech}. Each summand $(j_i)_0 (\ms{N} |_{D(f_i)})$ is $f_i$-torsion-free and hence has no sections supported on $X \setminus D(f_i) \supset X \setminus U$. It follows that $j_0 \ms{N}$ has no submodules supported on $X \setminus U$. 
    
    This implies that $\ms{N}' |_U \neq 0$ and hence $\ms{N}'$ is an irreducible holonomic extension of $\ms{N}$. If $\ms{N}''$ is any other irreducible holonomic extension then the adjunction $\ms{N}'' \to j_0 j^0 \ms{N}'' = j_0 \ms{N}$ shows that $\ms{N}''$ is a summand of the socle of $j_0 \ms{N}$. But localizing the exact sequence $0 \to \Soc (j_0 \ms{N}) \to j_0 \ms{N}$ to $U$ shows that 
\[
1 = \ell(\ms{N}) = \ell(j^0 j_0 \ms{N}) \ge \ell (j^0 (\Soc j_0 \ms{N})) = \ell(\Soc (j_0 \ms{N})).
\]
Hence $\Soc j_0 \ms{N}$ is irreducible. 
\end{proof}

If $\Hol{\sH_{\omega,\bc}(X,W)}_{X \setminus U}$ denotes the full subcategory of $\Hol{\sH_{\omega,\bc}(X,W)}$ consisting of all modules supported on $X \setminus U$, then \Cref{thm:jpushforwardholo} implies that $\Hol{\sH_{\omega,\bc}(X,W)}_{X \setminus U}$ is a localising subcategory with quotient equivalent to $\Hol{\sH_{\omega,\bc}(U,W)}$. Moreover, Corollary~\ref{cor:simpleholonomicsocle} implies that the set $\Irr \Hol{\sH_{\omega,\bc}(X,W)}$ of isomorphism classes of irreducible holonomic modules is the disjoint union of $\Irr \Hol{\sH_{\omega,\bc}(U,W)}$ and $\Irr \Hol{\sH_{\omega,\bc}(X,W)}_{X \setminus U}$.  	
	
\begin{cor}\label{cor:localcohomology}
	Under the assumptions of Lemma~\ref{lem:localcohomology}, if $\ms{M}$ is a holonomic $\sH_{\omega,\bc}(X,W)$-module then $R^k \Gamma_Y(\ms{M})$ is holonomic for all $k$. 
\end{cor}

\begin{proof}
	This follows directly from \Cref{thm:jpushforwardholo} and Lemma~\ref{lem:localcohomology}. 
\end{proof}

\begin{cor}\label{cor:pullbackclosedholo}
 Let $i \colon Y \hookrightarrow X$ be a $W$-invariant smooth closed subvariety. If $\ms{M}$ is a holonomic $\sH_{\omega,\bc}(X,W)$-module then $H^k (i^{!} \ms{M})$ is holonomic for all $k$.
\end{cor}

\begin{proof}
	This follows from Corollary~\ref{cor:holoclosedpushforward}, Lemma~\ref{lem:localcohomology} and Corollary~\ref{cor:localcohomology}. 
\end{proof}

\section{Preservation of holonomicity}

In this section we show that both $\varphi^!$ and $\varphi_+$ preserve complexes with holonomic cohomology, for arbitrary $\bc$-melys morphisms. This allows us to define adjoint functors $\varphi_!$ and $\varphi^+$. 

\subsection{Linear representations}

In this section we consider melys morphisms between linear representations of $W$. The goal is to prove the following. 

\begin{thm}\label{thm:linearpreservholonomic}
	Let $\h,\mathfrak{k}$ be linear representations of $W$ and $\varphi \colon \mathfrak{k} \to \h$ a $\bc$-melys morphism. Then $\varphi^!$ and $\varphi_+$ send complexes with holonomic cohomology to complexes with holonomic cohomology.  
\end{thm}

Before we prove Theorem~\ref{thm:linearpreservholonomic}, we consider the case where $\h$ is irreducible. 

\begin{lem}\label{lem:melyslinearirr}
		Let $\h,\mathfrak{k}$ be linear representations of $W$ such that $(\h,W)$ is an irreducible complex reflection group with $W = W(\bc)$. If $\varphi \colon \mathfrak{k} \to \h$ is a non-zero $\bc$-melys morphism then $\mathfrak{k} = \mathfrak{k}_W \oplus \mathfrak{k}^W$, where $(\mathfrak{k}_W,W)$ is an irreducible complex reflection group and either
		\begin{enumerate}
			\item[(i)] $\mathfrak{k}_W \cong \h$ and $\varphi$ is projection onto $\mathfrak{k}_W$; or
			\item[(ii)] $(\h,W) = (\mathbb{A}^1,\Z / \ell \Z) = (\mathfrak{k}_W,W)$ and $\varphi$ is projection followed by 
			\[
			\mathbb{A}^1 \to \mathbb{A}^1, \quad x \mapsto x^r,
			\]
			for some $r$ coprime to $\ell$.
		\end{enumerate} 
	In particular, $\varphi$ is homogeneous. 
\end{lem}

\begin{proof}
	Let $H$ be a reflecting hyperplane such that $(s,H)\in \mc{S}_{\bc}(\h)$ for some reflection $s$. That is, $\bc(s,H) \neq 0$. Let $\alpha_H \in \h^*$ be a linear functional whose kernel is $H$. By definition of melys morphism, $V(\varphi^*(\alpha_H)) \subset \mathfrak{k}^s$. If $\mathfrak{k}^s = \mathfrak{k}$ then $s$ acts trivially on $\mathfrak{k}$. This means that $\varphi^*(\alpha_H)$ is $s$-invariant. But $s (\alpha_H) = \lambda_s \alpha_H$ with $\lambda_s \neq 1$, which would force $\varphi^*(\alpha_H) = 0$. Since $(\h,W)$ is irreducible, $\h^*$ is spanned by the $\alpha_{H'}$ as we vary over the hyperplanes $H'$ in the $W$-orbit of $H$. Moreover, for each such $H'$ there exists a reflection $s'$ such that $\bc(s',H') \neq 0$ since the function $\bc$ is $W$-invariant. This implies that $\varphi^*(\h^*) = 0$. In other words, $\varphi = 0$. This contradicts our assumptions, hence $s$ acts non-trivially on $\mathfrak{k}$. 
	 
	 Since $\varphi$ is $\bc$-melys, $V(\varphi^*(\alpha_H)) \subset \mathfrak{k}^s \neq \mathfrak{k}$. If $V(\varphi^*(\alpha_H)) = \emptyset$ then $\varphi^*(\alpha_H) \in \kk[\mathfrak{k}]^{\times} = \kk^{\times}$. But this again implies that $s$ acts trivially on $\varphi^*(\alpha_H)$; a contradiction. Thus, $V(\varphi^*(\alpha_H))$ is a non-empty hypersurface contained in the subspace $\mathfrak{k}^s$. We deduce that $\mathfrak{k}^s$ is a hyperplane and $V(\varphi^*(\alpha_H)) = \mathfrak{k}^s$ since $\mathfrak{k}^s$ is irreducible. Hence (up to a scalar) $\varphi^*(\alpha_H) = \beta_s^{r}$, where $\beta_s \in \mathfrak{k}^*$ has kernel $\mathfrak{k}^s$. In particular, $s$ acts as a reflection on $\mathfrak{k}$. Since this is true for all $(s,H) \in \mc{S}_{\bc}(\h)$ and $W = \langle \mc{S}_{\bc}(\h) \rangle$, we deduce that $W$ acts on $\mathfrak{k}$ as a complex reflection group. The action of $W$ on $\mathfrak{k}$ is faithful since $\mathfrak{h}^*$ is a faithful representation and $\mathfrak{h}^* \subset \kk[\mathfrak{k}]$. This implies that $\mathfrak{k} = \mathfrak{k}_W \oplus \mathfrak{k}^W$, where $(\mathfrak{k}_W,W)$ is an irreducible complex reflection group because $(\h,W)$ is irreducible. One can check directly from Shephard-Todd's classification of irreducible complex reflection groups that reflection representations of a given $W$ always have the same dimension. In particular, $\dim \mathfrak{k}_W = \dim \h$.  
	 
	 Since $\h^*$ is irreducible and $\varphi^*(\h^*) \cap \kk[\mathfrak{k}]_r \neq 0$, we have $\varphi^*(\h^*) \subset \kk[\mathfrak{k}]_r$ and thus $\varphi$ is homogeneous. This also shows that $r$ is independent of $(s,H)$. 

	Assume first that $r = 1$. Then the map $\varphi^* |_{\h^*}$ is linear and the irreducibility of the group $W$ forces $\varphi^* |_{\h^*} \colon \h^* \iso \mathfrak{k}_W^*$. Thus, $\varphi$ is projection onto $\mathfrak{k}_W \cong \h$ as in (i). 
	
	Assume now that $r > 1$ and set $U = \varphi^*(\h^*)$. Let $\mc{A}$ denote the set of all reflection hyperplanes in $\h$. We claim that $\mc{A}$ contains exactly $d := \dim \h$ elements. Since $\h$ is irreducible, $\{ \alpha_{H} \, | \, H \in \mc{A} \}$ is a spanning set for $\h^*$. Writing $\alpha_{H} = \beta_{H}^r$, the set $\{ \beta_{H} \, | \, H \in \mc{A} \}$ is a spanning set for $\mf{k}_W^*$ since $\{ \Ker(\beta_{H}) \, | \, H \in \mc{A} \}$ is the set of reflection hyperplanes in $(\mf{k}_W,W)$. If $|\mc{A}| > d = \dim \mf{k}_W$ then there is a non-trivial relation $\beta_{H} = \sum_{i = 1}^d a_i \beta_{H_i}$, where we may assume the $\beta_{H_i}$ are a basis of $\mf{k}_W^*$. Expanding 
	\[
	\beta_{H}^r = \left( \sum_{i = 1}^d a_i \beta_{H_i} \right)^r
	\]  
	gives some vector in $U$ that is not a linear combination of the $\beta_{H_i}^r$; a contradiction since the $\beta_{H_i}^r$ are linearly independent. If $|\mc{A}| = \dim \h$ then $W$ is an abelian group. Since $\h$ is irreducible, we deduce that $W = \Z / \ell \Z$.   
\end{proof}

\begin{lem}\label{lem:lastfactormultpreserve}
	Assume $\h = \h_1 \times \mathbb{A}^1$, $W = W_1 \times (\Z / \ell \Z)$ and $\varphi \colon \h \to \h$ is the identity on $\h_1$ and $x \mapsto x^r$ on $\mathbb{A}^1$. Then  $\varphi^!$ and $\varphi_+$ send complexes with holonomic cohomology to complexes with holonomic cohomology. 
\end{lem}

\begin{proof}
	Let $U = \h_1 \times (\mathbb{A}^1 \setminus \{ 0 \})$ and identify $\h_1$ with the closed complement to $U$. Note that $\varphi^{-1}(U) = U$ and $\varphi^{-1}(\h_1) = \h_1$. Moreover, if $i \colon \h_1 \to \h$ is the closed embedding then $\varphi \circ i = i$. Indeed, on functions it is given by 
	\[
	\begin{tikzcd}
		\kk[\h_1][x] \ar[r,"{x \, \mapsto \, x^r}"] &  \kk[\h_1][x] \ar[r,"{x \, \mapsto \, 0}"] &  \kk[\h_1]. 
	\end{tikzcd}
	\]
	Clearly, $\varphi |_{U} \colon U \to U$ is \'etale. Let $\ms{M}^{\idot} \in D^b_{\mr{Hol}}(\sH)$. Then we have an exact triangle
	\[
	i_+ (i^! \varphi^! \ms{M}^{\idot}) \to  \varphi^! \ms{M}^{\idot} \to j_+ (j^! \varphi^! \ms{M}^{\idot}) \stackrel{+1}{\to}. 
	\]
	Using that $\varphi \circ i = i$ and $\varphi \circ j = j \circ \varphi$ are $\bc$-melys morphisms, this becomes 
	\[
	i_+ (i^! \ms{M}^{\idot}) \to  \varphi^! \ms{M}^{\idot} \to j_+ (\varphi^! j^! \ms{M}^{\idot}) \stackrel{+1}{\to}. 
	\]
	Since the cohomology groups of $i^! \ms{M}^{\idot}$ and $j_+ (\varphi^! j^! \ms{M}^{\idot})$ are holonomic by Corollary~\ref{cor:holoclosedpushforward}, \Cref{thm:jpushforwardholo} and Corollary~\ref{cor:pullbackclosedholo}, it follows that the cohomology of $\varphi^! \ms{M}^{\idot}$ is holonomic too.
	
	Similarly, if we apply $\varphi_+$ to the triangle $i_+ (i^! \ms{M}^{\idot}) \to \ms{M}^{\idot} \to j_+ (j^! \ms{M}^{\idot}) \stackrel{+1}{\to}$ we get 
	\[
	\varphi_+ (i_+ (i^! \ms{M}^{\idot})) \to  \varphi_+ \ms{M}^{\idot} \to \varphi_+(j_+ (j^! \ms{M}^{\idot})) \stackrel{+1}{\to}
	\]
	which is isomorphic to 
	\[
	i_+ (i^! \ms{M}^{\idot}) \to  \varphi_+ \ms{M}^{\idot} \to j_+ (\varphi_+ (j^! \ms{M}^{\idot})) \stackrel{+1}{\to}. 
	\]
	It follows that the cohomology of $\varphi_+ \ms{M}^{\idot}$ is holonomic too.
\end{proof}	

Let $p \colon \mathbb{A}^n_{\kk} \times \h \to \h$ be the projection map with $W$ acting linearly on $\h$ and trivially on $\mathbb{A}^n_{\kk}$ (so that $p$ is $\bc$-melys for any $\bc$). 

\begin{prop}\label{prop:projectpreserveholo}
	The derived functors $p_+,p^!$ preserve objects with holonomic cohomology. 
\end{prop}

\begin{proof}
	Factoring $p$ as a composition $\mathbb{A}^n_{\kk} \times \h \to  \mathbb{A}^{n-1}_{\kk} \times \h \to \cdots \to \mathbb{A}^1_{\kk} \times \h \to \h$ of projections, where each $\mathbb{A}^i_{\kk} \times \h$ is thought of as a linear representation of $W$, we may assume that $n = 1$. We have in place the analogues for Cherednik algebras of all the results needed in the standard proof for $\dd$-modules. Namely, it is shown in \cite[Lemma~4.4]{ThompsonHolI} that the Fourier transform on $\H_{\bc}(\kk \times \h,W)$ preserves the category of holonomic modules. If $M$ is a holonomic $\H_{\bc}(\kk \times \h,W)$-module and $i \colon \{ 0 \} \times \h \hookrightarrow \kk \times \h$ is the closed embedding then the cohomology of $i^{!} M$ is holonomic by Corollary~\ref{cor:pullbackclosedholo}. Then the claim for $p_+$ follows from \cite[Proposition~3.2.6]{HTT}, noting that $W$ acts trivially on $\kk$ so $\H_{\bc}(\kk \times \h,W) = \dd(\kk) \otimes \H_{\bc}(\h,W)$. 
	
	The claim for $p^!$ is immediate; see \cite[page~86]{HTT}.
\end{proof}

\begin{proof}[Proof of Theorem~\ref{thm:linearpreservholonomic}]
	By Lemma~\ref{lem:forgetcommutepushpull}, we may assume $W = W(\bc)$, in which case $W$ is a complex reflection group generated by the reflections in $\mc{S}_{\bc}(\h)$. Decompose $\h = \h_1 \oplus \cdots \oplus \h_k$, where $W = W_1 \times \cdots \times W_{k-1}$ with $(\h_i,W_i)$ an irreducible complex reflection group and $\h_k = \h^W$. Let $p_i \colon \h \to \h_i$ be the projection and note that $\varphi_i := p_i \circ \varphi$ is $(\bc |_{W_i})$-melys. Then Lemma~\ref{lem:melyslinearirr} says that one of the following four possibilities is realized: 
	\begin{enumerate}
		\item[(I)]  $\varphi_i = 0$, 
		\item[(II)] $\varphi_i$ is the projection onto $\h_i$,
		\item[(III)] $W_i \cong \Z / \ell \Z$ and $\varphi_i$ is the projection followed by $x \mapsto x^{r_i}$; or 
		\item[(IV)] $i = k$ and $\varphi_k \colon \mf{k} \to \h^W$. 
	\end{enumerate} 
	Since $\varphi_i$ is melys and $W_i$ is generated by reflections $s$ with $\bc(s) \neq 0$, $\varphi_i = 0$ implies that $W_i$ acts trivially on $\mf{k}$. Therefore, if we decompose $\mf{k} = \mf{k}_W \oplus \mf{k}^W$ and $\h = \h_W\oplus \h^W$ then $\mf{k}_W \hookrightarrow \h_W$ is $\bc$-melys. For each factor $\h_i$ of type (III), we write $\varphi(r_i) \colon \h \to \h$ for the melys morphism which is $x \mapsto x^{r_i}$ on $\h_i$ and the identity on every other factor. The composition of all such morphisms is $\varphi(\boldsymbol{r}) = \varphi(r_d) \circ \cdots \circ \varphi(r_1)$. We may then factorize $\varphi$ as
	\[
	\begin{tikzcd}
	 \h_W \times \mf{k}^W \times \h^W \cong \h \times \mf{k}^W \ar[r,"\varphi(\boldsymbol{r})"] & \h \times \mf{k}^W \ar[r,"p"] & \h \\
	\mf{k} \times \h^W = \mf{k}_W \times \mf{k}^W \times \h^W \ar[u,hook] & &  \\
	\mf{k} \ar[u,hook,"\mr{Id} \times \varphi_k"] \ar[uurr,"\varphi"'] & &  
	\end{tikzcd}
	\]
	Thus, the above commutative diagram shows that $\varphi$ can be factorized as a closed embedding followed by the finite morphism $\varphi(\boldsymbol{r})$ followed by the projection, all of which are melys. This gives diagram~\eqref{eq:factor1} of the introduction. 
	
Proposition~\ref{prop:projectpreserveholo} says that $p_+$ and $p^!$ preserve holonomic modules. Lemma~\ref{lem:lastfactormultpreserve} says that $\varphi(\boldsymbol{r})_+$ and $\varphi(\boldsymbol{r})^!$ preserve holonomicity. Finally, the functors $i_+$ and $i^!$ associated to a closed embedding preserve holonomicity by Corollary~\ref{cor:holoclosedpushforward} and Corollary~\ref{cor:pullbackclosedholo} respectively.     
\end{proof}

\subsection{The general case}

We now explain how the general case reduces to the linear case. The hard work is in proving a relative version (Proposition~\ref{prop:melysliftstolinear}) of Proposition~\ref{prop:goodembedding}.

\begin{lem}\label{lem:passtangentconemelys}
	Let $(\h,W)$ be an irreducible complex reflection group and $\bc \in \mc{S}(\h)^W$ such that $W = W(\bc)$. If $\varphi \colon Y \to \h$ is a $\bc$-melys morphism with $y \in Y^W$ mapping to $0 \in \h$ then $\varphi$ induces a $\bc$-melys morphism $\gr \, \varphi \colon T_y Y \to \h$. 
\end{lem}

\begin{proof}
	We assume that $Y$ is a good neighbourhood of $y$, as in Lemma~\ref{lem:goodneighbourhood}, with $I(Z) = (f_Z)$ for each reflection hypersurface $Z$ and every such hypersurface passes through $y$. Let $\mf{m}$ be the maximal ideal in $\mc{O}(Y)$ defining $y$. Since $Y$ is smooth, we identify $\kk[T_y Y]$ with the ring of functions $\bigoplus_{i \ge 0} \mf{m}^i / \mf{m}^{i-1}$ on the tangent cone. Then there exists $r \ge 1$ such that $\varphi^*(\h^*) \subset \mf{m}^r \setminus \mf{m}^{r+1}$. The composition $\h^* \to \mf{m}^r \to \mf{m}^{r} / \mf{m}^{r+1}$ induces a $W$-equivariant morphism $\gr \, \varphi \colon T_y Y \to \h$. We claim that this morphism is $\bc$-melys. 
    
    Let $(s,H) \in \mc{S}_{\bc}(\h)$. Note that since $Y^s$ is a disjoint union of smooth varieties and every reflecting hypersurface is assumed to pass through $y$, either $Y^s = Y$ or $Y^s = Z$ is a reflecting hypersurface. Either way, there exists a unit $u \in \mc{O}(Y)$ and $r_H \ge 0$ such that $\varphi^*(\alpha_H) = u f_Z^{r_H}$. Since $\alpha_H \in \h^*$ and $f_Z \in \mf{m}$, we have $r_H \ge r$. In fact, since $\varphi^*(\h^*)$ is irreducible, it is spanned by the polynomials $w( \varphi^*(\alpha_H))$ for $w \in W$, thus we must have $r_H = r$, otherwise $\varphi^*(\h^*) \subset \mf{m}^{r+1}$. We deduce that $(\gr \, \varphi)^*(\alpha_H) = u(y) \overline{f}_Z^r$, where $\overline{f}_Z$ is the image of $f_Z$ in $\mf{m} / \mf{m}^2$. Since $(T_y Y)^s = T_y Z$ is the zero-set of $\overline{f}_Z$, the claim follows. 
\end{proof}

\begin{prop}\label{prop:melysliftstolinear}
Let $(\h,W)$ be a complex reflection group and $\bc \in \mc{S}(\h)^W$ be such that $W = W(\bc)$. Let $\varphi \colon Y \to \h$ be a $\bc$-melys morphism with $y \in Y^W$ mapping to $0 \in \h$. 

There exists a finite \'etale $\varphi^* \bc$-melys morphism $\phi \colon Y' \to Y$, a strongly $(\varphi \circ \phi)^* \bc$-melys closed embedding $(Y',0) \hookrightarrow (\mf{k},0)$ and a $\bc$-melys morphism $\Phi \colon \mf{k} \to \h$ such that the diagram  
\begin{equation}\label{eq:reducelinearcommdiagram}
	\begin{tikzcd}
		Y  \ar[dr,"\varphi"'] & Y' \arrow[r,hook] \ar[l,"\phi"'] & \mf{k} \ar[dl,"\Phi"] \\ 
		& \h & 
	\end{tikzcd} 
\end{equation}
is commutative with $\phi(0) =y$ and $\Phi(0) = 0$. 
\end{prop}

\begin{proof}
	We may assume that $Y$ is a good affine neighbourhood of $y$ as in Lemma~\ref{lem:goodneighbourhood}, replacing $Y$ with such a neighbourhood if needed. In particular, every $Z$ is principal (defined by $f_Z$ say) for $(w,Z) \in \mc{S}(Y)$. Let $\mf{m} \lhd \mc{O}(Y)$ be the maximal ideal defining $y$. Note that if $(w,Z) \in \mc{S}(Y)$ and $f_Z \in \mf{m} \setminus \mf{m}^2$ such that $w(f_Z) = \lambda_w f_Z$ for some $1 \neq \lambda_w \in \kk$, then replacing $Y$ by some smaller good neighbourhood of $y$ if necessary, the function $f_Z$ generates $I(Z)$. 
	
	 Factor $\h = \h_1 \times \cdots \times \h_k \times \h^W$ with $W = W_1 \times \cdots \times W_k$ so that $(\h_i,W_i)$ is irreducible.
For each $(s,H) \in \mc{S}_{\bc}(\h)$, we fix $\alpha_H \in \h^*$ with kernel $H$. We say that $H$ belongs to $\h_i$ if $s \in W_i$; equivalently $\alpha_{H} \in \h_i^*$. If $\varphi^*(\alpha_H) = 0$ for some $(s,H) \in \mc{S}_{\bc}(\h)$ then the melys condition implies that $Y^s = Y$. Moreover, we have $\varphi^*(w(\alpha_H)) = 0$ for all $w \in W$ implying that there is some irreducible factor $W_i$ of $W$ that acts trivially on $Y$. In this case, we factor the morphism $\varphi$ as 
\[
\begin{tikzcd}
Y \ar[r] \ar[rr,"\varphi",bend left=30] & \h^{W_i} \ar[r,hook] & \h,
\end{tikzcd}
\]
so that $\varphi^*(\alpha_H) \neq 0$ for all $(s,H) \in \mc{S}_{\bc}(\h)$. The function $\varphi^*(\alpha_H)$ is not a unit either because $\varphi^*(\alpha_H)(y) = 0$.   

We claim that, after passing to an \'etale extension, there exists for each $1 \le i \le k$ a $W$-submodule $\mf{k}_i^* \subset \mf{m}$ (with $W$ acting through $W_i$) such that for each $(s,H) \in \mc{S}_{\bc}(\h_i)$ there is an eigenvector $\beta_{H} \in \mf{k}^*_i$ for $s$ with $\varphi^*(\alpha_H) = \beta_H^{r_H}$. If $p_i \colon \h \to \h_i$ is the projection onto $\h_i$ then replacing $\varphi$ by $p_i \circ \varphi$, the proof of Lemma~\ref{lem:passtangentconemelys} shows that $r_i := r_H = r_{H'}$ for all $H,H'$ belonging to the factor $\h_i$. Moreover, combining Lemma~\ref{lem:passtangentconemelys} with Lemma~\ref{lem:melyslinearirr}, we see that $r_i > 1$ implies that $\dim \h_i = 1$. 

If $r_i = 1$, we simply take $\mf{k}_i^* = \varphi^*(\h_i^*)$. 

If $r_i > 1$ then we choose $f_Z \in \mf{m} \setminus \mf{m}^2$ a (non-trivial) eigenvector for $s$ and invariant for every other factor $W_j$. Then $\varphi^*(\alpha_H) = u_H f_Z^{r_i}$ for some unit $u_H$. Since both $\varphi^*(\alpha_H)$ and $f_Z^{r_i}$ are eigenvectors for $s$ with eigenvalues in $\kk$, $s(u_H) = q u_H$ for some $q \in \kk^{\times}$. But $s(y) = y$ and $u_H(y) \neq 0$ forces $q = 1$. This implies that $u_H \in \mc{O}(Y)^W$. Let $B = \mc{O}_{Y/W,y}^h \o_{\mc{O}_{Y/W,y}} \mc{O}_{Y,y}$, where $\mc{O}_{Y/W,y}^h$ is the henselization of $\mc{O}_{Y/W,y}$. Within $B$ we may take $v_H$ to be a $W$-invariant $r_i$-th root of $u_H$. Let $A$ be the subalgebra of $B$ generated by $\mc{O}(Y)$ and the $v_H$. We take $Y' = \Spec A$ so that $\phi \colon Y' \to Y$ is \'etale and equivariant. The algebra $A$ is a finite $\mc{O}(Y)$-module, meaning that $\phi$ is finite and hence surjective. Since $\phi$ is \'etale, it is automatically strongly $\varphi^*\bc$-melys. Let $\mf{n}$ be a maximal ideal of $A$ lying over $\mf{m}$. Then $(\varphi \circ \phi)^*(\alpha_H) = (v_H f_Z)^{r_i}$ and so we take $\mf{k}_i^* = \kk \{ v_H f_Z \} \subset \mf{n}$.      

Finally, we repeat the proof of Proposition~\ref{prop:goodembedding}. Let $\mf{k}_k^* \subset \mf{n}$ be some $W$-invariant lift of $(\mf{n}/\mf{n}^2)^W$. The inclusion $\mf{k}^*_1 \times \cdots \times \mf{k}_{k}^* \to \mc{O}(Y')$ defines an \'etale morphism $Y' \to \mf{k}'$. Just as in the proof of Proposition~\ref{prop:goodembedding}, there is some $Y' \to \mathbb{A}^n$ (with $W$ acting trivially on $\mathbb{A}^n$) such that the resulting morphism $Y' \to \mf{k}' \times \mathbb{A}^n =: \mf{k}$ is a strongly $(\varphi \circ \phi)^* \bc$-melys closed embedding. This has been constructed so that there is a canonical map $\mf{k}' \to \h$. One can check that composing the projection along $\mathbb{A}^n$ with this map $\mf{k}' \to \h$ is the desired melys morphism $\Phi$.  
\end{proof}  

Finally, we come to the two main results of this section. 

\begin{thm}\label{thm:preservholonomicback}
	Let $\varphi \colon Y \to X$ be a $\bc$-melys morphism. Then the functor $\varphi^!$ sends complexes with holonomic cohomology to complexes with holonomic cohomology.  
\end{thm}

\begin{proof}
	Let $\ms{M} \in D^b_{\mr{Hol}}(\sH_{\omega,\bc}(X,W))$. It suffices to check locally on $Y$ that the cohomology of $\varphi^! \ms{M}$ is holonomic. Choose $y \in Y$ and set $x = \varphi(y)$. 	Let $\Pa = W_x$ and write $U_0$ for the set of points $x' \in X$ with $W_{x'} \subset \Pa$. Let $V_0 = \varphi^{-1}(U_0)$ and write $\widetilde{\varphi} = \varphi|_{V_0}$. Note that $W_{y'} \subset \Pa$ for all $y' \in V_0$ and both $U_0,V_0$ are $\Pa$-stable; that is, $V_0$ and $U_0$ satisfy \eqref{eq:Eetale1} and \eqref{eq:Eetale2}. We can form the commutative diagram  
	\begin{equation}\label{eq:U0V0opens}
	\begin{tikzcd}
	W \times_{\Pa} V_0 \ar[r,"\mr{Id} \times \widetilde{\varphi}"] \ar[d,"g"'] & W \times_{\Pa} U_0 \ar[d,"f"] \\
	Y \ar[r,"\varphi"] & X. 
	\end{tikzcd}
	\end{equation}
	Note that $f,g$ are \'etale and strongly melys (for $\bc$ and $\varphi^* \bc$ respectively). By Proposition~\ref{prop:etalemelysiso}, there are isomorphisms
	\begin{align}
	\sH_{(g \circ \varphi)^* \omega,(g \circ \varphi)^* \bc}(W \times_{\Pa} V_0,W) & = Z(W,\Pa,\sH_{\widetilde{\varphi}^* \omega,\widetilde{\varphi}^* \bc}(V_0,\Pa)), \label{eq:ZHsheaf1} \\ 
	\sH_{f^* \omega,f^* \bc}(W \times_{\Pa} U_0,W) &= Z(W,\Pa,\sH_{\omega,\bc}(U_0,\Pa)), \label{eq:ZHsheaf2}
	\end{align}
	allowing us to make $\ms{M} |_{U_0}$ into a complex of $\sH_{\omega,\bc}(U_0,\Pa)$-modules such that 
	\[
	f^! \ms{M} \cong \mr{Fun}_{\Pa}(W,\ms{M} |_{U_0}). 
	\]
	By Corollary~\ref{cor:holEtLoc}, it suffices to show that $g^!(\varphi^! \ms{M})$ has holonomic cohomology. By the commutativity of \eqref{eq:U0V0opens}, 
	\[
	\mr{Fun}_{\Pa}(W,(\varphi^! \ms{M}) |_{V_0}) \cong g^!(\varphi^! \ms{M}) \cong (\mr{Id} \times \widetilde{\varphi})^! f^!(\ms{M}) \cong \mr{Fun}_{\Pa}(W,\widetilde{\varphi}^!(\ms{M} |_{U_0})).
	\]
	Thus, it suffices to prove that $\widetilde{\varphi}^!$ restricts to a functor
    \[
    \widetilde{\varphi}^! \colon D^b_{\mr{Hol}}(\sH_{\omega,\bc}(U_0,W')) \to D^b_{\mr{Hol}}(\sH_{\widetilde{\varphi}^* \omega,\widetilde{\varphi}^* \bc}(V_0,W')).
    \]
    Thus, we may assume $y$ is fixed by $W$ and $Y$ is irreducible. This implies that $x =\varphi(y)$ is also fixed and we may assume $X$ is irreducible.
	
	Shrinking $X$ and $Y$, we may assume, by Proposition~\ref{prop:goodembedding}, that there is a strongly $\bc$-melys closed embedding $i \colon X \to \h$ mapping $x$ to $0$ and that $\omega= 0$. By Theorem~\ref{thm:Kash}(ii), we may replace $\ms{M}$ by $i_+\ms{M}$ and $\varphi$ by $i \circ \varphi$ so that we are reduced to showing that if $\ms{M}$ is an object of $D^b_{\mr{Hol}}(\sH_{\bc}(\h,W))$ and $\varphi \colon Y \to \h$ then $\varphi^! \ms{M}$ has holonomic cohomology at $y$. Moreover, by Lemma~\ref{lem:forgetcommutepushpull}, we may assume $W = W(\bc)$. Therefore, we can form the commutative diagram \eqref{eq:reducelinearcommdiagram} of Proposition~\ref{prop:melysliftstolinear}. By Corollary~\ref{cor:holEtLoc}, $\varphi^!\ms{M}$ has holonomic cohomology if and only if $\phi^!(\varphi^!\ms{M})$ does so. If we write $g \colon Y' \hookrightarrow \mf{k}$ for the strongly melys closed embedding then $\phi^!(\varphi^!\ms{M}) \cong g^!(\Phi^!\ms{M})$, which has holonomic cohomology by Theorem~\ref{thm:linearpreservholonomic} and Corollary~\ref{cor:pullbackclosedholo}.   
\end{proof}

\begin{thm}\label{thm:preservholonomicforward}
	Let $\varphi \colon Y \to X$ be a $\bc$-melys morphism. Then the functor $\varphi_+$ sends complexes with holonomic cohomology to complexes with holonomic cohomology.  
\end{thm}

Even reducing the proof of Theorem~\ref{thm:preservholonomicforward} to the linear case is rather involved. We leave the proof to Section~\ref{sec:thmvarphidirect}. 

Now that we know that the functors $\varphi^!,\varphi_+$ and $\D_X$ preserve holonomicity, we can define 
\[
\varphi^+ := \D_Y \circ \varphi^! \circ \D_X \colon D^b_{\mr{Hol}}(\sH_{\omega,\bc}(X,W)) \to D^b_{\mr{Hol}}(\sH_{\varphi^* \omega,\varphi^* \bc}(Y,W))
\]
and 
\[
\varphi_! := \D_X \circ \varphi^+ \circ \D_Y \colon D^b_{\mr{Hol}}(\sH_{\varphi^* \omega,\varphi^* \bc}(Y,W)) \to D^b_{\mr{Hol}}(\sH_{\omega,\bc}(X,W)).
\] 
As for $\dd$-modules, they satisfy the usual adjunctions. One should use Remark~\ref{rem:melysop} when checking the behaviour of the parameters $(\omega,\bc)$ under these functors. 

\begin{remark}
	Since the embedding of the diagonal $X \hookrightarrow X \times X$ is not $\bc$-melys for $\bc \neq 0$, the category $D^b_{\mr{Hol}}(\sH_{\omega,\bc}(X,W))$ does not have any natural internal tensor product. 
\end{remark}

\subsection{Proof of Theorem~\ref{thm:preservholonomicforward}}\label{sec:thmvarphidirect}

For each $x \in X$, we define the rank of $W$ at $x$ to be 
\[
\rk_x W := \dim T_x X - \dim (T_x X)^{W_x}.
\] 
Let $\ms{M} \in D^b_{\mr{Hol}}(\sH_{\varphi^* \omega,\varphi^* \bc}(Y,W))$. We say that $\varphi_+ \ms{M}$ has holonomic cohomology at $x$ if there is some ($W$-stable) neighbourhood $U$ of $x$ such that $(\varphi_+ \ms{M}) |_U$ has holonomic cohomology. We prove by induction on $\rk_x W$ that $\varphi_+ \ms{M}$ has holonomic cohomology at $x$. \\	

\vspace{2mm}

\noindent \textit{Step 1:} $\rk_x W = 0$. Let $X^{\circ}$ be the complement to the union of all reflection hypersurfaces in $\mc{S}_{\bc}(X)$. Then $\varphi$ being $\bc$-melys implies that $V_0 = \varphi^{-1}(X^{\circ})$ is contained in the complement to the union of all reflection hypersurfaces in $\mc{S}_{\varphi^*\bc}(Y)$. In particular, we have $\sH_{\varphi^* \omega,\varphi^* \bc}(Y,W) |_{V_0} = \dd_{\varphi^* \omega}(V_0) \rtimes W$. 
	
	If $\rk_x W = 0$ then $x$ cannot lie on any reflection hypersurface, meaning that $x \in X^{\circ}$. Thus, $(\varphi_+ \ms{M}) |_{X^{\circ}} = \varphi_+(\ms{M} |_{V_0})$ has holonomic cohomology at $x$ because the direct image functor preserves holonomic cohomology for (twisted) $W$-equivariant $\dd$-modules. \\	
	
	\vspace{2mm}
	
	\noindent \textit{Step 2:} $\rk_x W > 0$ and $x \in X^W$. Replacing $X$ by a suitable neighbourhood of $x$, we may assume by Proposition~\ref{prop:goodembedding} that there is a strongly melys closed embedding $X \hookrightarrow \h$ sending $x$ to $0$. We can also assume that $\omega =0$. Composing $\varphi$ with this embedding and replacing $\varphi$ by this composition, we may assume by Corollary~\ref{cor:holoclosedpushforward} that $X = \h$. Moreover, by Lemma~\ref{lem:forgetcommutepushpull}, we may assume that $W = W(\bc)$ acts faithfully on $\h$. 
	
	If $\varphi(Y) \cap \h^W = \emptyset$ then we factor $\varphi = j \circ \eta$, where $\eta \colon Y \to (\h \setminus \h^W)$ and $j \colon (\h \setminus \h^W) \hookrightarrow \h$. Then $\varphi_+\ms{M} = j_+(\eta_+ \ms{M})$. Since $\rk_{x'} W < \rk_0 W$ for all $x' \in \h \setminus \h^W$, we deduce by induction that $\eta_+ \ms{M}$ has holonomic cohomology, hence so too does $\varphi_+\ms{M}$ by \Cref{thm:jpushforwardholo}.
	
	Let $x_0 \in \h^W$ and $(s,H) \in \mc{S}_{\bc}(\h)$. Then $x_0 \in H$ and the melys condition implies that 
	\[
	\varphi^{-1}(x_0) \subset \varphi^{-1}(H) \subset Y^s.
	\]
	Since $W = W(\bc)$ is generated by the reflections in $\mc{S}_{\bc}(\h)$, we deduce that $\varphi^{-1}(\h^W) \subset Y^W$. For each $y \in Y^W$ mapping to some $x_0 \in \h^W$, we can find some $W$-stable affine open neighbourhood $V_y$ of $y$ in $Y$ and a commutative diagram 
	\[
	\begin{tikzcd}
		V_y  \ar[dr,"\varphi"'] & V_y' \arrow[r,hook,"i"] \ar[l,"\phi"'] & \mf{k}_y \ar[dl,"\Phi"] \\ 
		& \h & 
	\end{tikzcd} 
	\]  
	satisfying the conclusions of Proposition~\ref{prop:melysliftstolinear}. Since $\phi$ is finite \'etale, Proposition~\ref{prop:etalepushpullsummand} again implies that $\ms{M} |_{V_y}$ is a direct summand of $\phi_+ \phi^! (\ms{M} |_{V_y})$ in $D^b_{\mr{Hol}}(\sH_{\varphi^* \bc}(V_y,W))$. Therefore, $\varphi_+(\ms{M} |_{V_y})$ will have holonomic cohomology at $x_0$ if $\varphi_+(\phi_+ \phi^! (\ms{M} |_{V_y}))$ has holonomic cohomology at $x_0$. We have
	\begin{align*}
		\varphi_+(\phi_+ \phi^! (\ms{M} |_{V_y})) & \cong (\varphi \circ \phi)_+(\phi^! (\ms{M} |_{V_y})) \\
		& \cong (\Phi \circ i)_+ (\phi^! (\ms{M} |_{V_y})) \\
		& \cong \Phi_+( i_+ (\phi^! (\ms{M} |_{V_y}))), 	
	\end{align*}
	which has holonomic cohomology at $x_0$ by Theorem~\ref{thm:linearpreservholonomic}.
	
	We have shown that $Y$ can be covered by $W$-stable affine open sets $V_i$ such that either $\varphi(V_i) \cap \h^W = \emptyset$ or $V_i = V_{y_i}$ as above, for some $y_i \in Y^W$ mapping to $\h^W$. In both cases, we have shown that $\varphi_+(\ms{M} |_{V_i})$ has holonomic cohomology at any $x_0 \in \h^W$ (in particular, at $0$). It follows by a standard spectral sequence argument (see part (d) of the proof of \cite[Theorem~10.1(ii)]{BorelDmod}) that $\varphi_+ \ms{M}$ has holonomic cohomology at $0$.
	
	\vspace{2mm}

\noindent \textit{Step 3:} $\rk_x W > 0$ arbitrary. Let $\Pa = W_x$, $U_0$ and $V_0 = \varphi^{-1}(U_0)$ as in the first part of the proof of Theorem~\ref{thm:preservholonomicback}. We form again the commutative diagram \eqref{eq:U0V0opens}. Then the isomorphisms \eqref{eq:ZHsheaf1}-\eqref{eq:ZHsheaf2} allow us to make $\ms{M} |_{V_0}$ into a complex of $\sH_{\widetilde{\varphi}^* \omega,\widetilde{\varphi}^* \bc}(V_0,\Pa)$-modules such that 
	\[
	g^! \ms{M} = \mr{Fun}_{\Pa}(W,\ms{M} |_{V_0}). 
	\]
	Step 2 implies that $\widetilde{\varphi}_+ (\ms{M} |_{V_0})$, considered as a complex of $\sH_{\omega,\bc}(U_0,\Pa)$-modules, has holonomic cohomology at $x$. Then $(\mr{Id} \times \widetilde{\varphi})_+ (g^! \ms{M})$ has holonomic cohomology at $(1,x)$. Hence,
	\[
	f_+ (\mr{Id} \times \widetilde{\varphi})_+ (g^! \ms{M}) \cong \varphi_+( g_+ g^! \ms{M})
	\]  
	has holonomic cohomology at $x$. Proposition~\ref{prop:etalepushpullsummand} implies that $\ms{M}$ is a direct summand of $g_+ g^! \ms{M}$ in $D^b_{\mr{Hol}}(\sH_{\varphi^* \omega,\varphi^* \bc}(Y,W))$. We deduce that $\varphi_+(\ms{M})$ has holonomic cohomology at $x$.\\	
 
\vspace{2mm}

This completes the proof of Theorem~\ref{thm:preservholonomicforward}. 

\subsection{Regular holonomic modules}

Assume that $\kk = \C$. Recall from Section~\ref{sec:stabstrata} that $X$ admits a stabilizer stratification $X = \bigsqcup_{j \in J} \St_j$. Let $i_j$ denote the inclusion of $\St_j$ in $X$; this is a melys locally closed embedding. If $(w,Z)$ is a reflection hypersurface then $Z \cap \St_j \neq \emptyset$ implies that $Z \cap \St_j$ is a union of connected components of $\St_j$. Therefore,  $\sH_{i_j^*\omega,i_j^*\bc}(\St_j,W) = \dd_{i_j^* \omega}(\St_j) \rtimes W$ and the derived pullback $i_j^! \ms{M}$ is a complex of $\dd_{i_j^* \omega}(\St_j)$-modules. If $\ms{M}$ is holonomic, so too is the cohomology of $i_j^! \ms{M}$, by Corollary~\ref{cor:pullbackclosedholo}. We make the following definition. 

\begin{defn}\label{defn:regsingdefn}
A holonomic $\sH_{\omega,\bc}(X,W)$-module $\ms{M}$ has regular singularities if every cohomology sheaf $H^k(i_j^! \ms{M})$ has regular singularities. 
\end{defn}

A holonomic module $\ms{M}$ has regular singularities if and only if each of its composition factors does: apply the derived functor $i_j^!$ to the exact sequence $0 \to \ms{M}' \to \ms{M} \to \ms{M}'' \to 0$ with $\ms{M}''$ irreducible and consider the resulting long exact sequence in cohomology, noting that $H^{k}(i_j^! \ms{M}'') = 0$ for $k >0$.  

\begin{conjecture}\label{conj:preservregsing}
	The functors $\varphi^!,\varphi_!,\varphi^+, \varphi_+$ and $\D_X$ preserve regular holonomic modules. 
\end{conjecture} 

The case of pullback along locally closed embeddings is immediate. 

\begin{lem}
	If $\varphi$ is a locally closed embedding then $\varphi^!$ preserves regular holonomic modules. 
\end{lem}

\begin{proof}
Let $\varphi \colon Y \to X$ be a locally closed embedding and $Y = \bigsqcup_{k \in K} \mathscr{Y}_k$ the stabilizer stratification of $Y$. Then, for each $k$ there exists $j$ such that $\mathscr{Y}_k \subset \mathscr{X}_j$. Hence 
\[
i_k^! (\varphi^! \ms{M}) = (\varphi \circ i_k)^! \ms{M} = (\varphi |_{\mathscr{Y}_k})^! (i_j^! \ms{M}). 
\]
But $ (\varphi |_{\mathscr{Y}_k})^!$ is pullback in the category of $\dd$-modules, hence preserves regular singularities \cite[11.4(c)]{BorelDmod}. We deduce that  $\varphi^! \ms{M}$ has regular holonomic cohomology.
\end{proof}

\begin{remark}
	One can check that if $\ms{M}$ is a $\sH_{\omega,\bc}(X,W)$-module that is coherent over $\mc{O}_X$ then the $\dd_{i_j^* \omega}(\St_j)$-module structure on $i_j^0 \ms{M}$ agrees with that constructed in \cite[Theorem~1.4(2)]{Wilcox}. Thus, Definition~\ref{defn:regsingdefn} is very similar to the definition of regular singularities given in \cite[Definition~1.5]{Wilcox} in this case, except that we require that all cohomology groups of $i_j^! \ms{M}$ have regular singularities. 
\end{remark}

Let $X^{\an}$ denote the complex analytic space associated to $X$. Kashiwara has extended the Riemann-Hilbert correspondence to the twisted setting in \cite[Theorem~3.16.1]{KasAsterisque}. As in \textit{loc. cit.}, we may cover $X^{\an}$ by $W$-stable open sets $U_i$ such that $[\omega |_{U_i}] = 0$. Then $\mc{O}_{U_i}$ is a module over $(\sH_{\omega,\bc}(X,W) \o_{\mc{O}_X} \mc{O}_{X^{\an}}) |_{U_i} = \sH_{\bc}(U_i,W)$. This allows one to define a de Rham functor 
\[
\DR \colon D^b_{\mr{rh}}(\sH_{\omega,\bc}(X,W)) \to D^b_{\mr{\omega}}(X^{\an},W)
\]
from the derived category of $\sH_{\omega,\bc}(X,W)$-modules with regular holonomic cohomology to the category of $W$-equivariant, $\omega$-twisted, $\C$-sheaves, as in \cite[Section~3.16]{KasAsterisque}. We expect that $\DR(\ms{M}^{\idot})$ is actually (Zariski) constructible and that the resulting functor $D^b_{\mr{rh}}(\sH_{\omega,\bc}(X,W)) \to D^b_{\mr{\omega},\mr{cs}}(X^{\an},W)$ is an equivalence for generic $\bc$. 

\begin{problem}
	Describe intrinsically the image of the (abelian) category of regular holonomic $\sH_{\omega,\bc}(X,W)$-modules under the de Rham functor $\DR$. 
\end{problem} 

Finally, we consider the special case where $X = \h$ and $(W,\h)$ is a complex reflection group. Let $\mc{O}_{\bc}$ denote the category $\mc{O}$ for $\H_{\bc}(\h,W)$, as introduced in \cite{GGOR}, and let $\mc{H}_q(W)$ be the Hecke algebra associated to $W$, as defined in \cite[(5.2.5)]{GGOR}, with parameter $q$ as in \cite[Theorem~5.13]{GGOR}. We recall that $\mc{H}_q(W)$ is a finite-dimensional quotient of the group algebra $\C B_W$ of the fundamental group $B_W$ of $\h_{\reg}^{\an}/W$. Thus, $\Lmod{\mc{H}_q(W)}$ is a full subcategory of $\Lmod{B_W}$. The result \cite[Theorem~5.13]{GGOR} says that there exists an exact \textit{Knizhnik-Zamolodchikov functor} $\mathbf{kz} \colon \mc{O}_{\bc} \to \Lmod{\mc{H}_q(W)}$. 

\begin{prop}
Every $M \in \mc{O}_{\bc}$ is regular holonomic and the de Rham functor $\DR$ followed by restriction to $\h_{\reg}$ recovers the Knizhnik-Zamolodchikov functor $\mathbf{kz} \colon \mc{O}_{\bc} \to \Lmod{\mc{H}_q(W)}$. More precisely, if $\pi \colon \h_{\reg} \to \h_{\reg}/W$ is the quotient map and $j \colon \h_{\reg} \hookrightarrow \h$ then there exists a commutative diagram 
\[
\begin{tikzcd}
    & \mc{O}_{\bc} \ar[dl,"\DR"'] \ar[dr,"\mathbf{kz}"] & \\
D^b(\h^{\an},W) \ar[rr,"{(\pi_{\idot} \circ \, j^{-1} ( - ))^W}"'] &    &  \Lmod{B_W}
\end{tikzcd}
\]
\end{prop}

\begin{proof}
It has been shown in \cite[Proposition~1.3]{Wilcox} that $i_j^0 M$ is a holonomic $\dd(\St_j)$-module with regular singularities. In particular, this applies to the Verma modules $\Delta(\lambda)$. But $\Delta(\lambda)$ is free over $\mc{O}(\h)$ and hence $H^k(L i_j^0 \Delta(\lambda)) =0$ for $k \neq 0$. We deduce that $\Delta(\lambda)$ has regular singularities. As noted previously, a holonomic module has regular singularities if and only if each of its composition factors does. Since each irreducible in $\mc{O}_{\bc}$ is a quotient of some $\Delta(\lambda)$, it follows that every module in $\mc{O}_{\bc}$ is regular holonomic. 

The second claim about the Knizhnik-Zamolodchikov functor follows directly from the fact that
\[
\mathbf{kz}(M) = (\pi_{\idot} (\DR(M) |_{\h_{\reg}^{\an}}))^W, \quad \forall \, M \in \mc{O}_{\bc}, 
\]
where we identify the category of local systems on $\h_{\reg}^{\an}/W$ with $\Lmod{B_W}$ and $ \Lmod{\mc{H}_q(W)}$ as the full subcategory of representations that factor through the quotient $\C B_W \twoheadrightarrow \mc{H}_q(W)$. 
\end{proof}   

\section{The category of holonomic modules}\label{sec:classifcationholonomic}

Recall from Section~\ref{sec:stabstrata} that $X$ admits a finite stabilizer stratification. For each stratum $\St$, we fix once and for all a connected component $\St_0$ of $\St$ and let $\Pa$ be the pointwise stabilizer of $\St_0$ in $W$ and $N = \{ w \in W \, | \, w(\St_0) \subset \St_0 \}$. Then $\Pa$ is a parabolic subgroup of $W$, normal in $N$. Set $\uN = N/\Pa$. This group acts freely on $\St_0$. 
  
We have $\St = W(\St_0)$ and write $\partial \overline{\St} = \overline{\St} \setminus \St$. As in Definition \ref{defn:SerreYXsubcat}, we have full Serre subcategories $\Qcoh{\sH_{\omega,\bc}(X,W)}_{\overline{\St}}$, $\Coh{\sH_{\omega,\bc}(X,W)}_{\overline{\St}}$ and $\Hol{\sH_{\omega,\bc}(X,W)}_{\overline{\St}}$ of their respective parent categories. We drop $W$ from the notation when it is clear from the context. We choose $U \subset X$ a $W$-stable open set such that $\overline{\St} \cap U = \St$. Just as for quasi-coherent $\mc{O}_X$-modules \cite{GabrielCategories}, restriction to $U$ induces an equivalence
\begin{equation}\label{eq:Lmodquotequiv}
\gr_\St \Qcoh{\sH_{\omega,\bc}(X)} := \Qcoh{\sH_{\omega,\bc}(X)}_{\overline{\St}} / \Qcoh{\sH_{\omega,\bc}(X)}_{\partial \overline{\St}} \cong \Qcoh{\sH_{\omega,\bc}(U)}_{\St}.
\end{equation}
Lemma~\ref{lem:coherentextend} and \Cref{thm:jpushforwardholo} imply that we have the corresponding coherent and holonomic equivalences:  
\begin{align}
\gr_\St \Coh{\sH_{\omega,\bc}(X)} & := \Coh{\sH_{\omega,\bc}(X)}_{\overline{\St}} / \Coh{\sH_{\omega,\bc}(X)}_{\partial \overline{\St}} \cong \Coh{\sH_{\omega,\bc}(U)}_{\St}, \label{eq:Cohquotequiv} \\
\gr_\St \Hol{\sH_{\omega,\bc}(X)} & := \Hol{\sH_{\omega,\bc}(X)}_{\overline{\St}} / \Hol{\sH_{\omega,\bc}(X)}_{\partial \overline{\St}} \cong \Hol{\sH_{\omega,\bc}(U)}_{\St}. \label{eq:Holquotequiv} 
\end{align}
Let $U_0$ be a $N$-stable open subset of $U$ such that $\St_0$ is closed in $U_0$ and the stabilizer of any point in $U_0$ is contained in $\Pa$. 

\begin{thm}\label{thm:Serresubquotmelysequiv}
	There is an equivalence 
	\begin{enumerate}
		\item[(i)] $\gr_\St \Qcoh{\sH_{\omega,\bc}(X,W)} \cong \Qcoh{\sH_{\omega,\bc}(U_0,N)}_{\St_0}$
	\end{enumerate} 
restricting to equivalences
	\begin{enumerate}
		\item[(ii)] $\gr_\St \Coh{\sH_{\omega,\bc}(X,W)} \cong \Coh{\sH_{\omega,\bc}(U_0,N)}_{\St_0}$,
		\item[(iii)] $\gr_\St \Hol{\sH_{\omega,\bc}(X,W)} \cong \Hol{\sH_{\omega,\bc}(U_0,N)}_{\St_0}$.  
	\end{enumerate}
\end{thm}

\begin{proof}
	Since we already have identifications \eqref{eq:Lmodquotequiv}, \eqref{eq:Cohquotequiv} and \eqref{eq:Holquotequiv}, we may replace $X$ by $U$ and assume that $\St$ is closed in $X$. 
	
	The smooth but disconnected variety $V := W \times_N U_0$ carries an action of $W$. Let $\varphi \colon V \to X$ be the map $\varphi(w,u) = w(u)$. Shrinking $U$ if necessary, we may assume that the image of $\varphi$ equals $U$. Then we are in the situation described in Section~\ref{sec:redtrivstab}. 	For brevity, we write $\omega,\bc$ for $\varphi^*\omega,\varphi^* \bc$ below. Notice that $\sH_{\omega,\bc}(U_0,N)$ contains $\sH_{\omega,\bc}(U_0,\Pa)$ as a subsheaf.

	We claim that $\varphi^0$ induces equivalences 
	\begin{equation}\label{eq:varphistarequiMod}
		\varphi^0 \colon \Qcoh{\sH_{\omega,\bc}(X,W)}_\St \iso \Qcoh{\sH_{\omega,\bc}(V,W)}_{\varphi^{-1}(\St)}
	\end{equation}
and $\varphi^0 \colon \Coh{\sH_{\omega,\bc}(X,W)}_\St \iso \Coh{\sH_{\omega,\bc}(V,W)}_{\varphi^{-1}(\St)}$.

	Since $\varphi$ is a finite map, it suffices to check that \eqref{eq:varphistarequiMod} is an equivalence. 
		
	We have adjunctions $\mr{Id} \to \varphi_0 \varphi^0$ and $\varphi^0 \varphi_0 \to \mr{Id}$. To show that these are isomorphisms on $\Qcoh{\sH_{\omega,\bc}(X,W)}_\St$ and $\Qcoh{\sH_{\omega,\bc}(V,W)}_{\varphi^{-1}(\St)}$ respectively, it suffices to check on the level of quasi-coherent $\mc{O}_X$-modules (resp. $\mc{O}_V$-modules) supported on $\St$ (resp. on $\varphi^{-1}(\St)$). As in the proof of Proposition~\ref{prop:etalemelysiso}, we choose coset representatives $w_0, \ds, w_k$ of $N$ in $W$ and write $V$ as the disjoint union of the $U_i$. Then $\St = \bigsqcup_{i = 0}^k \St_i$, where $\St_i = w_i(\St_0)$. Let $\ms{M}$ be a $\sH_{\omega,\bc}(X,W)$-module supported on $\St$. As an $\mc{O}_X$-module, $\ms{M} = \bigoplus_{i = 0}^k \ms{M}_i$, where $\ms{M}_i$ is the subsheaf of sections supported on $\St_i$. Since $U_i \cap \St_j = \St_i$ if $i = j$ and is empty otherwise, 
	$$
	\varphi^0 \ms{M} = \bigoplus_{i = 0}^k \ms{M}_i |_{U_i},
	$$
	and hence $\ms{M} \to \varphi_0 \varphi^0 \ms{M}$ is an isomorphism because $\St_i$ is closed in $U$, which means that $\ms{M}_i \to \ms{M}_i |_{U_i}$ is an isomorphism. The same argument shows that $\varphi^0 \varphi_0 \ms{N} \cong \ms{N}$ for $\ms{N} \in \Qcoh{\sH_{\omega,\bc}(V,W)}_{\varphi^{-1}(\St)}$. 
	
	Equivalence (iii) then follows from Corollary~\ref{cor:holEtLoc}. 	
\end{proof}	

\subsection{Parallelizable subsets}

Suppose $\ms{M}$ is an irreducible holonomic $\sH_{\omega,\bc}(X,W)$-module. By Lemma \ref{lem:irredsupp}, the support of $\ms{M}$ is an irreducible subvariety of $X/W$; equivalently $\Supp_X \ms{M}$ is a $W$-irreducible subset of $X$.  Let $\St$ denote the (unique) largest stratum in the stabilizer stratification of $X$ such that $\Supp_X \ms{M} \cap \St \neq \emptyset$. In particular, $\Supp_X \ms{M} \cap \St$ is open and dense in $\Supp_X \ms{M}$, the latter equal to the closure in $X$ of its intersection with $\St$. Since we have fixed a connected component $\St_0$ of $\St$, Theorem~\ref{thm:Serresubquotmelysequiv} and Corollary~\ref{cor:simpleholonomicsocle} imply that the set of isomorphism classes of all such irreducible modules is in bijection with the set of isomorphism classes of all irreducible objects in $\Hol{\sH_{\omega,\bc}(U_0,N)}_{\St_0}$.

\begin{defn}\label{defn:admissibleset}
	A locally closed subset $Z \subset \St$ is said to be \textit{parallelizable} if:
	\begin{enumerate}
		\item[(a)] $Z$ is contained in the connected component $\St_0$ of $\St$,
		\item[(b)] $Z$ is irreducible, affine and smooth,
		\item[(c)] $\sigma(Z) = Z$ or $\sigma(Z) \cap Z = \emptyset$ for each $\sigma \in N$; and
		\item[(d)] the normal bundle $\mc{N}_{X/Z}$ to $Z$ in $X$ is $\Pa$-equivariantly trivial.  
	\end{enumerate} 
\end{defn}

The parallelizable subsets of $\St$ support the ``integrable connections'' whose minimal extensions will give the irreducible holonomic modules on $X$. Note that any dense affine open subset (satisfying (c)) of a parallelizable subset is again parallelizable. The following is immediate. 

\begin{lem}\label{lem:Z0admissibleopen}
	Let $Z_0$ be an irreducible subset of $\St$ satisfying conditions (a) and (c) of Definition~\ref{defn:admissibleset}. Then there exists a dense open subset $Z \subset Z_0$ that is parallelizable. 
\end{lem}

Recall that for a smooth locally closed subvariety $Z \subset X$, the normal bundle of $Z$ in $X$ is denoted $\mc{N}_{X/Z}$. The following result is presumably well-known, but we were unable to find a reference in the form given. 

\begin{prop}\label{prop:admissibletrivialnormalbundlecomplete}
	If $Z \subset \St$ is parallelizable then there is a $\Pa$-equivariant isomorphism of formal schemes $\widehat{\mc{N}}_{X/Z} \cong \widehat{X}_Z$.  
\end{prop}

\begin{proof}
	Replacing $X$ by some $\Pa$-stable affine open subset we may assume that $Z$ is closed in $X$. We work algebraically. Let $A = \mc{O}(X)$ and $B = \mc{O}(Z)$. Then $B = A/I$ for some prime ideal $I$. The fact that the normal bundle $\mc{N}_{X/Z}$ is $\Pa$-equivariantly trivial means that there exist $x_1, \ds, x_k$ in $I$ such that $\h = \mr{Span}_{\kk}\{ x_1, \ds, x_k \}$ is a $\Pa$-module and $B \o \h \iso I/I^2$ is an isomorphism of free $\Pa$-equivariant $B$-modules. The $x_i$ form a regular sequence in $A$. Moreover, since $B$ is a regular ring and locally finitely presented, it is formally smooth. Therefore, it follows by an inductive application of \cite[Lemma~1.2]{HartshorneAlgDeRham} that $B[\![x_1, \ds, x_k]\!] \iso \widehat{A}_I$. This isomorphism is $\Pa$-equivariant by construction.      	
\end{proof}

Before we come to the main classification result, we consider modules supported on a parallelizable subset. Let $\h$ be a $W$-module. Let $\mathcal{O}_{\bc}(\h^*)$ denote the category of $\H_\bc(\h,W)$-modules which are finitely generated over the subalgebra $\kk[\h^*]$ and on which $\h^*$ acts locally nilpotently. Up to swapping the role of $\h$ and $\h^*$, it is the category $\mc{O}$ introduced and studied in \cite{GGOR}. The irreducible objects in $\mc{O}_{\bc}(\h^*)$ are $L(\lambda)$ for $\lambda \in \Irr W$. 

\begin{lem} \label{lem:catOhol}
	$\mathcal{O}_{\bc}(\h^*) =  \Coh{\H_\bc(\h,W)}_{\{0\}} = \Hol{\H_{\bc}(\h,W)}_{\{0\}}$. 
\end{lem}

\begin{proof}
	The first equality is clear.
	
	Since $\{ 0 \} \times \h^*$ is Lagrangian in $\h \times \h^*$, Lemma~\ref{lem:isoclosed} implies that $\h^* / W$ is isotropic in the quotient space $(\h \times \h^*)/W$. A module in $\mathcal{O}_{\bc}(\h^*)$ is supported at the origin by definition, hence it is holonomic because its singular support in $(\h\times \h^*)/W$ is contained in $\h^*/W$. 
\end{proof}

\begin{lem}\label{lem:tensorprodsimpleZ}
	Every irreducible $\dd_{\omega}(Z) \otimes \H_{\bc}(\h,\Pa)$-module supported on $Z \times \{ 0 \}$ is of the form $M \boxtimes L(\lambda)$ for some $\lambda \in \Irr \Pa$. 
\end{lem}

\begin{proof}
	Let $K$ be an irreducible $\dd_{\omega}(Z) \otimes \H_{\bc}(\h,\Pa)$-module supported on $Z \times \{ 0 \}$; each section of $K$ is killed by some power of $\kk[\h]_+$. Let $I = \kk[\h^*]_+$. We claim that $K/ I K \neq 0$. 
	
	If $K$ were finitely generated over $\H_{\bc}(\h,\Pa)$ then the claim would follow immediately from the basic properties of category $\mc{O}$. Let $\mr{eu} \in \H_{\bc}(\h,\Pa)$ be the Euler operator. Its adjoint action on $\H_{\bc}(\h,\Pa)$ is semi-simple with integer eigenvalues; the eigenvalues of $\mr{eu}$ on $\kk[\h]_+$ are all negative and are all positive on $\kk[\h^*]_+$. The action of $\mr{eu}$ on $L(\lambda)$ is also semi-simple with eigenvalues belonging to the progression $\bc_{\lambda}, \bc_{\lambda} + 1, \bc_{\lambda} + 2, \ds$ for some $\bc_{\lambda} \in \kk$. Since every section of $K$ lies in some finitely generated $\H_{\bc}(\h,\Pa)$-module, it follows that the action of $\mr{eu}$ on $K$ is locally finite and the generalized eigenvalues of $\mr{eu}$ all lie in $\bigcup_{\lambda \in \Irr \Pa} \bc_{\lambda} + \Z_{\ge 0}$. Since the eigenvalues of $\mr{eu}$ on $I$ are all positive, it follows that $K / I K \neq 0$. 
	
	The (non-zero) space $K/IK$ is a $\dd_{\omega}(Z) \otimes \kk \Pa$-module. Since $\kk \Pa$ is split (semi-simple) each irreducible quotient of $K/IK$ is of the form $M \boxtimes \lambda$ for some $\lambda \in \Irr \Pa$ and irreducible $\dd_{\omega}(Z)$-module $M$. Recall from \cite[(2.3.1)]{GGOR} that $\nabla(\lambda) := \Hom_{\kk[\h^*] \rtimes \Pa}(\H_{\bc}(\h,\Pa),\lambda)_0$ is the space of morphisms in $\Hom_{\kk[\h^*] \rtimes \Pa}(\H_{\bc}(\h,\Pa),\lambda)$ that are locally nilpotent for $\kk[\h]_+$. By adjunction, the surjection $K \twoheadrightarrow M \boxtimes \lambda$ defines a non-zero element of 
	\begin{align*}
		\Hom_{\dd_{\omega}(Z) \otimes (\kk[\h^*] \rtimes \Pa)}(K,M \boxtimes \lambda) & \cong \Hom_{\dd_{\omega}(Z) \otimes \H_{\bc}(\h,\Pa)}(K,M \boxtimes \Hom_{\kk[\h^*] \rtimes \Pa}(\H_{\bc}(\h,\Pa),\lambda)) \\
		& \cong \Hom_{\dd_{\omega}(Z) \otimes \H_{\bc}(\h,\Pa)}(K,M \boxtimes \nabla(\lambda)).
	\end{align*}
	Since the socle of $\nabla(\lambda)$ is $L(\lambda)$, we get an injection $K \hookrightarrow M \boxtimes L(\lambda)$. 
	
	Finally, we must show that $M \boxtimes L(\lambda)$ is irreducible. Since neither $M$ nor $L(\lambda)$ need be finite-dimensional, this is not immediate. If $K$ is any non-zero $\dd_{\omega}(Z) \otimes \H_{\bc}(\h,\Pa)$-submodule, then repeated applications of $\kk[\h]_+$ force $K \cap (M \boxtimes \lambda) \neq 0$. As noted above, $M \boxtimes \lambda$ is an irreducible $\dd_{\omega}(Z) \otimes \Pa$-module meaning that $M \boxtimes \lambda \subset K$. It follows that $K  = M \boxtimes L(\lambda)$. 
\end{proof}

Let $\Lmod{(\dd_{\omega}(Z) \o \H_{\bc}(\h,\Pa))}_Z$ denote the category of finitely generated $\dd_{\omega}(Z) \o \H_{\bc}(\h,\Pa)$-modules supported on $Z \times \{ 0 \}$ in $Z \times \h$.

\begin{lem}\label{lem:admissblesplittingofcategories}
	Assume $X$ affine and $Z$ a parallelizable closed subset. Then there is an equivalence 
	$$
	\Phi_Z \colon \Lmod{(\dd_{\omega}(Z) \o \H_{\bc}(\h,\Pa))}_Z \iso \Coh{\sH_{\omega,\bc}(U_0,\Pa)}_{Z}, 
	$$
	restricting to an equivalence of full subcategories 
	$$
	\Hol{\dd_{\omega}(Z)} \boxtimes \mc{O}_{\bc}(\h^*) \iso \Hol{\sH_{\omega,\bc}(U_0,\Pa)}_{Z}.
	$$
\end{lem} 

\begin{proof}
	If $\H_{\omega,\bc}(\widehat{X}_Z,\Pa)$ is the Cherednik algebra on the formal scheme $\widehat{X}_Z$, then Proposition~\ref{prop:admissibletrivialnormalbundlecomplete} implies that $\H_{\omega,\bc}(\widehat{X}_Z,\Pa) \cong \mc{O}(Z)[\![\h]\!] \o_{\mc{O}(Z \times \h)} (\dd_{\omega}(Z) \otimes \H_{\bc}(\h,\Pa))$. 
	
	If $M$ is a finitely generated $\H_{\omega,\bc}(X,\Pa)$-module supported on $Z$ then the action of $\H_{\omega,\bc}(X,\Pa)$ extends by continuity to an action of $\H_{\omega,\bc}(\widehat{X}_Z,\Pa)$ giving an equivalence between finitely generated  $\H_{\omega,\bc}(X,\Pa)$-modules supported on $Z$ and finitely generated $\H_{\omega,\bc}(\widehat{X}_Z,\Pa)$-modules supported on $Z$. By the same reasoning, there is an equivalence between finitely generated $\H_{\bc}(\h,\Pa) \otimes \dd_{\omega}(Z)$-modules supported on the subvariety $Z \times \{ 0 \} \subset Z \times \h$ and finitely generated $\mc{O}(Z)[\![\h]\!] \o_{\mc{O}(Z \times \h)} (\dd_{\omega}(Z) \otimes \H_{\bc}(\h,\Pa))$-modules supported on $Z \times \{ 0 \}$. Combining these gives the equivalence $\Phi_Z$. 
	
	The equivalence $\Phi_{Z}$ restricts to an equivalence 
	\[
	\Hol{\dd_{\omega}(Z) \o \H_{\bc}(\h,\Pa)}_Z \iso \Hol{\sH_{\omega,\bc}(U_0,\Pa)}_{Z}.
	\]
	By Lemma~\ref{lem:catOhol}, $\Hol{\dd_{\omega}(Z)} \boxtimes \mc{O}_{\bc}(\h^*)$ is a full Serre subcategory of the category of holonomic $\dd_{\omega}(Z) \o \H_{\bc}(\h,\Pa)$-modules supported on $Z$. However, these are finite length abelian categories, where the irreducible objects of $\Hol{\dd_{\omega}(Z) \o \H_{\bc}(\h,\Pa)}_Z$ lie in $\Hol{\dd_{\omega}(Z)} \boxtimes \mc{O}_{\bc}(\h^*)$ by Lemma~\ref{lem:tensorprodsimpleZ}. We deduce that they are equal. 
\end{proof}

\subsection{Clifford theory}

We assume $U_0$ is affine. Let $\H(\Pa) = \H_{\omega,\bc}(U_0,\Pa)$ and $\H(N) = \H_{\omega,\bc}(U_0,N)$. Then $\H(\Pa) \subset \H(N)$ and Clifford theory, in the generality presented in \cite{DadeAnnals}, describes the irreducible $\H(N)$-modules in terms of the irreducible $\H(\Pa)$-modules and a certain twisted group algebra associated to $\uN$. First, we recall some terminology. If $L$ is an $\H(\Pa)$-module and $\sigma \in \uN$ then ${}^{\sigma} L$ denotes the twist of $L$ by some lift $\widetilde{\sigma} \in N$ of $\sigma$ (independent, up to isomorphism, of the lift). The stabilizer of $L$ is $\uN(L) = \{ \sigma \in \uN \, | \, {}^{\sigma} L \cong L \}$. We say that $L,L'$ are $\uN$-conjugate if there exists $\sigma \in \uN$ such that ${}^{\sigma} L \cong L'$. Finally, we say that an $\H(N)$-module $F$ \textit{lies over} the irreducible module $L$ if  
\begin{enumerate}
	\item[(1)] $F$ is semi-simple as an $\H(\Pa)$-module; and
	\item[(2)] $[ F |_{\H(\Pa)} : L'] \neq 0$ implies that $L'$ is $\uN$-conjugate to $L$. 
\end{enumerate}
Let $\Lmod{(\H(\uN),L)}$ denote the full subcategory of $\Lmod{\H(N)}$ consisting of all modules lying over $L$. 

\begin{thm}\label{thm:CliffordDade}
	Let $L$ be an irreducible $\H(\Pa)$-module and $\Ind^N_{\Pa} L:=  \H(N) \o_{\H(\Pa)} L$. Then,
	\begin{enumerate}
		\item[(i)] There exists a $2$-cocycle $\tau \in H^2(\uN(L),\kk^{\times})$ such that $\End_{\H(N)}(\Ind^N_{\Pa} L) \cong \kk_{\tau} \uN(L)$,
		\item[(ii)] $\Ind^N_{\Pa} L$ induces an equivalence between $\Lmod{\kk_{\tau} \uN(L)}$ and $\Lmod{(\H(\uN),L)}$; and
		\item[(iii)] Each irreducible $\H(N)$-module lies over exactly one $\uN$-conjugacy class of irreducible $\H(\Pa)$-modules. 
	\end{enumerate}
\end{thm}

\begin{proof}
	The claims all follow from \cite{DadeAnnals}, as we explain. First, notice that the algebra $\H(N) = \bigoplus_{i = 1}^k \H(\Pa) n_i$ is a crossed product algebra $\H(\Pa) \ast \uN$. As noted in \cite{DadeMathZ}, this implies that $\H(N)$ is strongly $\uN$-graded. This, in turn means (by \cite[Theorem~2.8]{DadeMathZ}) that the reduced tensor product $\H(N) \overline{\o} L$ of \cite{DadeAnnals} equals the usual tensor product $\H(N) \o_{\H(\Pa)} L$, for any $\H(\Pa)$-module $L$. 
	
	Next, we note that $\End_{\H(\Pa)}(L) = \kk$ for any irreducible $\H(\Pa)$-module by Dixmier's Lemma since $L$ is of countable dimension over $\kk$ and the latter is algebraically closed. Then \cite[Proposition~9.4]{DadeMathZ} says that $\End_{\H(N)}(\Ind^N_{\Pa} L)$ is a crossed product of $\uN(L)$ over $\End_{\H(\Pa)}(L)$. Thus, $\End_{\H(N)}(\Ind^N_{\Pa} L)$ must be a twisted group algebra of $\uN(L)$, giving (i). Then (ii) follows directly from (i) and \cite[Theorem~10.6]{DadeAnnals}; we note here that condition (8.1c) in the definition of  \say{lying over} given in \cite{DadeAnnals} is redundant for strongly graded rings. 
	
	Finally, claim (iii) follows from \cite[Theorem~10.10]{DadeAnnals}. 	
\end{proof}

By Lemma~\ref{lem:admissblesplittingofcategories}, an irreducible $\H_{\omega,\bc}(U_0,\Pa)$-module supported on $Z$ is isomorphic to one of the form $\Phi_Z(M \boxtimes L(\lambda))$, for some irreducible $\dd_{\omega}(Z)$-module and $\lambda \in \Irr \Pa$. Therefore, by Theorem~\ref{thm:CliffordDade}, we may associate to the pair $(M,\lambda)$ the subgroup 
\[
\uN(M,\lambda) := \uN(\Phi_Z(M \boxtimes L(\lambda)))
\] 
of $\uN$ and the group $2$-cocycle $\tau(M,\lambda)$. If $K$ is a subgroup of $\uN$, $\tau$ a $2$-cocycle on $K$ and $\sigma \in \uN$ then 
\[
(\sigma \cdot \tau)(g_1,g_2) := \tau(\sigma g_1 \sigma^{-1}, \sigma g_2 \sigma^{-1}) 
\]
defines a $2$-cocycle on $K^{\sigma} := \sigma^{-1} K \sigma$ and the map defined on group elements by $g \mapsto \sigma^{-1} g \sigma$ and extended linearly to $\kk_{\tau} K$ is a $\kk$-algebra isomorphism $\kk_{\tau} K \iso \kk_{\sigma \cdot \tau} (K^{\sigma})$. 

\begin{remark}\label{rem:Grzeszczuk}
	If $F$ is an irreducible $\H(N)$-module then its restriction $F |_{\H(\Pa)}$ to $\H(\Pa)$ is semi-simple by \cite{GrzeszczukSS}; this is implicit in Theorem~\ref{thm:CliffordDade}(iii). Therefore, condition (1) of lying over is automatic for irreducible modules.  
\end{remark}

\subsection{Classification of irreducibles}

Finally, we come to the classification of irreducible holonomic $\sH_{\omega,\bc}(X,W)$-modules. 

\begin{defn}\label{defn:admissibledatum}
	An \textit{admissible datum} on $\St$ is a tuple $(Z,M,\lambda,\xi)$, where 
	\begin{enumerate}
		\item[(a)] $Z$ is a parallelizable subset of $\St$,
		\item[(b)] $M$ is an irreducible integrable connection (over $\dd_{\omega}(Z)$) with $\Supp M= Z$, 
		\item[(c)] $\lambda \in \Irr \Pa$; and
		\item[(d)] $\xi$ is an irreducible $\kk_{\tau} \uN(M,\lambda)$-module. 
	\end{enumerate} 
We say that a pair of admissible data $(Z,M,\lambda,\xi)$ and $(Z',M',\lambda',\xi')$ are equivalent if there exists $\sigma \in \uN$ such that $Z'' := \sigma(Z') \cap Z$ is open dense in both $\sigma(Z')$ and $Z$ with
\[
{}^{\sigma} \Phi_{Z'}(M' \boxtimes L(\lambda')) |_{Z''} \cong \Phi_{Z}(M \boxtimes L(\lambda)) |_{Z''}, \quad \textrm{as $\sH_{\omega,\bc}(U_0 \cap U_0',\Pa)$-modules,} 
\]
and ${}^{\sigma} \xi' \cong \xi$ as $\kk_{\tau} \uN(M,\lambda)$-modules, where $\kk_{\tau} \uN(M,\lambda)$ acts on the vector space $\xi'$ via the isomorphism  
\[
\kk_{\tau} \uN(M,\lambda) \iso \kk_{\sigma \cdot \tau} (\uN(M,\lambda)^{\sigma}) = \kk_{\tau'} \uN(M',\lambda').
\]
\end{defn}

\begin{thm}\label{thm:mainclassificationnRCA}
	The irreducible holonomic $\sH_{\omega,\bc}(X,W)$-modules are parametrized, up to isomorphism, by the equivalence classes of admissible data as $\St$ runs over all strata. 
\end{thm} 

\begin{proof}
	Let $\ms{N}$ be an irreducible holonomic $\sH_{\omega,\bc}(X,W)$-module. Then, as explained above, there is a unique stratum $\St$ such that $\ms{N}$ is the (unique) irreducible extension to $X$ of the module $\ms{N} |_U$. Moreover, it follows from Theorem~\ref{thm:Serresubquotmelysequiv} that the irreducible module $\ms{N} |_U$ corresponds to an irreducible module $\ms{N}' \in \Hol{\sH_{\omega,\bc}(U_0,N)}_{\St_0}$. Therefore, we must show that the irreducible objects in the latter category are in bijection with equivalence classes of admissible data for $\St$. 
	
	Recall from Lemma~\ref{lem:irredsupp} that $\Supp_{U_0} \ms{N}'$ is $N$-irreducible. In other words, $N$ acts transitively on the irreducible components of $\Supp_{U_0} \ms{N}'$. Replacing $U_0$ by a smaller open set $U_1$ if necessary, we may assume that $\Supp_{U_1} \ms{N}' = Z_1 \sqcup \cdots \sqcup Z_r$ is smooth, with each $Z_i$ connected. As an $\mc{O}_{U_1}$-module, $\ms{N}' = \bigoplus_{i = 1}^r \ms{N}_i$, where $\ms{N}_i$ is supported on $Z_i$. Note that $\ms{N}_i$ is actually a $\sH_{\omega,\bc}(U_1,\Pa)$-submodule because $\Pa$ acts trivially on $\Supp_{U_1} \ms{N}' \subset \St_0$. Also, for each $n \in N$, either $n(Z_i) = Z_i$ or $n(Z_i) \cap Z_i = \emptyset$. By Lemma~\ref{lem:Z0admissibleopen}, there exists a dense parallelizable open subset $Z$ of $Z_1$. Again, shrinking $U_1$ if necessary, we may assume $Z = Z_1$ is closed in $U_1$. Considering $\ms{N}_1$ as a $\sH_{\omega,\bc}(U_1,\Pa)$-module, Lemma~\ref{lem:admissblesplittingofcategories}, together with Lemma~\ref{lem:tensorprodsimpleZ}, says that there exists an irreducible $\dd_{\omega}(Z)$-module $M$ with $\Supp M = Z$ and $\lambda \in \Irr \Pa$ such that $\Phi_Z(M \boxtimes L(\lambda))$ is an irreducible submodule of $\ms{N}_1$ (the latter is semi-simple by Remark~\ref{rem:Grzeszczuk}). Therefore $\ms{N}'$ is an irreducible $\sH_{\omega,\bc}(U_1,N)$-module lying over $\Phi_Z(M \boxtimes L(\lambda))$. By Theorem~\ref{thm:CliffordDade}, $\ms{N}'$ corresponds to an irreducible $\kk_{\tau} \uN(M,\lambda)$-module $\xi$, unique up to isomorphism. This defines the admissible datum $(Z,M,\lambda,\xi)$. 
	
	Assume now we shrank $U_0$ to $U_1'$ so that $\Supp_{U_1'} \ms{N}' = Z_1' \sqcup \cdots \sqcup Z_r'$, again with each $Z_i'$ parallelizable and closed in $U_1'$. Let $Z' = Z_1'$. Then the set $J = \{ \sigma \in \uN \, | \, \sigma(Z) \cap Z' \neq \emptyset \}$ is non-empty. Choose $\sigma_1 \in J$. If $\ms{N}_1$ is the summand of $\ms{N}'$ supported on $Z$ and $\ms{N}_1'$ the summand supported on $Z'$ then $({}^{\sigma_1}\ms{N}_1) |_{\sigma_1(Z) \cap Z'} \cong \ms{N}_1' |_{\sigma_1(Z) \cap Z'}$ as $\sH_{\omega,\bc}(U_1 \cap U_1',\Pa)$-modules because ${}^{\sigma_1}\ms{N}' |_{U_1 \cap U_1'} \cong \ms{N}' |_{U_1 \cap U_1'}$ as $\sH_{\omega,\bc}(U_1 \cap U_1',\Pa)$-modules. Now if we choose $M'$ and $\lambda'$ such that $\Phi_{Z'}(M' \boxtimes L(\lambda'))$ is an irreducible summand of $\ms{N}_1'$, then 
	$$
	({}^{\sigma_1} \Phi_{Z}(M \boxtimes L(\lambda)))|_{\sigma_1(Z) \cap Z'} \quad \textrm{and} \quad \Phi_{Z'}(M' \boxtimes L(\lambda'))|_{\sigma_1(Z) \cap Z'} 
	$$
	are two irreducible summands of $\ms{N}_1' |_{\sigma_1(Z) \cap Z'}$. Theorem~\ref{thm:CliffordDade}(iii) implies that there exists $\sigma_2 \in \uN$ such that $\sigma_2(Z') = Z'$ and  
	\begin{equation}\label{eq:uNconjcalss}
			({}^{\sigma_2 \sigma_1} \Phi_{Z}(M \boxtimes L(\lambda)))|_{\sigma_2 \sigma_1(Z) \cap Z'} \cong \Phi_{Z'}(M' \boxtimes L(\lambda'))|_{\sigma_2 \sigma_1(Z) \cap Z'}. 
	\end{equation}
	If we set $\sigma = \sigma_2 \sigma_1$ then $\ms{N}' |_{U_1 \cap U_1'}$, as a module lying over the module \eqref{eq:uNconjcalss}, corresponds to both $\xi'$ and ${}^{\sigma} \xi$. Theorem~\ref{thm:CliffordDade}(ii) implies that $\xi' \cong {}^{\sigma} \xi$. This shows that there is an injection from the class of irreducible objects in $\Hol{\sH_{\omega,\bc}(U_0,N)}_{\St_0}$ to the equivalence classes of admissible data supported on $\St$. 
		
	If, instead, we begin with admissible data $(Z,M,\lambda,\xi)$ then we may choose an $N$-stable open subset $U_1$ of $X$ such that $Z$ is closed in $U_1$. The representation $\xi$ corresponds to an irreducible holonomic $\sH_{\omega,\bc}(U_1,N)$-module $\ms{N}'$ lying over $\Phi_Z(M \boxtimes L(\lambda))$. By Corollary~\ref{cor:simpleholonomicsocle}, $\ms{N}'$ admits an irreducible extension to $U_0$, which we also denote by $\ms{N}'$. Thus, every admissible $(Z,M,\lambda,\xi)$ comes from some irreducible holonomic $\sH_{\omega,\bc}(U_0,N)$-module. This completes the proof of the theorem. 	
\end{proof}

\section{Applications}

In the  final section we describe a number of applications of the results of the paper. 

\subsection{Extensions between holonomic modules.} 

The goal of this section is to prove that Ext-groups between holonomic modules are finite-dimensional. This relies heavily on our classification of irreducible holonomic modules. For brevity, we write $\Hom_{D^b}(-,-) := \Hom_{D^b_{\mathrm{qc}}(\sH_{\omega,\bc}(X,W))}(-,-)$. 

\begin{thm}\label{thm:dimextRCA}
	For $\ms{M}^{\idot}, \ms{N}^{\idot} \in D^b_{\mr{Hol}}(\sH_{\omega,\bc}(X,W))$, 
	\begin{equation}\label{eq:dimextRCA}
		\dim_{\kk} \Hom_{D^b}(\ms{M}^{\idot}, \ms{N}^{\idot}) < \infty. 
	\end{equation}
\end{thm}

Note that \eqref{eq:dimextRCA} is equivalent to 
\[
 \sum_{i \in \Z} \dim_{\kk} \Hom_{D^b}(\ms{M}^{\idot}, \ms{N}^{\idot}[i]) < \infty,
\]
since we are working in the bounded derived category and $\sH_{\omega,\bc}(X,W)$ has finite global dimension. Since $\Ext^i_{\sH}(\ms{M},\ms{N}) = \Hom_{D^b}(\ms{M},\ms{N}[i])$, an immediate corollary of Theorem~\ref{thm:dimextRCA} is that:

\begin{cor}\label{cor:dimextRCA2}
	For $\ms{M},\ms{N}$ holonomic $\sH_{\omega,\bc}(X,W)$-modules,
	$$
	\dim_{\kk} \Ext^i_{\sH}(\ms{M},\ms{N}) <\infty \quad \textrm{ for all $i \ge 0$.}
	$$ 
\end{cor}

The proof of Theorem~\ref{thm:dimextRCA} will be done in a number of steps. We begin by noting that Theorem~\ref{thm:dimextRCA} and Corollary~\ref{cor:dimextRCA2} are equivalent.  Moreover, since holonomic modules have finite length, it suffices to consider irreducible modules in Corollary~\ref{cor:dimextRCA2}. Finally, we note that Verdier duality $\mathbb{D}$ is a contravariant equivalence $D^b_{\mr{Hol}}(\sH_{\omega,\bc}(X,W)) \to D^b_{\mr{Hol}}(\sH_{\omega,\bc}(X,W)^{op})$ that preserves support. In particular, $\dim_{\kk} \Hom_{D^b}(\ms{M}^{\idot}, \ms{N}^{\idot}) = \dim_{\kk} \Hom_{D^b}(\mathbb{D}(\ms{N}^{\idot}),\mathbb{D}(\ms{M}^{\idot}))$. We will use these facts repeatedly below. 

\begin{lem}\label{lem:RSextfreeShapiro}
	Let $R \subset S$ be a finite extension of $\kk$-algebras, with $S$ projective over $R$ both as a left and a right module, and $\ms{C}_R, \ms{C}_S$ Serre subcategories of $\Lmod{R}$ and $\Lmod{S}$ respectively such that $S \o_R L \in \ms{C}_S$ and $M |_R \in \ms{C}_R$ for all $L \in \ms{C}_R$ and $M \in \ms{C}_S$. If $\dim_{\kk} \Ext^i_R(L,L') < \infty$ for all $L,L' \in \ms{C}_R$ and $i \ge 0$ then $\dim_{\kk} \Ext^i_S(M,M') < \infty$ for all $M,M' \in \ms{C}_S$ and $i \ge 0$. 
\end{lem}

\begin{proof}
	This is by induction on $i$, the case $i = -1$ being vacuous. Let $L = M |_R$ and $I = S \o_R L$. Then there is a short exact sequence $0 \to K \to I \to M \to 0$ in $\Lmod{S}$. Since $\ms{C}_S$ is a Serre subcategory containing $I$ and $M$, $K$ also belongs to $\ms{C}_S$. Applying $\Hom_S( - , M')$ to this gives a long exact sequence 
	$$
	\cdots \to \Ext^{i-1}_S(K,M') \to \Ext^i_S(M,M') \to \Ext^i_S(I,M') \to \Ext^i_S(K,M')  \to \cdots.
	$$
	By induction, $\dim_{\kk} \Ext^{i-1}_S(K,M') < \infty$. Since $S$ is a projective $R$-module, the Eckmann-Shapiro Lemma says $\Ext^i_S(I,M') \cong \Ext^i_R(L,M' |_R)$, which is finite-dimensional by assumption. We deduce that $\dim_{\kk} \Ext^i_S(M,M') < \infty$. 
\end{proof}

\begin{prop}\label{prop:dimextRCAamissible}
	Assume $X$ is affine and let $Z$ be a parallelizable closed subset of $X$. If $\ms{M}^{\idot}, \ms{N}^{\idot} \in D^b_{\mr{Hol}}(\sH_{\omega,\bc}(X,W))$ with support contained in $W Z$ then  
	\begin{equation}
		\dim_{\kk} \Hom_{D^b}(\ms{M}^{\idot}, \ms{N}^{\idot}) < \infty. 
	\end{equation}
\end{prop}

\begin{proof}
	Since holonomic modules have finite length, it suffices to show that if $\ms{M}, \ms{N}$ are irreducible holonomic $\sH_{\omega,\bc}(X,W)$-modules whose supports are contained in $W Z$ then the dimension of $\Ext_{\sH}^i(\ms{M}, \ms{N})$ is finite for all $i \in \Z$. 
	
	Let $Z \subset \St_0$ and $U,U_0$ etc. be as in the previous subsections. Since $Z$ is affine and closed in $X$ (being parallelizable), we may assume that $U$ is also affine. Let $M = \Gamma(X,\ms{M})$ and $N = \Gamma(X,\ms{N})$. Since $\H_{\omega,\bc}(U,W) = \mc{O}(U) \o_{\mc{O}(X)} \H_{\omega,\bc}(X,W)$ and $M = \mc{O}(U) \o_{\mc{O}(X)} M$ for any $M$ supported on $W Z$, the Eckmann-Shapiro Lemma implies that 
	$$
	\Ext^i_{\H_{\omega,\bc}(U,W)}(M,N) \cong \Ext^i_{\H_{\omega,\bc}(X,W)}(M,N). 
	$$
	Let $\varphi \colon V \to U$ be the \'etale morphism of Proposition~\ref{prop:etalemelysiso}. Then, just as in the proof of Theorem~\ref{thm:Serresubquotmelysequiv}, we have $M \cong \varphi^* M$ for any $M$ supported on $W Z$ since the restriction of $\varphi$ to the preimage of $W Z$ is an isomorphism onto $W Z$. Hence,  
	$$
	\Ext^i_{\H_{\omega,\bc}(V,W)}(M,N) = \Ext^i_{\H_{\omega,\bc}(V,W)}(\varphi^* M,\varphi^* N), 
	$$
    which is isomorphic to $\Ext^i_{\H_{\omega,\bc}(U,W)}(M,N)$ by the Eckmann-Shapiro Lemma since $\varphi$ is flat and thus $\H_{\omega,\bc}(V,W)$ is a projective $\H_{\omega,\bc}(U,W)$-module. Applying the isomorphism on the right hand side of \eqref{eq:etalefactoriso} together with the equivalence \eqref{eq:matrixequivZ} means that 
	$$
	\Ext^i_{\H_{\omega,\bc}(U,W)}(M,N) \cong \Ext^i_{\H_{\omega,\bc}(U_0,N)}(M',N')
	$$ 
	for some irreducible $\H_{\omega,\bc}(U_0,N)$-modules $M',N'$ with supports contained in $N Z$. 
	
	Set $R = \H_{\omega,\bc}(U_0,\Pa)$, $S = \H_{\omega,\bc}(U_0,N)$ and $T = \dd_{\omega}(Z) \otimes \H_{\bc}(\h,\Pa)$. Let $\ms{C}_R$ be the Serre subcategory of $\Hol{R}$ consisting of all modules supported on $Z$ and $\ms{C}_S$ the Serre subcategory of $\Hol{S}$ consisting of all modules supported on $N Z$. Lemma~\ref{lem:RSextfreeShapiro} implies that $\Ext^i_{S}(M',N')$ is finite-dimensional if we can show that $\Ext^i_{R}(M_1,M_2)$ is finite-dimensional for all $M_1,M_2 \in \ms{C}_R$. If necessary, we may assume $M_1,M_2$ irreducible.

	As in the proof of Lemma~\ref{lem:admissblesplittingofcategories}, $\widehat{R}_Z := \H_{\omega,\bc}(\widehat{(U_0)}_Z,\Pa) = \widehat{\mc{O}(U_0)}_Z \o_{\mc{O}(U_0)} R$ is a flat $R$-module. If $M_1,M_2 \in \ms{C}_R$ then the $R$-module structure on $M_i$ extends to $\widehat{R}_Z$ by continuity such that $M_i \cong \widehat{\mc{O}(U_0)}_Z \o_{\mc{O}(U_0)} M_i$ as $\widehat{R}_Z$-modules. The Eckmann-Shapiro Lemma says 
	\begin{align*}
		\Ext^i_{R}(M_1,M_2) & \cong \Ext^i_{\widehat{R}_Z}(\widehat{\mc{O}(U_0)}_Z \o_{\mc{O}(U_0)} M_1,M_2) \\
		& = \Ext^i_{\widehat{R}_Z}(M_1,M_2) \\
		& \cong \Ext^i_{\mc{O}(Z)[\![\h]\!] \o_{\mc{O}(Z \times \h)} T}(\Phi_Z^{-1}(M_1),\Phi_Z^{-1}(M_2)) \\
		&  = \Ext^i_{\mc{O}(Z)[\![\h]\!] \o_{\mc{O}(Z \times \h)} T}(\mc{O}(Z)[\![\h]\!] \o_{\mc{O}(Z \times \h)} \Phi_Z^{-1}(M_1),\Phi_Z^{-1}(M_2)) \\
		& \cong \Ext^i_T(\Phi_Z^{-1}(M_1),\Phi_Z^{-1}(M_2)). 
	\end{align*}  
	If $M_i$ is irreducible, then $\Phi_Z^{-1}(M_i) \cong M_i' \boxtimes L(\lambda_i)$ by Lemma~\ref{lem:admissblesplittingofcategories}. We note that 
	$$
	\dim_{\kk} \Ext^{p}_{\dd_{\omega}(Z)}(M_1',M_2') < \infty, \quad \dim_{\kk} \Ext^{i -p}_{\H_{\bc}(\h,\Pa)}(L(\lambda_1),L(\lambda_2)) < \infty,
	$$
	by \cite[CH.3, Theorem~2.7]{BjorkBook} and \cite[Proposition~4.4]{EtingofLie} respectively. Hence, using the (degenerate) K\"unneth spectral sequence, the group $\Ext^i_T(\Phi_Z^{-1}(M_1),\Phi_Z^{-1}(M_2))$ is finite-dimensional.  	
\end{proof}

\begin{proof}[Proof of Theorem~\ref{thm:dimextRCA}] 
	The proof is by double induction on $(k,m)$, where $k = \dim \Supp \ms{M}^{\idot} + \dim \Supp \ms{N}^{\idot}$ and $m = \dim ((\Supp \ms{M}^{\idot}) \cap (\Supp \ms{N}^{\idot}))$. 
	
	If $k = m = 0$, we assume, as before, that $\ms{M}^{\idot} = \ms{M}$ and $\ms{N}^{\idot} = \ms{N}$ are irreducible. Then $\Supp_X \ms{M}$ and $\Supp_X \ms{N}$ are each a single $W$-orbit. If these orbits are disjoint then $\Hom_{D^b}(\ms{M}, \ms{N}) = 0$. If the orbits are equal then they are the $W$-saturation of a parallelizable closed subset $Z$ of $X$ and $\dim_{\kk} \Hom_{D^b}(\ms{M}, \ms{N}) < \infty$ follows from Proposition~\ref{prop:dimextRCAamissible}. 
	
	General case. Using $\dim_{\kk} \Hom_{D^b}(\ms{M}^{\idot}, \ms{N}^{\idot}) = \dim_{\kk} \Hom_{D^b}(\mathbb{D}(\ms{N}^{\idot}),\mathbb{D}(\ms{M}^{\idot}))$, we may assume $\dim \Supp \ms{M}^{\idot} \le \dim \Supp \ms{N}^{\idot}$. If $m < \max \{ \dim \Supp \ms{M}^{\idot}, \dim \Supp \ms{N}^{\idot}\}$ (equivalently, $2m < k$) then we may choose a $W$-stable open subset $U$ of $X$ such that $\Supp_X \ms{M}^{\idot} \subset X \setminus U$ but $U \cap \Supp_X \ms{N}^{\idot} \neq \emptyset$ and $\dim \Supp_X \ms{N}^{\idot} \cap (X \setminus U) < \dim \Supp_X \ms{N}^{\idot}$. Let $j \colon U \hookrightarrow X$ be the open embedding. By Lemma~\ref{lem:localcohomology} and Corollary~\ref{cor:localcohomology}, there is an exact triangle 
	\begin{equation}\label{eq:extdimtriangle}
		\ms{K}^{\idot} \to \ms{N}^{\idot} \to j_+ j^! \ms{N}^{\idot} \stackrel{+1}{\longrightarrow}
	\end{equation}
	with $\ms{K}^{\idot}$ and $j_+ j^! \ms{N}^{\idot}$ in $D^b_{\mr{Hol}}(\sH_{\omega,\bc}(X))$. Applying $\Hom_{D^b}(\ms{M}^{\idot}, -)$ and noting by adjunction that $\Hom_{D^b}(\ms{M}^{\idot}, j_+ j^! \ms{N}^{\idot}) \cong \Hom_{D^b}(j^! \ms{M}^{\idot},j^! \ms{N}^{\idot}) = 0$, we deduce that 
	$$
	\dim_{\kk} \Hom_{D^b}(\ms{M}^{\idot}, \ms{N}^{\idot}) = \dim_{\kk} \Hom_{D^b}(\ms{M}^{\idot}, \ms{K}^{\idot}) < \infty
	$$
	by induction since $\dim \Supp \ms{K}^{\idot} < \dim \Supp \ms{N}^{\idot}$. 
	
	The final case to consider is where 
	\[
	\dim \Supp \ms{M}^{\idot} = \dim \Supp \ms{N}^{\idot} = \dim (\Supp \ms{M}^{\idot}) \cap (\Supp \ms{N}^{\idot}).
	\]
Let $S = (\Supp_X \ms{M}^{\idot}) \cup (\Supp_X \ms{N}^{\idot})$, a $W$-stable closed subset of $X$. If $S_1, \ds, S_r$ are the $W$-irreducible components of maximal dimension, then we can find $W$-stable affine open subsets $U_i$ of $X$ such that $S_i \cap U_i$ is a parallelizable subset (non-empty) in a stratum $Y_i$ and $U_i \cap S_j = \emptyset$ for $i \neq j$. Let $U = \bigcup_i U_i$. Write $j_i \colon U_i \hookrightarrow X$ and $j \colon U \hookrightarrow X$ for the open embeddings. Then, we have once again the triangle \eqref{eq:extdimtriangle}. Since 
	$$
	\dim \Supp \ms{K}^{\idot} \le \dim (\Supp_X \ms{N}^{\idot}) \cap (X \setminus U) < \dim \Supp \ms{N}^{\idot},
	$$
	it suffices by induction to show the finiteness of $\Hom_{D^b}(\ms{M}^{\idot},j_+ j^! \ms{N}^{\idot})$. We have 
	$$
	\Hom_{D^b}(\ms{M}^{\idot},j_+ j^! \ms{N}^{\idot}) = \Hom_{D^b}(j^! \ms{M}^{\idot}, j^! \ms{N}^{\idot}) = \bigoplus_{i = 1}^r \Hom_{D^b}(j_i^! \ms{M}^{\idot}, j^!_i \ms{N}^{\idot}).
	$$
	By assumption, the support of $j_i^! \ms{M}^{\idot}$ and $j^!_i \ms{N}^{\idot}$ are contained in parallelizable closed sets in $U_i$. Therefore, it follows from Proposition~\ref{prop:dimextRCAamissible} that $\dim_{\kk} \Hom_{D^b}(j_i^! \ms{M}^{\idot}, j^!_i \ms{N}^{\idot}) < \infty$. 
\end{proof}

\subsection{Shift functors}

We define shift functors $\mc{S}_{\bc \to \bc'}$ on the category of holonomic modules. This imitates a construction of Berest-Chalykh \cite{BerestChalykhQuasi} of shift functors between category $\mc{O}$ for different parameters $\bc$. Let $\mc{A}$ denote the set of reflection hypersurfaces in $X$. Fix some $W$-stable subset $\Omega \subset \mc{A}$ and let $\bc, \bc'$ be parameters such that $\bc(w,Z) = \bc'(w,Z)$ for all $(w,Z)$ with $Z \notin \Omega$. Let $X^{\circ} = X \setminus \bigcup_{Z\in \Omega} Z$ and $j \colon X^{\circ} \to X$ be the (open) inclusion. Then 
\[
\sH_{\omega,\bc}(X,W) |_{X^{\circ}} = \sH_{\omega,\bc'}(X,W) |_{X^{\circ}}.
\]
Thus, if $\ms{M}$ is a $\sH_{\omega,\bc}(X,W)$-module then $\ms{M} |_{X^{\circ}}$ can be thought of as a $\sH_{\omega,\bc'}(X,W) |_{X^{\circ}}$-module. We write $(\ms{M} |_{X^{\circ}})_{\bc'}$ for this change of parameters. For $\ms{M} \in \mc{M}(\sH_{\omega,\bc}(X,W))$, define
\[
\mc{S}_{\bc \to \bc'}(\ms{M}) = j_0((j^0 \ms{M})_{\bc'}) \quad (\cong j_*(\ms{M} |_{X^{\circ}}) \textrm{ as an $\mc{O}_X$-module}).
\]  

\begin{prop}\label{prop:shiftfunctor}
	The functor $\mc{S}_{\bc \to \bc'} \colon \mc{M}(\sH_{\omega,\bc}(X,W)) \to \mc{M}(\sH_{\omega,\bc'}(X,W))$ is exact and restricts to a functor $\Hol{\sH_{\omega,\bc}(X,W)} \to \Hol{\sH_{\omega,\bc'}(X,W)}$.  
\end{prop}

\begin{proof}
Since the open subset $X^{\circ}$ is locally principal in $X$, the morphism $j$ is affine and hence $j_*$ is exact on quasi-coherent $\mc{O}$-modules. Thus, the functor $\mc{S}_{\bc \to \bc'}$ is exact. The fact that it preserves holonomic modules follows from Theorem~\ref{thm:jpushforwardholo}.
\end{proof}  

\begin{remark}
	In the case $X = \h$ is a reflection representation and $\Omega = \mc{A}$, Berest-Chalykh \cite{BerestChalykhQuasi} construct a shift functor $\mc{T}_{\bc \to \bc'} \colon \mc{O}_{\bc} \to \mc{O}_{\bc'}$ between the categories $\mc{O}$. For $M \in \mc{O}_{\bc}$, $\mc{T}_{\bc \to \bc'}(M)$ is the submodule of $\mc{S}_{\bc \to \bc'}(M)$ consisting of all sections locally nilpotent for the action of $\kk[\h^*]_+$. Theorem~\ref{thm:preshol} gives an alternative proof of \cite[Proposition~7.1]{BerestChalykhQuasi}, that $\mc{T}_{\bc \to \bc'}(M)$ is finitely generated (and hence $\mc{T}_{\bc \to \bc'}$ is well-defined). In fact, we see that the larger module $\mc{S}_{\bc \to \bc'}(M)$ is already finitely generated because $M$ is holonomic.    
\end{remark}

\subsection{An example}

We end this section with an explicit example to illustrate one of our main results, \Cref{thm:jpushforwardholo}, that holonomicity is preserved by push-forward along open embeddings. 

Let $X = \h$ be a linear representation of $W$ and $\h_{\reg}$ the locus in $\h$ where $W$ acts freely (assumed non-empty). We take $\omega = 0$ so that $\sH_{\omega,\bc}(\h_{\reg},W) = \dd(\h_{\reg}) \rtimes W$. If $j \colon \h_{\reg} \hookrightarrow \h$ is the embedding and $\ms{M}$ a holonomic $W$-equivariant $\dd$-module on $\h_{\reg}$, then \Cref{thm:jpushforwardholo} says that the $\H_{\bc}(\h,W)$-module 
\[
\Gamma(\h,j_0 \ms{M}) = \Gamma(\h_{\reg},\ms{M}) 
\]
is holonomic. In particular, for any $W$-equivariant integrable connection $\ms{M}$ on $\h_{\reg}$, the space of sections $\Gamma(\h_{\reg},\ms{M})$ is a holonomic $\H_{\bc}(\h,W)$-module. 

As a concrete example, assume that $(W,\h)$ is a complex reflection group, with associated discriminant $\delta \in \C[\h]^W$. Then $\Gamma(\h_{\reg},\mc{O}) = \C[\h][\delta^{-1}]$ is a holonomic $\H_{\bc}(\h,W)$-module. In particular, it has finite length. 

\begin{question}
	What are the composition factors of $\C[\h][\delta^{-1}]$ as a $\H_{\bc}(\h,W)$-module? 
\end{question} 

\section{Appendix - Rational Cherednik algebras}

In the appendix we collect a number of results on aspherical values for rational Cherednik algebras that are required in the paper. 

Assume throughout that $\h$ is a linear representation of $W$ (the pair $(\h,W)$ need not be a complex reflection group). Let $\mc{A}$ denote the set of reflection hyperplanes in $\h$ and for each $H \in \mc{A}$, we denote by $\ell_H$ the order of the pointwise stabilizer $W_H$ of $H$ in $W$. Fix $\alpha_H \in \h^*$ such that $\Ker \alpha_H= H$ and let $\alpha_H^{\vee} \in \h$ be a non-zero vector spanning the non-trivial eigenspace for all $s \in W_H \setminus \{ 1 \}$. Fixing a generator $s_H \in W_H$, we define idempotents
\[
e_{H,i} = \frac{1}{\ell_H} \sum_{j = 0}^{\ell_H-1} \mathrm{det}_{\h}(s_H)^{ij} s_H^j
\]
in $\kk W_H$. Let $\kappa_{H,i}$ for $H \in \mc{A}$ and $0 \le i \le \ell_H-1$ be formal variables with identification $\kappa_{H,i} = \kappa_{H',i}$ if there exists $w \in W$ such that $w(H) = H'$. Then, as explained in \cite[Remark~3.1]{GGOR}, we can identify $R = \kk[\kappa_{H,i} \, | \, H \in \mc{A}, \, 0 \le i \le \ell_H-1]$ with the ring of functions on the parameter space $\mc{S} := \mc{S}(\h)^W$ such that the generic rational Cherednik algebra $\H(\h,W)$ associated to $(\h,W)$ is the $R$-algebra with defining relation
\[
[y,x] = \langle x,y \rangle + \sum_{H \in \mc{A}} \frac{\langle \alpha_H,y \rangle \langle x, \alpha_H^{\vee} \rangle}{\langle \alpha_H, \alpha_H^{\vee} \rangle} \sum_{i = 0}^{\ell_H-1} (\kappa_{H,i+1} - \kappa_{H,i}) e_{H,i}, \quad \forall y \in \h, x \in \h^*. 
\]
Here we extend $\kappa_{H,i}$ to all $i \in \Z$ by requiring $\kappa_{H,i + \ell_H} := \kappa_{H,i}$. 

Specializing the $\kappa_{H,i}$ to a scalar gives the usual rational Cherednik algebra. Recall that for each $\lambda \in \Irr W$, we have a Verma module $\Delta(\lambda) = \H(\h,W) \o_{R[\h^*] \rtimes W} (\lambda \o_{\kk} R)$ whose specialization to a fixed $\kappa \in \mc{S}$ is $\Delta_{\kappa}(\lambda)$. The latter has a unique irreducible quotient $L_{\kappa}(\lambda)$.

\begin{defn}\label{defn:bcregularRCA}
	 We say that a parameter $\kappa$ is \textit{regular} if $L_{\kappa}(\lambda) = \Delta_{\kappa}(\lambda)$ for all $\lambda \in \Irr W$ and the set of all regular parameters is denoted $\mc{S}_{\reg}$. 
\end{defn}

We note that regular parameters are spherical. The Euler element 
\begin{equation}\label{eq:eulerelement}
\eu = \sum_{i} x_i y_i + \sum_{H \in \mc{A}} \ell_H \sum_{i = 0}^{\ell_H-1} \kappa_{H,i} e_{H,i} 
\end{equation}
acts by a scalar $\kappa(\lambda)$ on $1 \o \lambda \subset \Delta_{\kappa}(\lambda)$.  

\begin{lem}\label{lem:asphericalhyperplanes}
Assume $\h$ is a faithful $W$-representation. The set of aspherical values in $\mc{S}$ is contained in a finite union of affine hyperplanes of the form $\kappa(\lambda) - \kappa(\mu) + m = 0$ for $\lambda, \mu \in \Irr W$ and $m \in \Z_{\ge 1}$. 
\end{lem}

By \cite[Lemma~2.5(ii)]{DunklOpdam}, the scalar $\kappa(\lambda)$ is a linear expression in the $\kappa_{H,i}$ with integer coefficients. Therefore, extending scalars if necessary, we may assume $\kk = \C$ in the proof of Lemma~\ref{lem:asphericalhyperplanes}. We fix a conjugate linear $W$-equivariant isomorphism $\h^* \to \h$ denoted $x \mapsto \overline{x}$. This extends to a conjugate linear anti-isomorphism $\overline{(-)} \colon \H(\h,W) \iso \H(\h,W)$, with $\overline{w} = w^{-1}$ for $w \in W$. If $( - , - )$ is a sesquilinear non-degenerate $W$-invariant form on $\lambda$ then this extends uniquely to a form on $\Delta(\lambda)$ with values in $R$ such that for all $u,v \in \Delta(\lambda)$,  
\begin{align*}
	(w(u),w(v)) &= (u,v) \quad  \quad w \in W, \\
	(x \cdot u, v) & =(u,\overline{x} \cdot v) \quad  x \in \h^*, \\ 
	(\alpha u, \beta v) & = \alpha \overline{\beta} (u,v) \quad \alpha,\beta \in R,
\end{align*}
where $\overline{(-)} \colon R \to R$ is complex conjugation on $\C$ and fixes the $\kappa_{H,i}$. For any complex value of the parameters, the form specializes to a sesquilinear $W$-invariant form 
\[
( - , -)_{\kappa} \colon \Delta_{\kappa}(\lambda) \times \Delta_{\overline{\kappa}}(\lambda) \to \C, 
\]
whose left radical is the maximal proper submodule of $\Delta_{\kappa}(\lambda)$. 


\begin{proof}[Proof of Lemma~\ref{lem:asphericalhyperplanes}]

By \cite[Theorem~4.1(i)]{BE}, a parameter $\kappa$ is aspherical if and only if there exists $\lambda \in \Irr W$ such that $e L_{\kappa}(\lambda) = 0$.

	We first recall a standard consequence of the Euler grading. Suppose that
	$L_\kappa(\lambda)\neq \Delta_\kappa(\lambda)$. Then the maximal proper
	submodule of $\Delta_\kappa(\lambda)$ contains a non-zero singular vector. More
	precisely, there exist $\mu\in\Irr W$ and $m\in \Z_{\geq 1}$ such that the
	$\mu$-isotypic component in homogeneous degree $m$ contains a vector killed by
	$\h$. On such a vector the Euler element acts in two ways: since it lies in
	degree $m$ inside $\Delta_\kappa(\lambda)$, it acts by $\kappa(\lambda)+m$, whereas, since it is a lowest weight vector of type $\mu$, it acts by
	$\kappa(\mu)$. Hence
	\[
	\kappa(\lambda)-\kappa(\mu)+m=0.
	\]
	It follows that the non-regular locus is contained in the countable union of
	affine hyperplanes of this form.
	
	It remains to show that the aspherical locus is contained in only finitely many of these hyperplanes. Let $e$ denote the trivial idempotent of $\C W$. Since $e\Delta(\lambda)$ is a finitely generated module over the polynomial algebra $\C[\h]^W$, we may choose $N_\lambda>0$ (independent of $\kappa$) such that, putting
	\[
	V_\lambda
	=
	\bigoplus_{0\leq m<N_\lambda}
	\Delta(\lambda)_{m,\mathrm{triv}},
	\]
	we have $\C[\h]^W\cdot V_\lambda=e\Delta(\lambda)$. Here $\Delta(\lambda)_{m,\mathrm{triv}}$ denotes the trivial isotypic component of the homogeneous degree $m$ part of $\Delta(\lambda)$. Crucially, since $\h$ is a faithful $W$-module, $V_{\lambda} \neq 0$. 
    
    As in \cite[Proposition~2.17(iii)]{DunklOpdam}, the homogeneous isotypic
	components $\Delta(\lambda)_{m,\mu}$ are mutually orthogonal with respect to $( - , - )$ unless both the degree and the $W$-type agree. Let
	\[
	D_\lambda
	=
	\det\left((-, -)\big|_{V_\lambda}\right).
	\]
	This is a polynomial in the parameters $\kappa_{H,i}$. It is non-zero: indeed,
	for generic $\kappa$ the Verma modules $\Delta_\kappa(\lambda), \Delta_{\overline{\kappa}}(\lambda)$ are irreducible,
	and hence the specialized contravariant form is non-degenerate.
	
	We claim that every aspherical value $\kappa$ for which
	$eL_\kappa(\lambda)=0$ satisfies $D_\lambda(\kappa)=0$. Indeed, let $M_\kappa(\lambda)$ be the maximal proper submodule of
	$\Delta_\kappa(\lambda)$. Then
	\[
	L_\kappa(\lambda)=\Delta_\kappa(\lambda)/M_\kappa(\lambda).
	\]
	If $eL_\kappa(\lambda)=0$, then the image of $e\Delta_\kappa(\lambda)$ in
	$L_\kappa(\lambda)$ is zero. Since $e\Delta(\lambda)$ is generated over
	$\C[\h]^W$ by $V_\lambda$, this implies that the image of
	$(V_\lambda)_\kappa$ is contained in $M_\kappa(\lambda)$. The left radical of the contravariant form is precisely $M_\kappa(\lambda)$, so the restriction of the
	form to $(V_\lambda)_\kappa$ is degenerate. Therefore, $D_\lambda(\kappa)=0$. 
	
	Since $V_\lambda$ is supported only in degrees $m<N_\lambda$, $D_\lambda(\kappa)=0$ only if there exist singular vectors for this $\kappa$ in degrees $< N_{\lambda}$. This implies that the set of zeros $V(D_{\lambda})$ of $D_{\lambda}$ is contained in the union of the hyperplanes $\kappa(\lambda)-\kappa(\mu)+m = 0$ for $\mu\in\Irr W$ and $1\leq m<N_\lambda$. Since  $V(D_{\lambda})$  is a hypersurface, the irreducible factors of $D_\lambda$ must be among the linear forms $\kappa(\lambda)-\kappa(\mu)+m$. Thus, the set of aspherical parameters is contained in the union of the finitely many hyperplanes
	\[
	\kappa(\lambda)-\kappa(\mu)+m = 0, \quad \lambda, \mu\in\Irr W, \quad 1\leq m<N_\lambda.
	\]
\end{proof}

\begin{remark}\label{rem:preciseapsherical}
	The precise set of aspherical values is known for many irreducible complex reflection groups, most importantly for the infinite family $G(m,p,n)$ by \cite{DunklGriffeth}. In all known cases, this locus equals a union of the affine hyperplanes of the form $\kappa(\lambda) - \kappa(\mu) + m = 0$.  
\end{remark}

\begin{lem}\label{lem:simpleiffcompletesimple}
	Let $(\h,W)$ be a faithful linear representation. Then 
    \begin{enumerate}
        \item[(i)]  The maps 
	\[
	I \mapsto \widehat{\kk}[\h]_0 \o_{\kk[\h]} I \quad \textrm{and} \quad J \mapsto J \cap \H_{\kappa}(\h,W),
	\] 
	for $I \lhd \H_{\kappa}(\h,W)$ and $J \lhd \H_{\kappa}(\widehat{\h},W)$, define (inverse) bijections between the two-sided ideals of $\H_{\kappa}(\h,W)$ and $\H_{\kappa}(\widehat{\h},W)$. 
    \item[(ii)] In particular, $\H_{\kappa}(\widehat{\h},W)$ is simple if and only if $\H_{\kappa}(\h,W)$ is simple.  
    \item[(iii)] If $W^{\circ}$ is the normal subgroup of $W$ generated by all reflections $s$ with $\bc(s) \neq 0$ then $\H_{\bc}({\h},W)$ is simple if and only if $\H_{\bc}(\h,W^{\circ})$ is simple. 
    \end{enumerate}
\end{lem}

\begin{proof}
	Parts (i) and (ii). If $\H_{\kappa}(\h,W)$ contains a proper ideal $I$ then the fact that the Euler element $\eu$ in $\H_{\kappa}(\h,W)$ acts ad-locally finitely on $\H_{\kappa}(\h,W)/I$ implies that $0$ is contained in the support of $\H_{\kappa}(\h,W)/I$ since its support will be a closed $\Cs$-stable subset of $\h/W$. Therefore, applying $\widehat{\kk[\h]}_0 \o_{\kk[\h]} - $ to the short exact sequence 
	\[
	0 \to I \to \H_{\kappa}(\h,W) \to \H_{\kappa}(\h,W)/I \to 0
	\]
	shows that $\widehat{\kk}[\h]_0 \o_{\kk[\h]} I$ is a proper ideal of $\H_{\kappa}(\widehat{\h},W)$.   
	
	Conversely, we show that if $J \lhd \H_{\kappa}(\widehat{\h},W)$ is a two-sided ideal then $I := J \cap \H_{\kappa}(\h,W)$ is the set of $\ad(\eu)$-locally finite vectors in $J$ and $J = \widehat{\kk}[\h]_0 \o_{\kk[\h]} I$. If $\H_{\kappa}(\widehat{\h},W)$ were complete with respect to the $\kk[\h]_+$-adic topology and the eigenvalues of $\eu$ on $\H_{\kappa}(\h,W)$ were bounded below, this would have been standard; we explain the necessary modifications. 
	
	First, notice that each piece $\mc{F}_i$ of the order filtration on $\H_{\kappa}(\h,W)$ is stable under $\ad(\eu)$ with the eigenvalues of $\eu$ on $\mc{F}_i$ bounded below (by $-i$). Moreover, $\mc{F}_i \H_{\kappa}(\widehat{\h},W) = \widehat{\kk}[\h]_0 \o_{\kk[\h]} \mc{F}_i$ is complete with respect to the $\kk[\h]_+$-adic topology because $\mc{F}_i$ is finite over $\kk[\h]$. Therefore, the same holds for each $\mc{F}_i J$. Thus, arguing as in the proof of \cite[Theorem~2.3]{BE}, we deduce that $J_i := (\mc{F}_i J) \cap \H_{\kappa}(\h,W)$ is the subspace of $\ad(\eu)$-locally finite vectors in $\mc{F}_i J$ and $\mc{F}_i J = \widehat{\kk}[\h]_0 \o_{\kk[\h]} J_i$. Since $J = \bigcup_{i \ge 0} \mc{F}_i J$, the claims follow.  

    Part (iii). In this case, $\H_{\bc}({\h},W)$ is a finite free $\H_{\bc}(\h,W^{\circ})$-module (left or right). Moreover, for any $w \in W$, conjugation $h \mapsto w h w^{-1}$ defines an automorphism of $\H_{\bc}(\h,W^{\circ})$. We use the fact that, since $\h$ is a faithful $W$-representation, both of these algebras are prime rings of Gelfand-Kirillov dimension $2 \dim \h$. The PBW theorem implies that it suffices to check this fact is true for $\C[\h \times \h^*] \rtimes W$ and $\C[\h \times \h^*] \rtimes W^{\circ}$ respectively, which can be deduce from the results in \cite[Section~10.5]{MR}. Therefore, \cite[Korollar~3.5]{BKGelfandKirillov} says that if $R$ is a proper quotient of either ring then the GK-dimension of $R$ is strictly less than $2 \dim \h$. If $R$ is such a quotient of $\H_{\bc}({\h},W)$ then the image of $\H_{\bc}({\h},W^{\circ}) \to \H_{\bc}({\h},W) \to R$ will have Gelfand-Kirillov dimension less than $2 \dim \h$. This shows that $\H_{\bc}({\h},W)$ not being simple implies that $\H_{\bc}(\h,W^{\circ})$ is not  simple. 
    
    For the converse, let $I \lhd \H_{\bc}(\h,W^{\circ})$ be non-zero. Let $M = \H_{\bc}({\h},W) \otimes_{\H_{\bc}(\h,W^{\circ})} \H_{\bc}(\h,W^{\circ}) / I$, a $\H_{\bc}({\h},W)$-$\H_{\bc}(\h,W^{\circ})$-bimodule. Then there is a non-zero morphism 
    \[
    \H_{\bc}({\h},W) \to E := \End_{\H_{\bc}(\h,W^{\circ})}( \H_{\bc}({\h},W) \otimes_{\H_{\bc}(\h,W^{\circ})} \H_{\bc}(\h,W^{\circ}) / I), 
    \]
    given by multiplication on the left. Since $\H_{\bc}({\h},W)$ is a finite free right $\H_{\bc}(\h,W^{\circ})$-module, $E$ has the same Gelfand-Kirillov dimension as $\H_{\bc}(\h,W^{\circ}) / I$. Hence, the (non-zero) image of $\H_{\bc}({\h},W)$ in $E$ has Gelfand-Kirillov dimension strictly less than $2 \dim \h$. Thus, if $\H_{\bc}(\h,W^{\circ})$ is not simple then neither is $\H_{\bc}(\h,W)$. 
\end{proof}

\begin{lem}\label{lem:asphericalcomplete}
	The parameter $\kappa$ is aspherical for $\H_{\kappa}(\h,W)$ if and only if it is so for $\H_{\kappa}(\widehat{\h},W)$. 
\end{lem}

\begin{proof}
	If $\kappa$ is aspherical for $\H_{\kappa}(\widehat{\h},W)$ and $M$ a (non-zero) representation of this algebra with $e M= 0$ then $M$ is a representation of $\H_{\kappa}(\h,W)$ meaning that $\kappa$ is aspherical for $\H_{\kappa}(\h,W)$. 
	
	Conversely, assume $\kappa$ is aspherical for $\H_{\kappa}(\h,W)$. Then, \cite[Theorem~4.1(i)]{BE} implies that there exists an irreducible $L_{\kappa}(\lambda)$ in category $\mc{O}$ such that $e L_{\kappa}(\lambda) = 0$. Then $M = \widehat{\kk}[\h]_0 \o_{\kk[\h]} L_{\kappa}(\lambda) = \widehat{\kk}[\h]^W_0 \o_{\kk[\h]^W} L_{\kappa}(\lambda)$ is a non-zero $\H_{\kappa}(\widehat{\h},W)$-module with $eM = 0$. 
\end{proof} 



\newcommand{\etalchar}[1]{$^{#1}$}
\def\cprime{$'$} \def\cprime{$'$} \def\cprime{$'$} \def\cprime{$'$}
\def\cprime{$'$} \def\cprime{$'$} \def\cprime{$'$} \def\cprime{$'$}
\def\cprime{$'$} \def\cprime{$'$} \def\cprime{$'$} \def\cprime{$'$}
\def\cprime{$'$} \def\cprime{$'$}

\end{document}